\documentclass[11pt,reqno]{amsart}

\usepackage{bm}   % new PF 13 May
\usepackage{stmaryrd} 

\usepackage{amsmath,amsfonts,amsthm,color,amssymb,mathrsfs}
\usepackage{placeins}% for \FloatBarrier
\usepackage{float}  % new PF 30 Jun

\usepackage{longtable}
\usepackage{array}
\newcolumntype{L}[1]{>{\raggedright\arraybackslash}p{#1}}
\usepackage{xcolor}

\usepackage[T1]{fontenc}
\usepackage{lmodern}
\usepackage[normalem]{ulem}
\usepackage[left=1in, right=1in, top=1.1in,bottom=1.1in]{geometry}
\usepackage{enumitem}
\usepackage{hyperref}
\hypersetup{
colorlinks   = true,
citecolor    = blue,
linkcolor=blue
}
\usepackage{cancel}
\usepackage{comment}
\usepackage{changepage}
\allowdisplaybreaks
\usepackage{csquotes}
\usepackage{graphicx}
\usepackage{appendix}

\usepackage{tikz} \usetikzlibrary{arrows.meta,positioning}

\usepackage{todonotes}

\makeatletter
\newcommand{\cond}[2]{%
    \phantomsection
    \def\@currentlabel{(#1)}%
    \label{#2}%
    \textup{\bf(#1)}
} %for the specific condition of assumptions

 \makeatother
\usepackage{lipsum} 

\usepackage{xargs}         
\newtheorem{lemma}{Lemma}[section]
\newtheorem{theorem}[lemma]{Theorem}
\newtheorem{definition}[lemma]{Definition}
\newtheorem{corollary}[lemma]{Corollary}
\newtheorem{remark}[lemma]{Remark}
\newtheorem{proposition}[lemma]{Proposition}

\newtheorem{assumption}{Assumption}
\newtheorem{assumptionalt}{Assumption}[assumption]
\newenvironment{assumptionp}[1]{
\renewcommand\theassumptionalt{#1}
\assumptionalt
}{\endassumptionalt}

\DeclareMathOperator*{\esssup}{ess\,sup}

\newcommand{\B}{{\mathbb{B}}}

\newcommand{\FP}{{\mathfrak{P}}}

\newcommand{\FC}{{\mathfrak{C}}}

\newcommand{\BX}{{\mathbf{X}}}
\newcommand{\BB}{{\mathbf{B}}}

\newcommand{\MF}{{\mathcal{F}}}
\newcommand{\MG}{{\mathcal{G}}}

\newcommand{\MP}{{\mathcal{P}}}

\numberwithin{equation}{section}

\def\PP{{\mathbb P}}
\def\rme{{\rme}}

\def\R{\mathbb R}
\def\E{\mathbb E}

\newcommand{\normmm}[1]{{\left\vert\kern-0.25ex\left\vert\kern-0.25ex\left\vert #1 
   \right\vert\kern-0.25ex\right\vert\kern-0.25ex\right\vert}}
\newcommand{\shortthickbar}{%
  \vcenter{\hbox{\rule{1.8pt}{2.2ex}}}%
}

\newcommand{\thicknm}[1]{%
  \mathopen{\shortthickbar}%
  \mkern4mu #1 \mkern4mu%
  \mathclose{\shortthickbar}%
}

\newcommand{\triplenorm}[1]{%
\left|\!\left|\!\left|#1\right|\!\right|\!\right|%
}

\begin{document}

\title[Rough BSDEs]{Rough backward stochastic differential equations}

\author{Peter K. Friz}
\address{Institut für Mathematik, 
Technische Universität Berlin and Weierstraß Institut, Berlin, Germany}
\email{friz@math.tu-berlin.de}

\author[J. Song]{Jian Song}\address{Research Center for Mathematics and Interdisciplinary Sciences; Frontiers Science Center for Nonlinear Expectations, Ministry of Education, Shandong University, Qingdao 266237, China}
\email{txjsong@sdu.edu.cn}

\author[H. Zhang]{Huilin Zhang}\address{Research Center for Mathematics and Interdisciplinary Sciences; Frontiers Science Center for Nonlinear Expectations (Ministry of Education), Shandong University, Qingdao 266237, China}
\email{huilinzhang@sdu.edu.cn}

\author[K. Zhang]{Kuan Zhang}\address{Institute for Theoretical Sciences, Westlake Institute for Advanced Study, Westlake University, Hangzhou, Zhejiang, 310030, China}
\email{zhk01815@163.com}


\begin{abstract}

We develop an intrinsic well-posedness theory for nonlinear backward stochastic differential equations (BSDEs) driven simultaneously by Brownian motion and a (level-$2$) rough path of finite $p$-variation. 
Unlike earlier approaches based on smooth approximation or transformation methods, we formulate the equation directly by viewing its solution as a rough semimartingale (in the sense of \cite{friz2023rough}). 
This framework is particularly suited to BSDEs, whose Brownian martingale component is only implicitly defined and lacks the \emph{a priori} time regularity required by stochastic-sewing-based controlled rough path methods \cite{fhl21,allan2024rough}. 

We establish comparison, existence, uniqueness, and stability under the natural regularity condition $H\in C_b^\gamma$, $\gamma>p$. The main analytical difficulty, loss of integrability arising from nonlinear composition, is overcome through new conditional $p$-variation norms with BMO-type properties. Finally, by randomizing the rough driver as a Brownian rough path, we establish a direct correspondence between rough BSDEs and backward doubly stochastic differential equations (BDSDEs). 
\end{abstract}

\maketitle
\tableofcontents

\section{Introduction}

Let $\mathbf X$ be a $p$-rough path with $2<p<3$. On a filtered probability
space carrying a Brownian motion $W$, we consider the following rough backward
stochastic differential equation (rough BSDE)
\begin{equation}\label{e:BSDE}
Y_{t}
=
\xi
+
\int_{t}^{T} f(r,Y_{r},Z_{r})\,dr
+
\int_{t}^{T} H(Y_{r})\,d\mathbf X_r
-
\int_{t}^{T} Z_r\,dW_r,
\qquad t\in[0,T].
\end{equation}
This equation was originally introduced by Diehl and Friz \cite{DF}. Their construction
starts with smooth paths $\mathrm X^n$ whose canonical rough path lifts
$\mathbf X^n$ converge to $\mathbf X$. The corresponding classical BSDEs are
solved via a Doss--Sussmann (type flow) transformation, and the solution to \eqref{e:BSDE} is then defined as the limit of the resulting pairs $(Y^n,Z^n)$.
Formally, the rough integral is interpreted as
\begin{equation}\label{e:approximate rough integral}
\int_t^T H(Y_r)\,d\mathbf X_r
:=
\lim_{n\to\infty}
\int_t^T H(Y^n_r)\,d\mathbf X^n_r,
\end{equation}
where the limit is independent of the approximating sequence and gives a robust
solution map in the rough path topology. It does not, however, provide an
intrinsic definition of the integral in~\eqref{e:BSDE}. In particular,
$H(Y)$ is not treated as an integrand within a rough stochastic integration
theory, and the dynamics of $Y$ are characterized only through approximation.

A direct theory for~\eqref{e:BSDE} is available in the Young regime. Diehl and Zhang
\cite{diehl2017young} consider drivers of finite $q$-variation with
$q<2$ and establish well-posedness via a fixed-point argument. In that setting,
the integral involving $Y$ is a pathwise Young integral and the equation is
satisfied directly. Song et al. \cite{BSDEYoung-I, BSDEYoung-II} further develop the nonlinear Young regime, providing a probabilistic interpretation of stochastic partial differential equations (SPDEs) driven by space-time noise in the Young regime. A general jump case is studied by Becherer and Sun \cite{BS25}. For genuinely rough drivers, Liang and Tang
\cite{liang2025multidimensional} extend the approximation and flow approach
of \cite{DF} to multidimensional BSDEs, under either a smallness condition or
a diagonal structural assumption on the nonlinear rough coefficient. They
also develop an intrinsic theory for multidimensional equations but with a
linear rough drift. Their construction
uses stochastic controlled rough paths, a $p$-variation rough stochastic
integral, and the stochastic sewing lemma. Stochastic controlled rough paths
and stochastic sewing also play a central role in recent developments for
rough (forward) stochastic differential equations; see, for example,
\cite{fhl21,allan2024rough}. A reflected rough BSDE theory is introduced by Li et al. \cite{li2024reflected}, who also study the associated  rough PDEs with obstacles, as well as optimal stopping and American option pricing problems. 

The aim of the present paper is to solve \eqref{e:BSDE} intrinsically
with a
nonlinear rough coefficient $H(Y)$ appearing in the rough integral,
without resorting to  Doss--Sussmann transformation or defining the solution through smooth approximations. We
show that the pair $\bigl(Y,-H(Y)\bigr)$
belongs to an appropriate space of rough semimartingales (RSMs), originally introduced in \cite{friz2023rough}. Within this
space, the integral against $\mathbf X$ in \eqref{e:BSDE} is defined intrinsically as a rough stochastic
integral via compensated Riemann sums:
\begin{equation}\label{e:intro rough integral}
\int_t^T H(Y_r)\,d\mathbf X_r
:=
\lim_{|\pi|\to0}
\sum_{[u,v]\in\pi}
\left(
H(Y_u)\,\delta\mathrm X_{u,v}
-
DH(Y_u)H(Y_u)\,\mathbb X_{u,v}
\right),
\end{equation}
where the limit is taken in probability. Thus, the solution pair $(Y,Z)$
satisfies \eqref{e:BSDE} as a hybrid rough stochastic integral equation in
the precise sense of Definition~\ref{def:def of BSDEs}.

A notable feature of our work is that we are able to obtain all required estimates in the RSM setting. In particular, neither the
construction of \eqref{e:intro rough integral} nor the well-posedness proof
uses the stochastic sewing lemma. This differs from the stochastic
controlled rough path approach of Liang and Tang
\cite{liang2025multidimensional} and from the rough SDE theory developed in
\cite{fhl21,allan2024rough}. The present work therefore provides an original alternative to stochastic sewing for natural classes of equations that combine rough signals with (implicit) martingale components.

The distinction is relevant for BSDEs because the martingale integrand $Z$
is part of the solution and is known only implicitly as a general predictable process. Stochastic controlled
rough path methods require suitable time regularity of the martingale
component in an $L^k$-scale. Such regularity is not available from the
definition of the classical BSDE and must normally be established as part of the
analysis. A rough semimartingale, by contrast, is specified through a
Doob--Meyer-type decomposition and does not impose time regularity on its
martingale part. Moreover, its definition is independent of a prescribed
$L^k$-integrability exponent. This allows the notion of solution to be
formulated before any integrability estimates are proved.

The intrinsic formulation also allows the comparison principle and stability estimates
to be established directly from the equation. Strict comparison, for example, is
not generally preserved under passage to the limit: approximation readily
gives non-strict ordering, but strictness may be lost. Similarly, direct
difference estimates are difficult when the rough integral is known only
through the approximation~\eqref{e:approximate rough integral}. Working in the RSM space, we
prove a strict comparison theorem, an \emph{a priori} estimate, existence
and uniqueness, and local stability with respect to the terminal condition,
the drift coefficient, and the driving rough path.

Our method also reduces the regularity required for the rough coefficient, and we only need 
\begin{equation}\label{e:con}
H\in C_b^\gamma, \quad 
 \text{for some } \gamma>p.
\end{equation}
This is the standard regularity threshold for the corresponding RDE. In contrast, the flow transformation approach in \cite{DF} requires the stronger condition
$H\in C_b^{\gamma+2}$. The additional two derivatives are needed to control
the transformed stochastic equation and are not intrinsic to the rough
dynamics. Since no transformation is used in the present approach, the strict comparison
theorem, the \emph{a priori} estimate, and the well-posedness result all hold
under the natural condition~\eqref{e:con}.

We now describe the main estimates. Suppose first that $Y$ is essentially
bounded. To control the RSM norm of $(Y,-H(Y))$, one needs estimates for the
remainder in the rough expansion of $H(Y)$. In contrast to the deterministic
RDE case, this remainder contains a martingale component, which can be controlled via a recent result on $p$-variation estimates for martingale transforms \cite{friz2023rough}. Its conditional version Lemma~\ref{lem:p-var martingale estim} gives the
required control of the mixed conditional moments of the $p$-variation
norm; see \eqref{e:K4} in Proposition~\ref{prop:G(Y) in D}. This is the main
input in the composition estimates for rough semimartingales.

An essential bound for $Y$ is obtained through comparison. The difference
of two solutions satisfies a linear rough BSDE. We derive an explicit
exponential representation of this linear equation and combine it with a
John--Nirenberg inequality for $\mathrm{VMO}^{p\text{-var}}$ processes. This yields the comparison theorem and, by comparison with suitable deterministic
RDEs, the required uniform bound. The same representation also yields strict
comparison under the corresponding strict ordering assumptions.

The remaining difficulty is the loss of integrability caused by nonlinear
composition. As observed in the stochastic controlled rough path setting
\cite{fhl21}, estimates for $H(Y)-H(\widetilde Y)$ naturally require a
higher integrability exponent than the one initially available for
$Y-\widetilde Y$. We address this by proving an energy estimate for suitable
increasing processes. It implies that the conditional norms $\left\|
\|\bm\cdot\|_{p\text{-}\mathrm{var}}
\right\|_{\infty;k,\infty}$, 
defined by conditional moments of the $p$-variation norm, are equivalent
for different values of $k$. These norms have the same self-improving
feature as classical $k$-\textrm{BMO} norms, and we therefore refer to them
as \textrm{BMO}-type norms. Their equivalence compensates for the
integrability loss in the nonlinear composition estimate. This leads to a
local contraction and, by iteration, to existence and uniqueness of the solution on the
full time interval. The same estimates yield continuous dependence of the
solution on the data.

Finally, we relate rough BSDEs to backward doubly stochastic differential
equations (BDSDEs), famously introduced by Pardoux and Peng \cite{PardouxPeng1994BDSDE}, as a counterpart in the BSDE theory for SPDEs.
To this end, we first introduce a backward causality property for processes with rough parameters to describe their measurability with respect to the backward filtration. Then we show that the solution to rough BSDEs has such a measurable version with respect to the $\sigma$-algebra generated by the forward Brownian motion and the backward rough path flow, via the measurable selection initiated by Dellacherie and Meyer \cite{Meyer1978} and recently developed by Denkert et al. \cite{DKP24}. Randomizing $\mathbf X$ by a Brownian rough path then produces
a BDSDE, while the intrinsic rough integral identifies the corresponding
backward stochastic integral. This, for the first time, provides a direct connection between the
rough BSDE and BDSDE formulations without passing through smooth
approximations of the additional Brownian signal. (Applications to semilinear rough and stochastic PDEs, in the spirit of \cite{PardouxPeng1994BDSDE,LionsSouganidis2000semilinear,DF,liang2025multidimensional}, are also possible but will be pursued elsewhere.)

In conclusion, our contributions are threefold. Firstly, we introduce the intrinsic notion of solution to rough BSDEs in the framework of rough semimartingales, as well as the corresponding rough stochastic analysis (without reliance on --- and complementary to --- the work \cite{fhl21}).
Secondly, we establish the well-posedness and overcome the problem of degenerating integrability by establishing energy estimates for BMO-type processes. Thirdly, we connect rough BSDEs and BDSDEs via a novel direct randomization argument.

{\em Organization of Paper.}
In Section~\ref{sec:notations}, we introduce notations and assumptions. Section~\ref{sec:rough sto analysis} recalls the RSM framework and the construction of rough stochastic integrals. 
In Section~\ref{subsec:energy estimates and RSM estimates}, we prove energy estimates for increasing processes and the equivalence of the \textrm{BMO}-type $p$-variation norms. Composition and integration estimates for RSMs are established in Section~\ref{subsec:estimates for rough semimartingales}, and the RSM seminorm used in the stability analysis is introduced in Section~\ref{subsec:seminorms}. 
Section~\ref{sec:Rough BSDE} contains the well-posedness theory. We first prove the comparison theorem in Section~\ref{subsec:comparison}, derive the \emph{a priori} estimate in Section~\ref{subsec:a priori estimate}, and then establish existence, uniqueness, and stability under the condition~\eqref{e:con} in Section~\ref{subsec:Rough BSDE with low regularity H}. Multidimensional equations are studied in Section~\ref{subsec:Multi dimensional}.
Section~\ref{sec:BDSDEs} studies the connection with BDSDEs, including backward causality in Section~\ref{subsec:causality} and Brownian randomization in Section~\ref{subsec:randomization of BDSDEs}. 
Appendix~\ref{sec:Appendix A} collects L\'epingle-type inequalities, conditional $p$-variation estimates for martingale transforms, and estimates for RDEs with unbounded drift. The symbolic index is provided in Appendix~\ref{app:symbolic_index}.

{\em Mathematical Roadmap.} First and importantly,   Lemma~\ref{lem:BMO of q-var} implies that the $q$-variation of a continuous process is a \textrm{BMO}-process. Consequently, the conditional norms $\| \|\bm\cdot\|_{q\text{-}\mathrm{var}} \|_{\infty;k,\infty}$ are comparable for all $k \ge 1$. Together with the conditional martingale-transform estimate of Lemma~\ref{lem:p-var martingale estim}, this underlies two parallel lines of the argument.
The first controls a single solution: Proposition~\ref{prop:G(Y) in D} bounds the remainder of $H(Y)$, Corollary~\ref{cor:int of the solution} converts this into a bound for $\int H(Y)\,d\mathbf X$, and Propositions~\ref{prop:bounded by BMO} and~\ref{prop:BMO and ess sup} then show that, once $\|Y\|_{\infty}$ is controlled, both $\|\delta M^{Y}\|_{\mathrm{BMO}}$ and the RSM seminorm are bounded.
The missing bound for $\|Y\|_{\infty}$ comes from a different direction: the comparison theorem (Theorem~\ref{thm:comparison}), proved via the exponential representation of the linearized equation and the John--Nirenberg inequality, compares $Y$ with two deterministic RDEs and yields the \emph{a priori} estimate of Theorem~\ref{thm:BMO bound}. 
The second line of argument controls differences: Proposition~\ref{prop:R^G - R^G}, in which the equivalence of norms compensates for the integrability loss in the nonlinear composition, feeds into Lemma~\ref{lem:Gamma and I} and then into the contraction property of Proposition~\ref{prop:contraction}. Theorem~\ref{thm:the existence} combines the contraction on small intervals with the \emph{a priori} estimate to obtain global well-posedness, and the same difference estimates give the stability result, Theorem~\ref{thm:stability}.  Since the comparison theorem is the only genuinely scalar step, Theorem~\ref{thm:multi-d conditional} records that the whole argument goes through in the multidimensional case once the sup-norm bound is available.
Theorem~\ref{thm:BDSDE} shows that every randomized rough BSDE solution is the unique $L^2$ solution of a BDSDE, and its proof relies on the preceding theory only through well-posedness, stability, and the RDE bound underlying the \emph{a priori} estimate.

%\begin{funding}
{\bf Acknowledgment}
PKF and HZ acknowledge support from DFG CRC/TRR 388 ``Rough
Analysis, Stochastic Dynamics and Related Fields'', Projects A07, B04 and B05. Part of this work was carried out during a visit of the first author to Shandong University. 
%KL acknowledges supports from EPSRC
%[grant number EP/Y016955/1] and from the Humboldt fellowship while at TU Berlin where this project was commenced. 
JS is partially supported by National Natural Science Foundation of China (Grant Numbers 12521001, 12471142); JS and HZ are supported by the Fundamental Research Funds for the Central Universities.  HZ is partially supported by NSF of Shandong (ZR2023MA026).
%\end{funding}

\section{Preliminaries}

\subsection{Notations and Assumptions}\label{sec:notations}

We collect notations that  will be used in this article. Let $(E,\|\bm\cdot\|_{E})$ be a Banach space, and let $V$ and $K$ be Euclidean spaces. We denote by $|\bm\cdot|$ the Euclidean norm  for both vectors and matrices. Some of the definitions introduced below depend on the underlying time interval or terminal time. For simplicity of notation, we omit this dependence when the interval is $[0,T]$ (with terminal time $T$) and no confusion arises.

\begin{itemize}[leftmargin=1.2em, labelsep=0.2em]

\item {\it Probability space:} Fix a finite time horizon $T>0$. For an integer $d\ge 1$, let $\left\{W_t\right\}_{t\in[0,T]}$ be a $d$-dimensional standard Brownian motion on the probability space $(\Omega,\mathcal F,\mathbb P)$, and let $(\mathcal F_{t})_{t\in[0,T]}$ be the augmented filtration generated by $W$. Denote the conditional expectation by $\E_{t}[\bm\cdot] := \E[\bm\cdot |\mathcal F_{t}]$. 

\

\item {\it Partition:} For $0\le s\le t\le T$, set $\Delta_{[s,t]}:=\{(u, v): s \leq u \leq v \leq t\}$ and $\Delta := \Delta_{[0,T]}$. Let $\mathcal P_{[s,t]}$ denote the set of all partitions $\pi=\{[t_i,t_{i+1}];s=t_0< t_1<\cdots< t_n=t\}$ of $[s,t]$, and define the mesh size of $\pi$ by $|\pi| :=\max_{0\le i<n}(t_{i+1}-t_i)$.

\

\item {\it Control function:} For a function $w:\Delta \to \R^{+}$, we call $w$ a control function if $w$ is continuous and satisfies that $w(s,u) + w(u,t) \le w(s,t)$ for $s\le u \le t$. In particular, $w(s,s) = 0$ for all $s$.

\ 

\item {\it Linear map:} We denote by $\mathscr L(V,K)$ the space of linear maps from $V$ to $K$, endowed with the operator norm $\|F\|_{\mathscr L(V,K)}=\sup _{x \in V,|x| \leq 1}|F(x)|$. We write $V\otimes K$ for the (algebraic) tensor product of $V$ and $K$,
equipped with the projective norm.
In finite dimensions, the spaces $\mathscr L(V,K)$ and $K \otimes V$
are canonically isomorphic, and the above norms are equivalent.

\ 

\item {\it Space of functions:} 
For any closed set $U\subset V$ and a measurable function $f:U\to E$, define the uniform norm by
\begin{equation*}
\|f\|_{\infty;U} := \esssup_{r\in U}\|f(r)\|_{E}.
\end{equation*}
If $U$ is clear from the context, we simply write $\|f\|_{\infty}$ for $ \|f\|_{\infty;U}$.
Denote by $C(U,E)$ the space of continuous functions $f:U\to E$. Denote by $C_{b}(U,E)$ the space of functions $f\in C(U,E)$ with finite uniform norm. 

For $(s,t) \in \Delta$ and $k \ge 1$, denote by $L^{k}([s,t],E)$ the space of measurable functions $f:[s,t] \to E$ such that  
\begin{equation*}
\|f\|_{L^{k}([s,t])}:=\Big\{\int_{s}^{t} |f_r|^{k} dr\Big\}^{1/k} < \infty. 
\end{equation*}
For $\lambda\ge 1$, denote by $C^{\lambda\text{-var}}(\Delta_{[s,t]},E)$ the space of continuous functions $f\in C(\Delta_{[s,t]},E)$ such that  
\begin{equation*}
\|f\|_{\lambda\text{-var};[s,t]} := \sup_{\pi\in \mathcal P_{[s,t]}}\Big\{\sum_{[u,v]\in\pi}\| f_{u,v}\|^{\lambda}_{E}\Big\}^{1/\lambda} < \infty.
\end{equation*}
Also denote the $\lambda$-variation norm of $f$ as a two-parameter function defined by   
\begin{equation*}
(\|f\|_{\lambda\text{-}\mathrm{var}})_{s,t} := \|f\|_{\lambda\text{-}\mathrm{var};[s,t]}.
\end{equation*}

For any $\alpha \in (0,1]$ and $f\in C(\Delta_{[s,t]},E)$, we write $f\in C^{\alpha}(\Delta_{[s,t]},E)$ if 
\begin{equation*}
\|f\|_{\alpha\text{-H\"ol};[s,t]}  := \sup\limits_{(u,v)\in  \Delta_{[s,t]};u\neq v} \frac{\|f_{u,v}\|_E}{|u-v|^{\alpha}}  < \infty.
\end{equation*}

For $g\in C([s,t],E)$, define the increment operator $\delta g:\Delta_{[s,t]}\to E$ by 
\begin{equation*}
\delta g_{u,v} := g_{v} - g_{u},\quad \text{where } (u,v)\in \Delta_{[s,t]}.
\end{equation*}
For a one-parameter function $g\in C([0,T],E)$, we write $\|\delta g\|_{\lambda\text{-}\mathrm{var}}$ and similarly for the H\"older norm. 

Assume $\mathbf X = (\mathrm X,\mathbb X)$ with $\mathrm X\in C([0,T],V)$ and  $\mathbb X\in C(\Delta,V^{\otimes 2})$. 
For any $(s,t)\in \Delta$, set 
\begin{equation}\label{eq:XX-norm}
|\delta \mathbf X |_{\lambda\text{-}\mathrm{var};[s,t]} := \|\delta \mathrm X\|_{\lambda\text{-}\mathrm{var};[s,t]} + \|\mathbb X\|_{\lambda/2\text{-}\mathrm{var};[s,t]} \, 
\end{equation}
and 
\[
\triplenorm{\delta \mathbf X}_{\lambda\text{-}\mathrm{var};[s,t]}
:=
\|\delta \mathrm X\|_{\lambda\text{-}\mathrm{var};[s,t]} + \|\mathbb X\|^{1/2}_{\lambda/2\text{-}\mathrm{var};[s,t]}.
\]
We set $|\delta \mathbf X |_{\alpha;[s,t]} := \|\delta \mathrm X\|_{\alpha\text{-H\"ol};[s,t]} + \|\mathbb X\|_{2\alpha\text{-H\"ol};[s,t]}$ and $\triplenorm{\delta \mathbf X}_{\alpha;[s,t]}
:=
\|\delta \mathrm X\|_{\alpha\text{-H\"ol};[s,t]} + \|\mathbb X\|^{1/2}_{2\alpha\text{-H\"ol};[s,t]}$. 

A path $\mathrm{X} \in C([0,T],V)$ is called a Lipschitz path if it is Lipschitz continuous in time. It is termed a smooth path if it possesses continuous derivatives of all orders on the closed interval $[0,T]$, in which case we write $\mathrm{X} \in C^{\infty}([0,T],V)$. For a Lipschitz path $\mathrm{X}$, we define its (canonical) lift as the pair $\mathbf{X} = (\mathrm{X},\mathbb{X})$, where the second-order component $\mathbb{X}_{s,t}$ is given by $\mathbb X_{s,t} := \int_{s}^{t} \delta \mathrm X_{s,r} \otimes \partial_r \mathrm X_{r} dr$, $(s,t) \in \Delta$. 

\

\item {\it Differentiability:} For $\gamma\ge 0$, let $\lfloor \gamma \rfloor$ denote the largest integer less than or equal to $\gamma$. Denote by $C^{\gamma}_b(V, K)$ the space of $\lfloor\gamma\rfloor$ times continuously differentiable functions $F: V \to K$ with bounded derivatives up to order $\lfloor\gamma\rfloor$,
and such that $D^{\lfloor\gamma\rfloor}F$ is H\"older continuous with
exponent $\gamma-\lfloor\gamma\rfloor$. The norm is defined by
\begin{equation*}
\|F\|_{C^{\gamma}_b}:=\sum_{i=0}^{\lfloor\gamma\rfloor} \| D^i F \|_{\infty} + \|D^{\lfloor \gamma \rfloor} F\|_{(\gamma - \lfloor \gamma \rfloor)\text{-H\"ol}}.
\end{equation*}

\ 

\item {\it $L^{k}$-norm:} For $k \ge 1$ ($k = \infty$ resp.), denote by  
$L^{k}_{t}(V) := L^{k}(\mathcal F_t,V)$ the space of $V$-valued
$\mathcal F_t$-measurable random variables $\xi$ such that
\begin{equation*}
\|\xi\|_{k} := \big\{\E\big[|\xi|^k\big]\big\}^{\frac1k} < \infty \quad (\|\xi\|_{\infty} := \esssup_{\omega\in\Omega} |\xi(\omega)| < \infty\ \text{resp.})
\end{equation*}
Moreover, we denote by $L^0$ the space of measurable random variables.

For $(s,t) \in \Delta$ and $k\ge 1$ ($k=\infty$ resp.), define $\mathcal L^{k}([s,t]\times\Omega,V)$ ($\mathcal L^{\infty}([s,t]\times\Omega,V)$ resp.) as the space of progressively measurable processes $S$ satisfying 
\begin{equation*}
\|S\|_{\mathcal L^{k}([s,t] \times \Omega)} := \left\| \|S_{\bm\cdot}\|_k \right\|_{L^{k}([s,t])}
< \infty \quad \Big(\|S\|_{\mathcal L^{\infty}([s,t] \times \Omega)}:=  \sup_{r\in[s,t]}\|S_r\|_{\infty} < \infty  \quad \text{resp.}\Big), 
\end{equation*}
where we may also use  $\|\bm\cdot\|_{\infty} := \|\bm\cdot\|_{\mathcal L^{\infty }([0,T]\times \Omega)}$ for processes as well if there is no confusion. 

Denote by $\mathcal H^{k,\infty}([s,t])$ the space of progressively measurable processes $S$ such that  
\begin{equation*}
\|S\|_{\mathcal H^{k,\infty};[s,t]} := \left\{ \left\| \E_{\bm\cdot}\left[\|S\|^k_{L^{k}([\bm\cdot,t])}\right] \right\|_{\mathcal L^{\infty}([s,t]\times\Omega)} \right\}^{\frac{1}{k}} < \infty. 
\end{equation*}

\

\item {\it Mixed integrability:} Set $1\le m\le n\le \infty$ and $m<\infty$. Let $\mathcal G$ be a sub-$\sigma$-field of $\mathcal F$. Define the conditional $L^m$-norm as follows: 
\begin{equation*}
\left\|\left. \xi \right|\mathcal G\right\|_{m} := \left\{\E\left[\left.|\xi|^{m}\right|\mathcal G\right]\right\}^{\frac{1}{m}}\quad \end{equation*}
Define the following $(m,n)$-norm by taking the $L^n$-norm of the conditional $L^m$-norm:
\begin{equation}\label{e:mixed moment}
\left\|\left.\xi \right|\mathcal G\right\|_{m,n} := \left\| \left\|\left. \xi \right|\mathcal G\right\|_{m} \right\|_{n}.
\end{equation}
We have (\cite[Proposition~2.3]{fhl21}), for $n\ge m$,
\[ \|\xi\|_m =\| \|\xi |\mathcal G\|_m\|_m \le \|\|\xi |\mathcal G\|_m\|_n\le \|\|\xi |\mathcal G\|_n\|_n=\|\xi \|_n, \quad \text{for all } \xi \in L^n.\]
For a two-parameter process  $F:\Delta\times \Omega\to V$ define
\begin{equation*}
\|F\|_{\infty;m,n} := \sup_{\substack{(s,t)\in\Delta}} \left\|\left. F_{s,t} \right|\mathcal F_{s}\right\|_{m,n},\ \text{ and }\ \|F\|_{\infty;n} :=\sup_{\substack{(s,t)\in\Delta}} \left\| F_{s,t} \right\|_{n}.
\end{equation*}
\ \ 

\item {\it \emph{BMO}-martingale:}  
For $(s,t)\in\Delta$, let $\mathcal M^{\text{loc}}_{[s,t]}(V)$ be the space of continuous local martingales on $[s,t]$, with $M_{s} = 0$. For $M_{\bm\cdot} - M_{s}\in \mathcal M^{\text{loc}}_{[s,t]}(V)$, we call $M\in \text{BMO}_{[s,t]}(V)$ if 
\begin{equation*}
\|\delta M\|_{\mathrm{BMO};[s,t]}:=\esssup_{(u,\omega)\in[s,t]\times \Omega}\left\{\E_{u}\left[\delta \langle M \rangle_{u,t} \right]\right\}^{\frac{1}{2}}<\infty.
\end{equation*}

\ 

\item {\it \emph{BMO} and \emph{VMO}-process:} Let $S$ be a $V$-valued adapted right continuous process with left limits (\textrm{RCLL}). Then, we call $S$ a \textrm{BMO}-process on $[s,t]$ if
\begin{equation*}
\varrho_{s,t}(S) := \sup_{s \le \tau \le \tau' \le t} \big \|S_{\tau'} - S_{\tau-} \big| \mathcal F_{\tau}\big\|_{1,\infty} < \infty,
\end{equation*}
where the supremum is taken over all stopping times $\tau,\tau'$, and  $S_{\tau-}$ is the left limit of $S_{\tau}$. 
When $S$ is continuous in time, the stopping times in the supremum can be replaced by deterministic times, which can be referred to \cite[Proposition~2.2]{le2022quantitative}. 

As a corollary of \eqref{e:BMO vs p-var} in Lemma~\ref{lem:Lepingle}, for any continuous local martingale $M$, the BMO norm $\|\delta M\|_{\mathrm{BMO};[s,t]}$ is equivalent to the modulus of mean oscillation $\varrho_{s,t}(M)$. Hence, a continuous local martingale is a \textrm{BMO}-process on $[s,t]$ if and only if it belongs to $\text{BMO}_{[s,t]}$. 

For a \textrm{BMO}-process $S$ on $[0,T]$, we call $S$ a \textrm{VMO} process if 
\begin{equation*}
\lim_{h\downarrow 0} \sup_{(s,t)\in\Delta, t-s \le h} \varrho_{s,t}(S) = 0.
\end{equation*}
For $\lambda \ge 1$ and a \textrm{VMO} process $S$, we call $S$ a $\mathrm{VMO}^{\lambda\text{-}\mathrm{var}}$ process if $\|\varrho(S)\|_{\lambda\text{-}\mathrm{var};[0,T]} < \infty$.
\end{itemize}

\begin{lemma}\label{lem:Banach property of L C}
Assume $1\le m \le n \le \infty$, $m < \infty$, and $q \ge 1$. For $(s,t) \in \Delta$, denote by $L^{\infty;m,n}C^{q\text{-}\mathrm{var}}_2([s,t])$ the set of all two-parameter continuous processes $P$ with $\| \|P\|_{q\text{-}\mathrm{var}} \|_{\infty;m,n;[s,t]} < \infty$, and denote by $L^{\infty;m,n}C^{q\text{-}\mathrm{var}}([s,t])$ the set of all one-parameter continuous processes $S$ with $\|S\|_{L^{\infty;m,n}C^{q\text{-}\mathrm{var}};[s,t]} := \|S_t\|_n + \| \|\delta S\|_{q\text{-}\mathrm{var}} \|_{\infty;m,n;[s,t]} < \infty$.

Then, $(L^{\infty;m,n}C^{q\text{-}\mathrm{var}}_2([s,t]), \| \|\bm\cdot\|_{q\text{-}\mathrm{var}} \|_{\infty;m,n;[s,t]})$ and $(L^{\infty;m,n}C^{q\text{-}\mathrm{var}}([s,t]), \|\bm\cdot\|_{L^{\infty;m,n}C^{q\text{-}\mathrm{var}};[s,t]})$ are two Banach spaces. 
\end{lemma}

\begin{proof}
Let $\{P^n\}_{n \ge 1}$ be a Cauchy sequence of $L^{\infty;m,n}C^{q\text{-}\mathrm{var}}_2([s,t])$. Then, there exists a subsequence $\{n_j\}_{j \ge 1}$ such that 
\begin{equation*}
\lim_{j' \to \infty}\sum_{j \ge j'} \E \left[ \|P^{n_j} - P^{n_{j+1}}\|_{q\text{-}\mathrm{var};[s,t]} \right] = 0.
\end{equation*}
Hence, by L\'evy's theorem, we have 
\begin{equation*}
\E \Big[\lim_{j' \to \infty}\sum_{j \ge j'}  \|P^{n_j} - P^{n_{j+1}}\|_{q\text{-}\mathrm{var};[s,t]} \Big] = 0.
\end{equation*}
Therefore, a.s., $\{P^{n_j}\}_{j \ge 1}$ forms a Cauchy sequence under the norm $\|\bm\cdot\|_{q\text{-}\mathrm{var};[s,t]}$. 

Denote by $P$ the limit of $\{P^{n_j}\}_{j \ge 1}$. By Fatou's lemma for the conditional expectation and the lower semi-continuity property of $\|\bm\cdot\|_{q\text{-}\mathrm{var}}$, we see that  
\begin{equation*}
\lim_{j\to \infty } \big\| \|P^{n_j} - P\|_{q\text{-}\mathrm{var}} \big\|_{\infty;m,n;[s,t]} = 0.
\end{equation*}
Consequently, the Cauchy property forces the sequence $\{P^n\}_{n\ge 1}$ to have a unique limit $P$, which implies the desired Banach space property of $(L^{\infty;m,n}C^{q\text{-}\mathrm{var}}_2([s,t]), \| \|\bm\cdot\|_{q\text{-}\mathrm{var}} \|_{\infty;m,n;[s,t]})$. Similarly, one can prove that $(L^{\infty;m,n}C^{q\text{-}\mathrm{var}}([s,t]), \|\bm\cdot\|_{L^{\infty;m,n}C^{q\text{-}\mathrm{var}};[s,t]})$ is also a Banach space. 
\end{proof}

Below are assumptions that will be used in the paper.  
\begin{assumptionp}{($\mathbf{A}_{\mathbf{par}}$)}
\noindent  \\[0.5em]
\cond{$p,p'$}{(p,p')}:
$p \in(2,3)$, $p \le p'$, and $\frac{2}{p} + \frac{1}{p'} > 1$. \\[0.5em] 
\noindent \cond{$p,q,q'$}{(p,q,q')}:
$p \in(2,3)$, $p \le q \le q'$, and $\frac{1}{p} + \frac{1}{q} + \frac{1}{q'} > 1$. In particular, this implies $q < 4$.
\\[0.5em] 
\noindent \cond{$\gamma,p,q,q'$}{(gamma)}:
Assumption~\ref{(p,q,q')} holds, and $\gamma > p$ satisfies 
\begin{equation*}
\gamma - 1 \ge (q+q')/q'.
\end{equation*}
\end{assumptionp}

Throughout this paper, for $q \le q'$ we denote by
\begin{equation*}
q''/2 := \left(1/q + 1/q'\right)^{-1} = qq'/(q + q').
\end{equation*}

\begin{assumptionp}{($\mathbf{A}_{\mathbf{f}}$)}
\noindent  \\[0.5em]
\cond{$f$}{(f)}: 
Let $C_f > 0$ be a constant. Suppose that $f: \Omega \times [0,T] \times \R \times \R^{d}\to \R$ is a progressively measurable function such that the following bounds hold 
\begin{equation*}
|f(\omega,t,0,0)| \le C_f \quad\text{and}\quad |f(\omega,t,y,z) - f(\omega,t,y',z')| \le C_f (|y-y'| + |z-z'|)
\end{equation*}
for all $(\omega,t,y,y',z,z') \in \Omega \times [0,T] \times \R \times \R \times \R^d \times \R^d$. \\[0.5em]
\noindent 
\cond{$f_{\mathrm{det}}$}{(fdet)}: 
Assumption~\ref{(f)} holds. Assume further that $f$ is a deterministic function. 
\end{assumptionp} 

Under Assumption \ref{(p,q,q')}, for simplicity of notation, we denote the parameter set by 
\begin{equation}\label{e:Theta}
\Theta := (T,d,q,q',p).
\end{equation}
In addition, we write $\lesssim$ to represent that an inequality holds up to the right-hand side (RHS) multiplied by a constant that depends only on $\Theta$; for a parameter $K$, the notation $\lesssim_{K}$ represents that an inequality holds up to the RHS multiplied by a constant that depends only on $(\Theta,K)$.

\subsection{Elements of Rough Stochastic Analysis}\label{sec:rough sto analysis}
Throughout this section, $p\in(2,3)$ is a fixed number, $p',q,q'$ are  numbers greater than two such that $p'\ge p$ and $q'\ge q$, and $V_1$, $V_2$, and $V$ are Euclidean spaces.

\begin{definition}[$p$-rough path]\label{def:rough path}
For any $\mathrm X\in C^{p\text{-}\mathrm{var}}([0,T],V)$ and  $\mathbb X\in C^{\frac{p}{2}\text{-}\mathrm{var}}(\Delta,V^{\otimes 2})$, we call $\mathbf X = (\mathrm X,\mathbb X)$ a \emph{$p$-rough path} (a rough path with finite $p$-variation), denoted by $\mathbf X\in \mathscr{C}^{p\text{-}\mathrm{var}}([0,T], V)$, if the following \emph{Chen's relation} holds:
\begin{equation}\label{e:chen's relation}
\mathbb X_{s,t} - \mathbb X_{s,u} - \mathbb X_{u,t} = \delta\mathrm X_{s,u}\otimes \delta\mathrm X_{u,t},\quad 0\le s\le u\le t\le T.
\end{equation}
We call $\mathbf X$ a \emph{geometric $p$-rough path}, denoted by $\mathbf X\in \mathscr{C}^{0,p\text{-}\mathrm{var}}_{g}([0,T],V)$,  if $\mathbf X \in \mathscr{C}^{p\text{-}\mathrm{var}}([0,T],V)$ and there exists a sequence of smooth paths $\{\mathrm X^n\}_{n\ge 1}$ with their canonical lifts $\{\mathbf X^n:=(\mathrm X^n , \mathbb X^n)\}_{n\ge 1} $ such that $|\delta( \mathbf X^n - \mathbf X) |_{p\text{-}\mathrm{var}}\to 0$ as $n\to \infty$.
\end{definition}

\begin{definition}[$1/p$-H\"older rough path]\label{def:Holde rough path}
For $\mathbf X \in \mathscr C^{p\text{-}\mathrm{var}}([0,T],V)$, we call $\mathbf X$ a $1/p$-\emph{H\"older} rough path, denoted by $\mathbf X \in \mathscr C^{1/p}([0,T],V)$, if $|\delta \mathbf X|_{1/p} < \infty$. We call $\mathbf X$ a geometric $1/p$-\emph{H\"older} rough path, denoted by $\mathbf X \in \mathscr C^{0,1/p}_g([0,T],V)$, if $\mathbf X \in \mathscr C^{1/p}([0,T],V)$ and there exists a sequence of smooth paths $\{\mathrm X^n\}_{n\ge 1}$ with their canonical lifts $\{\mathbf X^n:=(\mathrm X^n , \mathbb X^n)\}_{n\ge 1} $ such that $|\delta( \mathbf X^n - \mathbf X) |_{1/p}\to 0$ as $n\to \infty$. 
\end{definition}

\begin{remark}\label{rem:chen's relation}
Define the bracket process
\begin{equation*}
[\mathbf X]_{t} := \delta \mathrm X_{0,t}\otimes \delta \mathrm X_{0,t}- \left(\mathbb X_{0,t} + \mathbb X^{\top}_{0,t} \right).
\end{equation*}
By \eqref{e:chen's relation}, we have $\delta[\mathbf X]_{s,t} =  \delta \mathrm X_{s,t}^{\otimes2}-2 \mathrm{Sym}(\mathbb X_{s,t})$, and hence, $[\mathbf X]\in C^{\frac{p}{2}\text{-}\mathrm{var}}([0,T],V^{\otimes 2})$. 
\end{remark}

Let $I\subset[0,T]$ be an interval and let $\mathbf X=(\mathrm X,\mathbb X)\in \mathscr{C}^{p\text{-}\mathrm{var}}(I,V)$. We introduce the following three notions as integrands for rough stochastic integrals.

\begin{definition}[Controlled rough path]\label{def:control rough path} 
Assume $\psi:I \to V_1$ and $\psi':I \to \mathscr L(V,V_1)$ are two deterministic continuous functions. We call $(\psi,\psi')$ an $\mathrm X$-controlled rough path with finite $(p,p')$-variation, denoted by $(\psi,\psi')\in \mathscr{D}^{(p,p')\text{-}\mathrm{var}}_{\mathrm X}(I,V_1)$, if 
\begin{equation*}
\|(\psi,\psi')\|_{\mathscr{D}^{(p,p')\text{-}\mathrm{var}}_{\mathrm X};I} := \|\delta \psi\|_{p\text{-}\mathrm{var};I} + \|\delta \psi'\|_{p'\text{-}\mathrm{var};I} + \|R^{\psi}\|_{\frac{pp'}{p+p'}\text{-}\mathrm{var};I} <\infty ,
\end{equation*}
where the remainder $R^{\psi}_{s,t} := \delta \psi_{s,t} - \psi'_{s}\delta\mathrm X_{s,t}$ for $s,t\in \Delta_I.$
\end{definition}

\begin{definition}[Rough semimartingale]\label{def:rough semimartingale}
Let
\[
Y : I \times \Omega \to V_1,
\qquad
Y' : I \times \Omega \to \mathscr{L}(V,V_1)
\]
be continuous adapted processes. We call $(Y,Y')$ an
$\mathrm{X}$-controlled rough semimartingale ({\em RSM}) with $(p,p')$-variation,
and write
\[
(Y,Y')
\in
\mathrm{RSM}^{(p,p')\text{-}\mathrm{var}}_{\mathrm{X}}(I,V_1),
\]
if there exist a continuous local martingale
\[
M^Y \in \mathcal{M}_I^{\mathrm{loc}}(V_1)
\]
and a two-parameter process $P^Y$ such that
\[
\delta Y_{s,t}
=
Y'_s\,\delta \mathrm{X}_{s,t} 
+
P^Y_{s,t}
+
\delta M^Y_{s,t} 
\equiv
Y'_s\,\delta \mathrm{X}_{s,t} 
+
R^Y_{s,t}
\qquad
(s,t)\in\Delta_I,
\]
and $(Y-M^Y,Y') 
\in \mathscr{D}^{(p,p')\text{-}\mathrm{var}}_{\mathrm X}(I,V_1)$, a.s.
% , i.e. is an $\mathrm{X}$-controlled process with a.s. $(p,p')$-variation,
% the definition of which includes a.s. $(1/p+1/p')$-variation of $P^Y$. 
We call $P^Y$ (resp. $R^Y$) the pathwise (resp. stochastic) remainder 
of the \emph{RSM} $(Y,Y')$. 
\end{definition}

\begin{remark}\label{rem:given (Y,Y') in RSm}
Given $(Y,Y')\in\mathrm{RSM}^{(p,p')\text{-}\mathrm{var}}_{\mathrm X}([0,T],V)$, whenever $\frac{1}{p} + \frac{1}{p'} > \frac{1}{2}$, the decomposition $Y = (Y - M^Y) + M^{Y}$ in Definition~\ref{def:rough semimartingale} is unique, as a consequence of \cite[Theorem~1.6]{friz2023rough}.
\end{remark}

\begin{definition}[Rough stochastic integral]\label{def:def of int}
Let $\mathbf X:=(\mathrm X, \mathbb X)$ belong to $ \mathscr{C}^{p\text{-}\mathrm{var}}([0,T],V)$. Let $(Y,Y')\in \mathrm{RSM}^{(p,p')\text{-}\mathrm{var}}_{\mathrm X}([0,T],V_1)$. By \cite[Theorem~1.7]{friz2023rough}, assuming \ref{(p,p')}, the limit in probability of the following Riemann sum exists:
\begin{equation}\label{e:Riemann sum}
\int_{u}^{v} Y_r d\mathbf X_r := \int_{u}^{v} (Y_r,Y'_r) d\mathbf X_r:= \lim_{\substack{|\pi|\to 0\\ \pi \in \mathcal P_{[u,v]}}}\sum_{[s,t]\in\pi} Y_{s}\delta \mathrm X_{s,t} + Y'_{s} \mathbb X_{s,t}.
\end{equation}
We call $\int_{u}^{v} Y_r d\mathbf X_r$ a rough stochastic integral.
\end{definition}

\begin{remark}\label{rem:degenerate Lebesgue}
When $\mathrm X$ is Lipschitz continuous and $\mathbf X$ is the canonical lift of $\mathrm X$ (we refer to \cite[Definition~7.2]{friz2010multidimensional} for the definition), then $[\mathbf X] = 0$. Moreover, for any continuous process $Y$, the integral $\int_{u}^{v} Y_{r} d\mathbf X_{r}$ defined by the Riemann sum in \eqref{e:Riemann sum} is equal to the Lebesgue integral $\int_{u}^{v} Y_{r} \partial_{r}\mathrm X_r dr$, where $\partial_r \mathrm X $ denotes the time derivative of $\mathrm X$. 
\end{remark}

\begin{remark}\label{rem:int by parts}
For $(Y,Y')\in \mathrm{RSM}^{(p,p')\text{-}\mathrm{var}}_{\mathrm X}([0,T],V_1)$ with $p'$ satisfying \ref{(p,p')}, by the integration by parts formula in \cite[Section~5.3]{friz2023rough}, the following equality holds:
\begin{equation}\label{e:def of inte}
\int_{u}^{v} Y_{r} d\mathbf X_{r} := M^{Y}_{v}\mathrm X_{v} - M^{Y}_{u}\mathrm X_{u} - \int_{u}^{v} \mathrm X_{r} dM^{Y}_{r} + \int_{u}^{v} (Y - M^{Y})_r d\mathbf X_{r}, 
\end{equation}
where $\int \mathrm X_{r} dM^{Y}_r$ is an It\^o integral, and $\int (Y - M^{Y})_r d\mathbf X_r$ is a  rough integral (see, e.g., \cite{friz2020course}) defined pathwise. 
\end{remark}

\begin{remark}\label{rem:Gamma-RSM}
For an \emph{RSM}, its integration against a rough path yields again an \emph{RSM}. 
More precisely, given $(Y, Y')\in\mathrm{RSM}^{(p,p')\text{-}\mathrm{var}}_{\mathrm X}([0,T],\mathscr L(V,V_1))$ with $p'$ satisfying \ref{(p,p')}, let $\Gamma_t := \int_{0}^{t} Y_r d\mathbf X_{r}$ be defined as in Definition~\ref{def:def of int}. 
By \cite[Theorem~1.7]{friz2023rough},
$(\Gamma, Y)\in\mathscr{D}^{(p,p)\text{-}\mathrm{var}}_{\mathrm X}([0,T],V_1)$ a.s. Consequently, $M^{\Gamma} = 0$ and $P^{\Gamma} = R^{\Gamma}$. 
\end{remark}

\section{Rough Stochastic Analysis of Rough Semimartingales}\label{sec:advanced analysis of rough path}

In this section, we establish several estimates for RSMs that will be used in the study of BSDE~\eqref{e:BSDE} in Section~\ref{sec:Rough BSDE}. Recall that $p\in(2,3)$ is a fixed number associated with a given rough path $\mathbf X=(\mathrm X, \mathbb X)$. Throughout this section, $V$, $V_1$, and $V_2$ are Euclidean spaces; and $p'$, $q$, and $q'$ are positive numbers satisfying 
\begin{equation}
\label{cond31}
p\le p', \quad p\le q \le q'.
\end{equation}

\subsection{Energy Estimates and Equivalence of BMO-Type Norms}\label{subsec:energy estimates and RSM estimates}

One of our main tools is the following energy inequality for adapted, non-decreasing processes, the proof of which can be referred to, e.g., \cite[Lemma~A.2]{le2022quantitative}. In the following, we denote $A_{u-} := \lim_{v \uparrow u} A_v$ for $u > 0$ and $A_{0-} := A_0$. 

\begin{lemma}\label{lem:Energy inequality}
Let $s,t \in [0,T]$ be two fixed times. Assume $C$ is a deterministic constant and that $\{A_r\}_{r \in [s,t]}$ is an adapted, right continuous, non-decreasing process such that 
\begin{equation*}
\|\E_{\tau} [A_{t} - A_{\tau-}]\|_{\infty} \le C,\ \text{ for all stopping times }\ \tau \in [s,t].
\end{equation*} 
That is, $\{A_r\}_{r \in [s,t]}$ is a \emph{BMO}-process on $[s,t]$. Then, for every stopping time $\tau' \in [s,t]$ and $k \ge 1$,
\begin{equation*}
\| \delta A_{\tau',t} | \mathcal F_{\tau'}\|_{k,\infty} \le |k!|^{1/k} \, C.
\end{equation*}
\end{lemma}

Recall $\| \|\delta V\|_{q\text{-}\mathrm{var}} \|_{\infty;k,\infty;[s,t]} := \sup_{(u,v) \in \Delta_{[s,t]}} \| \|\delta V\|_{q\text{-}\mathrm{var};[u,v]} |\mathcal F_u\|_{k,\infty}$. 
Lemma~\ref{lem:BMO of q-var} shows that the \textrm{BMO}-type norm $\| \|\delta V\|_{q\text{-}\mathrm{var}} \|_{\infty;k,\infty}$ with any $k \ge 1$ is dominated by the version with $k=1$. 

\begin{lemma}\label{lem:BMO of q-var}
For a one-parameter continuous process $V$, assume $\| \|\delta V\|_{q\text{-}\mathrm{var}} \|_{\infty;1,\infty;[s,t]} < \infty$. 

Then its $q$-variation, $\{  \|\delta V\|_{q\text{-}\mathrm{var};[s,r]}\}_{r\in [s,t]}$, is a \emph{BMO}-process on $[s,t]$. Moreover, we have for every $r \in [s,t]$ and $k \ge 1$,
\begin{equation}\label{e:enery esti for q-var}
\big\| \|\delta V\|_{q\text{-}\mathrm{var};[r,t]} \big| \mathcal F_{r} \big\|_{k,\infty} \le |k!|^{1/k} \, \big\| \|\delta V\|_{q\text{-}\mathrm{var}} \big\|_{\infty;1,\infty;[s,t]}. 
\end{equation}
Consequently, it holds $\big\| \|\delta V\|_{q\text{-}\mathrm{var}} \big\|_{\infty;k,\infty;[s,t]} \le |k!|^{1/k} \, \big\| \|\delta V\|_{q\text{-}\mathrm{var}} \big\|_{\infty;1,\infty;[s,t]}$.  
\end{lemma}
 
\begin{remark}\label{rem:use of BMO of q-var}
Note that the above lemma is the key observation to solve the problem of degenerating integrability when one considers a composition of rough semimartingales with smooth functions (see \eqref{e:dM and d hat M terms} in Proposition~\ref{prop:R^G - R^G}).  
When $q > 2$, the above equivalence of norms for continuous martingales also follows from \eqref{e:BMO vs p-var} in Lemma~\ref{lem:Lepingle}. 
\end{remark}

\begin{proof}
For $r \in [s,t]$, denote $A^r_u := \|\delta V\|_{q\text{-}\mathrm{var};[r,u]}$. Then, a.s., $u \mapsto \E_{u} [\delta A^r_{u,t}] = \E_{u} [\delta A^r_{r,t}] - \delta A^r_{r,u}$ is continuous. Moreover, for any stopping time $\tau \in [r,t]$, $\E_{\tau}[\delta A^{r}_{\tau,t}] = \E_{\tau}[\delta A^{r}_{r,t}] - A^r_{r,\tau} = \E_{u} [\delta A^r_{u,t}] \mid_{u = \tau}$. Hence, 
\begin{equation*}
\|\E_{\tau} [A^{r}_{t} - A^{r}_{\tau-} ]\|_{\infty} = \|\E_{\tau} [\delta A^{r}_{\tau,t} ]\|_{\infty} \le \sup_{u \in [r,t]} \| \E_u[\delta A^r_{u,t}] \|_{\infty}.
\end{equation*}

Furthermore, noting that $\delta A^r_{u,t} = \|\delta V\|_{q\text{-}\mathrm{var};[r,t]} - \|\delta V\|_{q\text{-}\mathrm{var};[r,u]} \le \|\delta V\|_{q\text{-}\mathrm{var};[u,t]}$ from the subadditivity of the $q$-variation norm, and that $[r,t] \subseteq [s,t]$, we obtain that for all $u \in [r,t]$, 
\begin{equation*}
\big\| \E_{u}[\delta A^{r}_{u,t}] \big\|_{\infty} \le \big\| \E_{u}[ \|\delta V\|_{q\text{-}\mathrm{var};[u,t]} ] \big\|_{\infty} \le \big\| \|\delta V\|_{q\text{-}\mathrm{var}} \big\|_{\infty;1,\infty;[s,t]} =: C. 
\end{equation*}

Consequently, $\|\E_{\tau} [A^{r}_{t} - A^{r}_{\tau-} ]\|_{\infty} \le C$ for any stopping time $\tau \in [r,t]$. By Lemma~\ref{lem:Energy inequality} with $\tau' = r$, 
\begin{equation*}
\big\| \|\delta V\|_{q\text{-}\mathrm{var};[r,t]} \big| \mathcal F_{r} \big\|_{k,\infty} = \big\| \delta A^{r}_{r,t} \big| \mathcal F_{r} \big\|_{k,\infty} \le |k!|^{1/k} C, 
\end{equation*}
which implies \eqref{e:enery esti for q-var}. The other inequality follows by taking the supremum over time. 
\end{proof}

\subsection{Estimates for Rough Semimartingales}\label{subsec:estimates for rough semimartingales}

Unless otherwise specified, throughout this subsection, we assume $\mathbf X \in \mathscr{C}^{p\text{-}\mathrm{var}}([0,T],V)$ and $(Y,Y')\in \mathrm{RSM}^{(p,p')\text{-}\mathrm{var}}_{\mathrm X}([0,T],V_1)$. 

The following lemma concerns the fact that the composition of $G(\bm\cdot)$ with an RSM remains an RSM. The following general integral form turns out to be helpful for later purposes (see Theorem~\ref{thm:comparison}).

\begin{lemma}\label{lem:int of RSM}
Assume $(p,p')$ satisfies \eqref{cond31}, and $\gamma - 2\ge p/p'$. Let $G \in C^{\gamma - 1}(\R,\R^d)$. Assume both $(Y,Y')$ and $(\bar Y,\bar Y')$ belong to $ \mathrm{RSM}^{(p,p')\text{-}\mathrm{var}}_{\mathrm X}([0,T],\R)$. Denote  
\begin{equation}\label{e:def of alpha}
\alpha_t := \int_{0}^{1}G( (1-\lambda) Y_t + \lambda\bar Y_t) d\lambda,\quad \alpha'_t := \int_{0}^{1} DG( (1-\lambda) Y_t + \lambda\bar Y_t) \cdot \big( (1-\lambda)Y'_{t} + \lambda \bar Y'_{t} \big) d\lambda.
\end{equation}

Then, we have $(\alpha,\alpha') \in \mathrm{RSM}^{(p,p')\text{-}\mathrm{var}}_{\mathrm X}([0,T],\R^d)$ with  the local martingale part given by 
\begin{equation}\label{e:rep of M^alpha}
M^{\alpha}_{t} = \int_{0}^{1} \Big\{\int_{0}^{t} DG( (1-\lambda) Y_r + \lambda\bar Y_r) \big[(1 - \lambda)dM_r^{Y} + \lambda dM^{\bar Y}_r\big] \Big\} d\lambda.
\end{equation}
In particular, when $\gamma = 3$ and $Y=\bar Y$, the pair $(G(Y), D G(Y)Y')$ is an \textrm{RSM}.

\end{lemma}

\begin{proof}
Noting that $p/p'\in(0,1]$, it suffices to prove the result further assuming $\gamma\le 3$.   For $\lambda \in [0,1]$ and $t\in[0,T]$, denote $Y^{\lambda}_t := (1-\lambda)Y_t + \lambda \bar Y_t$, $(Y^{\lambda})'_t := (1-\lambda)Y'_t + \lambda \bar Y'_t$, 
\begin{equation}\label{e:def of G^lambda}
G^{\lambda}_{t} := G(Y^{\lambda}_t),\ \text{ and }\ (G^\lambda)'_{t}:= DG( Y^{\lambda}_t ) \cdot (Y^{\lambda})'_t.
\end{equation}
For $(s,t)\in\Delta$, denote 
\begin{equation}\label{e:decom of G^lambda}
P^{\lambda}_{s,t} := \delta G^{\lambda}_{s,t} - (G^{\lambda})'_{s}\delta\mathrm X_{s,t} - \int_{s}^{t} DG(Y^{\lambda}_{r})dM^{Y^{\lambda}}_r.
\end{equation}

{\bf Step 1.}
In this step, we prove that $(G^{\lambda},(G^{\lambda})')$ is an RSM. Without loss of generality, assume $G(\bm\cdot) \in C^{\gamma - 1}_b$ and that every term below is bounded and hence has finite $k$-moments for any given $k \ge 1$; otherwise, the standard localization argument applies. For the simplicity of notations, denote $C_G := \|G(\bm\cdot)\|_{C^{\gamma - 1}_b}$. 

First, by the boundedness and the $(\gamma - 2)$-H\"older continuity of $DG$, for $t\in[0,T]$ it holds that
\begin{equation*}
\|\delta (G^{\lambda})'\|_{p'\text{-}\mathrm{var};[t,T]} \le C_G \Big( \|\delta (Y^{\lambda})'\|_{p'\text{-}\mathrm{var};[t,T]} + \|\delta Y^{\lambda}\|^{\gamma - 2}_{p\text{-}\mathrm{var};[t,T]} \|(Y^{\lambda})'\|_{\infty}\Big),\  
\|(G^{\lambda})'\|_{\infty} \le C_G \|(Y^{\lambda})'\|_{\infty},
\end{equation*}
where in the first inequality we used the condition $p'(\gamma-2)\ge p$. Clearly $\|(G^{\lambda})'\|_{\infty}$ is uniformly bounded in $\lambda \in [0,1]$. Furthermore, noting that $Y^{\lambda}$ ($(Y^{\lambda})'$ resp.) is a convex combination of $Y$ and $\bar Y$ ($Y'$ and $\bar Y'$ resp.), we have that for all $\lambda \in [0,1]$,  
\begin{equation}\label{e:G^lambda < Y}
\begin{aligned}
&\|\delta (G^{\lambda})'\|_{p'\text{-}\mathrm{var};[t,T]} \\
&\le C_G \Big( \|\delta Y'\|_{p'\text{-}\mathrm{var};[t,T]} + \|\delta \bar Y'\|_{p'\text{-}\mathrm{var};[t,T]} + \big(\|\delta Y\|^{\gamma - 2}_{p\text{-}\mathrm{var};[t,T]} + \|\delta \bar Y\|^{\gamma - 2}_{p\text{-}\mathrm{var};[t,T]}\big) (\|Y'\|_{\infty} + \|\bar Y'\|_{\infty}) \Big).
\end{aligned}
\end{equation}
This implies the uniform boundedness $\|\delta (G^{\lambda})'\|_{p'\text{-}\mathrm{var}}$, due to the finiteness of $\|\delta Y'\|_{p'\text{-}\mathrm{var}}$, $\|\delta \bar Y'\|_{p'\text{-}\mathrm{var}}$, $\|\delta Y\|_{p\text{-}\mathrm{var}}$, and $\|\delta \bar Y\|_{p\text{-}\mathrm{var}}$. 

Next, by the definition of $P^{\lambda}$ we have 
\begin{align}\nonumber
P^{\lambda}_{s,t} &= \int_{0}^{1}\left[D G(Y^{\lambda}_{s} + \theta\cdot \delta Y^{\lambda}_{s,t}) - D G(Y^{\lambda}_{s})\right]d\theta \cdot \delta M^{Y^{\lambda}}_{s,t} + \int_{0}^{1}D G(Y^{\lambda}_s + \theta \cdot \delta Y^{\lambda}_{s,t} )d\theta \cdot  P^{Y^{\lambda}}_{s,t}\\ \nonumber
&\quad + \int_{0}^{1} \left[D G(Y^{\lambda}_{s} + \theta \cdot \delta Y^{\lambda}_{s,t} ) - D G(Y^{\lambda}_{s})\right] d\theta \cdot ((Y^{\lambda})'_{s} \delta\mathrm X_{s,t}) - \int_{s}^{t} \left[D G(Y^{\lambda}_{r}) - D G(Y^{\lambda}_{s})\right]dM^{Y^{\lambda}}_{r}\\ \label{e:R^lambda}
&=: K^{\lambda;1}_{s,t} + K^{\lambda;2}_{s,t} + K^{\lambda;3}_{s,t} + K^{\lambda;4}_{s,t}.
\end{align}
For $K^{\lambda;1}$ and $K^{\lambda;3}$, since $DG(\bm\cdot)$ is $(\gamma - 2)$-H\"older continuous, we have
\begin{equation*}
\int_{0}^{1}|DG(Y^{\lambda}_{s} + \theta\cdot \delta Y^{\lambda}_{s,t}) - DG(Y^{\lambda}_{s})| d\theta \le \|G(\bm\cdot)\|_{C^{\gamma-1}_b} |\delta Y^{\lambda}_{s,t}|^{\gamma - 2} = C_G |\delta Y^{\lambda}_{s,t}|^{\gamma - 2}.
\end{equation*}
Thus, noting $\frac{pp'}{p+p'} \ge \frac{p}{\gamma - 1}$, by H\"older's inequality we have for $t\in [0,T]$, 
\begin{equation}\label{e:K lambda 1}
\|K^{\lambda;1}\|_{\frac{pp'}{p+p'}\text{-}\mathrm{var};[t,T]} \le \|K^{\lambda;1}\|_{\frac{p}{\gamma - 1}\text{-}\mathrm{var};[t,T]} \le C_{G} \|\delta Y^{\lambda}\|^{\gamma - 2}_{p\text{-}\mathrm{var};[t,T]} \|\delta M^{Y^{\lambda}}\|_{p\text{-}\mathrm{var};[t,T]},
\end{equation}
and similarly, 
\begin{equation}\label{e:K lambda 3}
\|K^{\lambda;3}\|_{\frac{pp'}{p+p'}\text{-}\mathrm{var};[t,T]} \le \|K^{\lambda;3}\|_{\frac{p}{\gamma - 1}\text{-}\mathrm{var};[t,T]}\le C_{G} \|(Y^{\lambda})'\|_{\infty;[t,T]} \|\delta Y^{\lambda}\|^{\gamma - 2}_{p\text{-}\mathrm{var};[t,T]} \|\delta \mathrm X\|_{p\text{-}\mathrm{var};[t,T]}.
\end{equation} 

For $K^{\lambda;2}$, by the boundedness of $\|DG(\bm\cdot)\|_{\infty}$, we have
\begin{equation}\label{e:K lambda 2}
\|K^{\lambda;2}\|_{\frac{pp'}{p + p'}\text{-}\mathrm{var};[t,T]} \le \|DG(\bm\cdot)\|_{\infty} \cdot \|P^{Y^{\lambda}}\|_{\frac{pp'}{p + p'}\text{-}\mathrm{var};[t,T]}.
\end{equation}
To estimate $K^{\lambda;4}$, for $k\ge 1$, by Lemma~\ref{lem:p-var martingale estim} with $r = \frac{pp'}{p+p'}$, $r_1 = \frac{p}{\gamma - 2}$, $k_1 = 2k$, $k_0 = 2k$, 
\begin{equation}\label{e:K lambda 4}
\begin{aligned}
\big\| \| K^{\lambda;4} \|_{\frac{pp'}{p+p'}\text{-}\mathrm{var};[t,T]} \big| \mathcal F_{t}\big\|_{k} &\lesssim_{k,\gamma} \big\| \|\delta DG(Y^{\lambda}_{\bm\cdot})\|_{\frac{p}{\gamma - 2}\text{-}\mathrm{var};[t,T]} \big| \mathcal F_{t} \big\|_{2k} \big\| \delta \langle M^{Y^{\lambda}} \rangle^{\frac{1}{2}}_{t,T} \big| \mathcal F_{t} \big\|_{2k} \\
&\le C_G \big\| \|\delta Y^{\lambda}\|_{p\text{-}\mathrm{var};[t,T]} \big| \mathcal F_t \big\|^{\gamma - 2}_{2k} \big\| \delta \langle M^{Y^{\lambda}} \rangle^{\frac{1}{2}}_{t,T} \big| \mathcal F_{t} \big\|_{2k}. 
\end{aligned}
\end{equation}
Combining \eqref{e:R^lambda}, \eqref{e:K lambda 1}, \eqref{e:K lambda 3}, \eqref{e:K lambda 2}, and \eqref{e:K lambda 4}, we have that a.s., $P^{\lambda}$ has finite $\frac{pp'}{p+p'}$-variation. This together with the finiteness of $\|\delta G^{\lambda}\|_{p\text{-}\mathrm{var}}+\|\delta (G^{\lambda})'\|_{p'\text{-}\mathrm{var}} + \|(G^{\lambda})'\|_{\infty}$ and the decomposition \eqref{e:decom of G^lambda} implies that $(G^{\lambda},(G^{\lambda})') \in \mathrm{RSM}^{(p,p')\text{-}\mathrm{var}}_{\mathrm X}$.

{\textbf{Step 2.}} In this step, we prove that $(\alpha,\alpha')$ is an RSM. 
Since $\alpha'$ is given by  the integral of $(G^{\lambda})'$, which is uniformly bounded in $\lambda\in[0,1]$, 
$\|\alpha'\|_{\infty}$ is bounded.
Furthermore, in view of the inequality \eqref{e:G^lambda < Y}, Minkowski's inequality yields  
\begin{equation*} \|\delta \alpha'\|_{p'\text{-}\mathrm{var};[t,T]} 
\le C_G \times \text{RHS of \eqref{e:G^lambda < Y}}.
\end{equation*}
Next, we establish the \textrm{RSM} decomposition of $(\alpha,\alpha')$. Note that the following equality holds:
\begin{equation*}
\delta \alpha_{s,t} - \alpha'_{s} \delta \mathrm X_{s,t} = \int_{0}^{1} P^{\lambda}_{s,t}  d\lambda + \int_{0}^{1} \delta M^{G^{\lambda}}_{s,t} d\lambda =: P_{s,t} + \delta M_{s,t}, 
\end{equation*}
where $P_{s,t} := \int_{0}^{1} P^{\lambda}_{s,t}  d\lambda$ and $M_{t} := \int_{0}^{1} M^{G^{\lambda}}_{t} d\lambda$. By Fubini's theorem, we have, noting $M^{G^{\lambda}}_t = \int_{0}^{t} DG(Y^{\lambda}_r)dM^{Y^{\lambda}}_r$ from \eqref{e:decom of G^lambda}, 
\begin{equation*}
M_{t} = \int_{0}^{1} \E_{t}\Big[\int_{0}^{T} DG(Y^{\lambda}_r)dM^{Y^{\lambda}}_r \Big] d\lambda = \E_{t}\Big[\int_{0}^{1}\int_{0}^{T}  DG(Y^{\lambda}_r)dM^{Y^{\lambda}}_r d\lambda \Big] = \E_{t}[M_T]. 
\end{equation*}
Therefore, $M_t$ is a continuous martingale, and then $(s,t)\mapsto P_{s,t}$ is continuous a.s. Moreover, similar to $\|\delta \alpha'\|_{p'\text{-}\mathrm{var}}$, Minkowski's inequality again implies the finiteness of the $\frac{pp'}{p+p'}$-variation of $P$. Consequently, we get that $(\alpha,\alpha')\in \mathrm{RSM}^{(p,p')\text{-}\mathrm{var}}_{\mathrm X}$ with $M^{\alpha} = M$ and $P^{\alpha} = P$.  
\end{proof}

Below we obtain  upper bounds for the pathwise remainder $P^{G(Y)}$ as well as for $M^{G(Y)}$. Recall $\frac{q''}{2} := \frac{qq'}{q+q'}$.

\begin{proposition}\label{prop:G(Y) in D}
Assume that \eqref{cond31} and $q < \min\{4,q'\}$ hold. Let $G\in C_b^2(V_1,V_2)$ with
$\|G\|_{C_b^2}\le L$ for some constant $L>0$. Fix $k\ge1$ and set
$k_0:=\frac{2k}{q-2}$. Then, we have that for $(s,t) \in \Delta$, 
\begin{align}\nonumber
\big\| \|P^{G(Y)}\|_{\frac{q''}{2}\text{-}\mathrm{var};[s,t]} | \mathcal F_{s} \big\|_{k} 
& \lesssim_{k, L} \big\| \|\delta Y\|_{q\text{-}\mathrm{var};[s,t]} | \mathcal F_{s} \big\|^{\frac{4-q}{2}}_{k} \big\| \delta \langle M^{Y}\rangle^{\frac{1}{2}}_{s,t} | \mathcal F_{s} \big\|_{k_0} + \big\|\|P^{Y}\|_{\frac{q''}{2}\text{-}\mathrm{var};[s,t]} | \mathcal F_{s}\big\|_{k} \\  \label{e:R^G(Y)}
&\qquad + \|\delta \mathrm X\|_{p\text{-}\mathrm{var};[s,t]}\big\| \|\delta Y\|_{q\text{-}\mathrm{var};[s,t]} | \mathcal F_{s} \big\|^{\frac{q}{q'}}_{k} \|Y'\|_{\infty} ,
\end{align}
and
\begin{equation}\label{e:M^G(Y)}
\big\|\delta \langle M^{G(Y)}\rangle^{\frac{1}{2}}_{s,t} | \mathcal F_{s}\big\|_{k}\lesssim_{k,L} \big\|\delta \langle M^{Y}\rangle^{\frac{1}{2}}_{s,t} | \mathcal F_{s} \big\|_{k}.
\end{equation}
\end{proposition}

\begin{proof}
For $(s,t) \in \Delta$, by standard calculus, we obtain 
\begin{align}\nonumber
&G(Y_{t}) - G(Y_{s}) - D G(Y_{s}) Y'_{s} \delta\mathrm X_{s,t}\\ \label{e:G(Yt)-G(Ys)}
&= \int_{0}^{1} D G(Y_{s} + \lambda\cdot \delta Y_{s,t} )d\lambda \cdot \delta Y_{s,t} - \int_{0}^{1} D G(Y_{s}) d\lambda (Y'_{s} \delta\mathrm X_{s,t})\\ \nonumber
&= \int_{0}^{1} D G(Y_{s} + \lambda\cdot \delta Y_{s,t} ) d\lambda (\delta M^{Y}_{s,t} + P^{Y}_{s,t}) + \int_{0}^{1} \left[ D G(Y_{s} + \lambda \cdot \delta Y_{s,t} ) - D G(Y_{s})\right] d\lambda (Y'_{s} \delta\mathrm X_{s,t}),
\end{align}
where the last step is due to $\delta Y_{s,t} = \delta M^{Y}_{s,t} + P^Y_{s,t} + Y'_{s} \delta \mathrm X_{s,t}$. For brevity, we denote 
\begin{equation*}
M := M^{G(Y)}\text{ and }P := P^{G(Y)}.
\end{equation*}
By \eqref{e:rep of M^alpha} in Lemma~\ref{lem:int of RSM} with $Y = \bar Y$, we have $M_t = \int_{0}^{t} D G(Y_{r}) d M^{Y}_{r}$, $t \in [0,T]$. 
For $(s,t)\in \Delta$, by the definition of $P$, we get
\begin{equation*}
\begin{aligned}
P_{s,t}&:=\text{RHS of \eqref{e:G(Yt)-G(Ys)}} - \delta M_{s,t}\\
&= \int_{0}^{1}\left[D G(Y_{s} + \lambda\cdot \delta Y_{s,t}) - D G(Y_{s})\right]d\lambda \cdot \delta M^{Y}_{s,t} + \int_{0}^{1} D G(Y_s + \lambda \cdot \delta Y_{s,t} )d\lambda \cdot  P^Y_{s,t}\\
&\quad + \int_{0}^{1} \left[D G(Y_{s} + \lambda \cdot \delta Y_{s,t} ) - D G(Y_{s})\right] d\lambda (Y'_{s} \delta\mathrm X_{s,t}) - \int_{s}^{t} \left[D G(Y_{r}) - D G(Y_{s})\right]dM^{Y}_{r}\\
&=: K^1_{s,t} + K^2_{s,t} + K^{3}_{s,t} + K^{4}_{s,t}.
\end{aligned}
\end{equation*}

For $K^{1}$, since $DG(\bm\cdot)$ and $D^2 G(\bm\cdot)$ are bounded, we have for all $(s,t) \in \Delta$ and $(u,v) \in \Delta_{[s,t]}$, 
\begin{equation*}
\int_{0}^{1}|DG(Y_{u} + \lambda\cdot \delta Y_{u,v}) - DG(Y_{u})| d\lambda \le 2^{\frac{q-2}{2}} \|DG(\bm\cdot)\|^{\frac{q-2}{2}}_{\infty} \|D^2 G(\bm\cdot)\|^{\frac{4-q}{2}}_{\infty} \cdot |\delta Y_{u,v}|^{\frac{4-q}{2}} =: C_G |\delta Y_{u,v}|^{\frac{4-q}{2}},
\end{equation*}
where $C_G := 2^{\frac{q-2}{2}} \|DG(\bm\cdot)\|^{\frac{q-2}{2}}_{\infty} \|D^2G(\bm\cdot)\|^{\frac{4-q}{2}}_{\infty}$. 
Hence, we have, with $Q := (\frac{1}{2} - \frac{1}{q} + \frac{1}{q'})^{-1} > 2$,  
\begin{equation}\label{e:K1}
\begin{aligned}
\|K^1\|_{\frac{q''}{2}\text{-var};[s,t]}&\le C_G \| \, |\delta Y|^{\frac{4 - q}{2}} \cdot \delta M^{Y} \|_{\frac{q''}{2}\text{-}\mathrm{var};[s,t]}  \le C_G \|\delta Y\|^{\frac{4-q}{2}}_{q\text{-var};[s,t]} \|\delta M^Y\|_{Q\text{-var};[s,t]},
\end{aligned}
\end{equation}
where the second inequality follows from  H\"older's inequality. For $K^2$, it is straightforward  to get 
\begin{equation*}
\|K^{2} \|_{\frac{q''}{2}\text{-var};[s,t]} \le  \|DG(\bm\cdot)\|_\infty \cdot \|P^Y\|_{\frac{q''}{2}\text{-var};[s,t]}.
\end{equation*}
For $K^3$, using the same argument leading to  \eqref{e:K1}, we obtain 
\begin{equation}\label{e:K3}
\| K^{3}\|_{\frac{q''}{2}\text{-var};[s,t]}\lesssim  \|DG(\bm\cdot)\|^{\frac{q'-q}{q'}}_\infty \|D^2 G(\bm\cdot)\|^{\frac{q}{q'}}_\infty \cdot \|Y'\|_{\infty;[s,t]} \|\delta Y\|^{\frac{q}{q'}}_{q\text{-var};[s,t]} \|\delta \mathrm X\|_{q\text{-var};[s,t]}.
\end{equation}

Recall $k_0 = \frac{2k}{q-2}$. Taking conditional $k$-moments in \eqref{e:K1}--\eqref{e:K3} 
and applying H\"older's inequality, we have 
\begin{equation}\label{e:K1K2K3}
\begin{aligned}
&\big\| \|K^{1} + K^2 + K^3\|_{\frac{q''}{2}\text{-var};[s,t]} | \mathcal F_{s} \big\|_{k}\\
&\qquad \lesssim_{k,L} \big\|\|\delta Y\|_{q\text{-var};[s,t]} | \mathcal F_{s} \big\|^{\frac{4-q}{2}}_{k}\big\| 
\|\delta M^{Y}\|_{Q\text{-var};[s,t]} | \mathcal F_{s} \big\|_{k_0} + \big\| \|P^Y\|_{\frac{q''}{2}\text{-var};[s,t]} \big| \mathcal F_{s} \big\|_{k}\\
&\qquad\qquad + \|\delta \mathrm X\|_{q\text{-var};[s,t]}\big\| \|\delta Y\|_{q\text{-var};[s,t]} | \mathcal F_{s} \big\|^{\frac{q}{q'}}_{k} \|Y'\|_{\infty}\\
&\qquad \lesssim_{k} \big\| \|\delta Y\|_{q\text{-var};[s,t]} | \mathcal F_{s} \big\|^{\frac{4-q}{2}}_{k} \big\| \delta \langle M^{Y}\rangle^{\frac{1}{2}}_{s,t} | \mathcal F_{s} \big\|_{k_0}
 + \big\| \|P^Y\|_{\frac{q''}{2}\text{-var};[s,t]} | \mathcal F_{s} \big\|_{k}\\
&\qquad\qquad + \|\delta \mathrm X\|_{q\text{-var};[s,t]} \big\| \|\delta Y\|_{q\text{-var};[s,t]} | \mathcal F_{s} \big\|^{\frac{q}{q'}}_{k} \|Y'\|_{\infty},
\end{aligned}
\end{equation}
where the second inequality follows from L\'epingle's inequality (Lemma~\ref{lem:Lepingle}).

For the term $K^{4}$, applying Lemma~\ref{lem:p-var martingale estim} with $r = \frac{q''}{2}, r_{1} = \frac{2q}{4-q}, k_{1} = \frac{2k}{4-q}$ and $k_{0} = \frac{2k}{q-2}$, we have
\begin{equation}\label{e:K4}
\begin{aligned}
\big\| \|K^{4}\|_{\frac{q''}{2}\text{-var};[s,t]} | \mathcal F_{s} \big\|_{k} &\lesssim_{k} \big\| \|\delta D G(Y_{\bm\cdot})\|_{\frac{2q}{4-q}\text{-var};[s,t]} | \mathcal F_{s} \big\|_{k_1} \big\| \delta \langle M^{Y}\rangle^{\frac{1}{2}}_{s,t} | \mathcal F_{s} \big\|_{k_0}\\
&\lesssim_{k} C_G \big\| \|\delta Y\|_{q\text{-var};[s,t]} | \mathcal F_{s} \big\|^{\frac{4-q}{2}}_{k} \big\| \delta \langle M^{Y} \rangle^{\frac{1}{2}}_{s,t} | \mathcal F_{s} \big\|_{k_0},
\end{aligned}
\end{equation}
where the second inequality is due to the fact that $DG$ is both Lipschitz and bounded. Combining \eqref{e:K1K2K3} and \eqref{e:K4}, we get the desired \eqref{e:R^G(Y)}.

Finally, recalling $M_t = \int_{0}^{t} D G(Y_{r}) d M^{Y}_{r}$, we get $\big\|\delta \langle M\rangle^{1/2}_{s,t}| \mathcal F_{s} \big\|_{k} \le \|DG(\bm\cdot)\|_{\infty} \big\| \delta \langle M^{Y}\rangle^{1/2}_{s,t} | \mathcal F_{s}\big\|_{k}$
and then \eqref{e:M^G(Y)} follows.
\end{proof}

The following proposition establishes an estimate for the rough stochastic integral and, in particular, implies that we can improve the regularity of the integral from $\mathrm{RSM}^{(p,p')\text{-}\mathrm{var}}_{\mathrm X}$ to $\mathrm{RSM}^{(p,p)\text{-}\mathrm{var}}_{\mathrm X}$.

\begin{proposition}\label{prop:estim of integral}
Assume \ref{(p,p')} holds. For $(Y,Y')\in\mathrm{RSM}^{(p,p')\text{-}\mathrm{var}}_{\mathrm X}([0,T], \mathscr L(V,V_1) )$, denote the rough stochastic integral $\Gamma_{t} := \int_{0}^{t}Y_{r}d\mathbf X_r$, $t\in[0,T]$,
as in Definition~\ref{def:def of int}.  

Then, $(\Gamma,Y)$ belongs to $\mathrm{RSM}^{(p,p)\text{-}\mathrm{var}}_{\mathrm X}([0,T],V_1)$. Moreover, assuming \ref{(p,q,q')} and $k\ge 1$, the remainder $R^{\Gamma}_{s,t} := \delta\Gamma_{s,t} - Y_{s}\delta\mathrm X_{s,t}$ satisfies the following estimate for all $(s,t) \in \Delta$:
\begin{equation*}
\begin{aligned}
\big\| \|R^{\Gamma}\|_{\frac{p}{2}\text{-}\mathrm{var};[s,t]} \big| \mathcal F_{s} \big\|_{k}  &\lesssim_{k} \|\mathbb X\|_{\frac{p}{2}\text{-}\mathrm{var};[s,t]}  \big( \|Y'\|_{\infty} + \big\|\|\delta Y'\|_{q'\text{-}\mathrm{var};[s,t]}\big|\mathcal F_{s}\big\|_{k} \big)\\
&\qquad + \|\delta \mathrm X\|_{p\text{-}\mathrm{var};[s,t]} \Big( \big\|\|P^Y\|_{\frac{q''}{2}\text{-}\mathrm{var};[s,t]}\big|\mathcal F_{s}\big\|_{k} + \|\delta M^{Y}\|_{\mathrm{BMO};[s,t]} \Big).  
\end{aligned}
\end{equation*}
\end{proposition}

\begin{remark}\label{rem:for Psi}
The expression on the right-hand side is closely related to the \emph{RSM} seminorm $\thicknm{(Y,Y')}_{\mathrm X;q,q';k,\infty}$ which is introduced in \eqref{e:Psi} in Section~\ref{subsec:seminorms}.
\end{remark}

\begin{proof} 
By \cite{friz2023rough}, we have $(\Gamma,Y) \in \mathrm{RSM}^{(p,p)\text{-}\mathrm{var}}_{\mathrm X}([0,T],V_1)$; see Remark~\ref{rem:Gamma-RSM} for details. Note that 
\begin{equation}\label{e:R^Gamma = Xi^1 + Xi^2}
R^{\Gamma}_{s,t} = \int_{s}^{t}\delta Y_{s,r} d\mathbf X_r=\Xi^1_{s,t} + \Xi^2_{s,t} + Y'_{s}\mathbb X_{s,t},
\end{equation}
where 
\begin{equation*}
\Xi^1_{s,t} := \int_{s}^{t}\delta Y_{s,r} d\mathbf X_r - Y'_{s} \mathbb X_{s,t} - \int_{s}^{t}\delta M^{Y}_{s,r}d\mathbf X_{r}\quad \text{ and }\quad \Xi^2_{s,t} := \int_{s}^{t}\delta M^{Y}_{s,r} d\mathbf X_{r}.
\end{equation*}

Recall that $\delta Y_{s,r} = \delta (Y - M^Y)_{s,r} + \delta M^{Y}_{s,r}$. Hence, denoting $A_{s,t} := (Y - M^Y)_{s}\delta\mathrm X_{s,t} + Y'_{s}\mathbb X_{s,t}$, 
\begin{equation*}
\Xi^1_{s,t} = \int_{s}^{t}\delta (Y - M^Y)_{s,r} d\mathbf X_r + \int_{s}^{t}\delta M^{Y}_{s,r} d\mathbf X_{r} - Y'_{s}\mathbb X_{s,t} - \int_{s}^{t}\delta M^{Y}_{s,r} d\mathbf X_{r} = \int_s^t (Y - M^Y)_{r} d\mathbf X_r - A_{s,t}.
\end{equation*}
Since $((Y - M^Y),Y')\in\mathscr{D}^{(p,p')\text{-}\mathrm{var}}_{\mathrm X}$ a.s., by the definition of rough integrals (see, e.g., \cite{friz2020course,FrizZhang-2018}), 
\begin{equation*}
\int_s^t (Y - M^Y)_{r} d\mathbf X_r = \lim_{|\pi|\to 0}\sum_{[u,v]\in\pi} A_{u\vee s,v\land t}\quad \text{a.s.}
\end{equation*}
Note $\Xi^{1}_{s,t} = \int_{s}^{t} (Y - M^Y)_{r}d\mathbf X_{r} - A_{s,t}$. Then, since \ref{(p,q,q')} holds, the sewing lemma (see \cite[Theorem~2.5]{FrizZhang-2018}) yields a universal constant $C>0$ such that, a.s., the following estimate holds 
\begin{equation}\label{e:Xi <= Y X + R^Y X}
\|\Xi^1\|_{\frac{p}{2}\text{-}\mathrm{var};[s,t]} \le C \Big( \|\delta Y'\|_{q'\text{-}\mathrm{var};[s,t]} \|\mathbb X\|_{\frac{p}{2}\text{-var};[s,t]} +  \|P^{Y}\|_{\frac{q''}{2}\text{-var};[s,t]}\|\delta \mathrm X\|_{p\text{-var};[s,t]} \Big).
\end{equation}

For the term $\Xi^{2}_{s,t}$, by \eqref{e:def of inte} we get $\Xi^{2}_{s,t} = \int_{s}^{t}\delta \mathrm X_{r,t} dM^{Y}_{r} = \delta \mathrm X_{s,t}\delta M^{Y}_{s,t} - \int_{s}^{t}\delta \mathrm X_{s,r} dM^{Y}_{r}$. 
Then, Lemma~\ref{lem:p-var martingale estim} and Lemma~\ref{lem:Lepingle} yield 
\begin{equation}\label{e:A <= M X}
\big\|\|\Xi^{2}\|_{\frac{p}{2}\text{-var};[s,t]}\big|\mathcal F_{s}\big\|_{k} \lesssim \big\| \|\delta M^{Y}_{s,\cdot}\|_{\infty;[s,t]} \big|\mathcal F_{s}\big\|_{k} \|\delta \mathrm X\|_{p\text{-}\mathrm{var};[s,t]}  \lesssim_k  \|\delta M^{Y}\|_{\mathrm{BMO};[s,t]}\|\delta \mathrm X\|_{p\text{-var};[s,t]}.
\end{equation}

The $\frac{p}{2}$-variation of the last term $Y_s'\mathbb X_{s,t}$ in \eqref{e:R^Gamma = Xi^1 + Xi^2} is plainly bounded by $ \|Y'\|_{\infty;[s,t]} \|\mathbb X\|_{\frac{p}{2}\text{-}\mathrm{var};[s,t]}$.
Combining this with \eqref{e:R^Gamma = Xi^1 + Xi^2}, \eqref{e:Xi <= Y X + R^Y X}, \eqref{e:A <= M X}, we get the desired inequality.
\end{proof}

The following corollary provides the estimate for the rough stochastic integral term $\int G(Y_r) d\mathbf X_r$ under the rough semimartingale norm. 

\begin{corollary}\label{cor:int of the solution}
Assume \ref{(p,p')} and  \ref{(p,q,q')} hold with $q < q'$. For $G\in C^{2}_b(V_2,\mathscr L(V,V_1))$ and $(Y,Y') \in \mathrm{RSM}^{(p,p')\text{-}\mathrm{var}}_{\mathrm X}([0,T],V_2)$ with $\| Y' \|_{\infty} + \| G \|_{C^{2}_{b}} \le K$, set $\Lambda_{t} :=\int_{0}^{t}G(Y_{r})d\mathbf X_{r}$. 
Then, for any $k \ge 1$, there is a constant $C_{k,K} > 0$ depending only on $(k,K)$ such that for all $(s,t)\in\Delta$,
\begin{equation}\label{e:(Gamma,-H(Y))}
\begin{aligned}
&\big\| \|R^{\Lambda}\|_{\frac{p}{2}\text{-}\mathrm{var};[s,t]}\big| \mathcal F_{s} \big\|_{k} + \big\| \|\delta \Lambda\|_{p\text{-}\mathrm{var};[s,t]} \big| \mathcal F_{s}\big\|_{k} \\
&\le C_{k,K} |\delta \mathbf X|_{p\text{-}\mathrm{var};[s,t]} \Big(1 + |\delta \mathbf X|_{p\text{-}\mathrm{var};[s,t]} + \big\| \|\delta Y'\|_{q'\text{-}\mathrm{var};[s,t]}\big|\mathcal F_{s}\big\|_{k}\\
& \quad + (1+  |\delta \mathbf X|_{p\text{-}\mathrm{var};[s,t]}) \cdot \big\|\|\delta Y\|_{q\text{-}\mathrm{var};[s,t]} \big| \mathcal F_{s}\big\|_{k} +  \|\delta M^{Y}\|^{\frac{2}{q-2}}_{\mathrm{BMO};[s,t]} + \big\|  \|P^Y\|_{\frac{q''}{2}\text{-}\mathrm{var};[s,t]}\big|\mathcal F_{s}\big\|_{k}  \Big).
\end{aligned}
\end{equation}
\end{corollary}

\begin{proof}
As explained at the beginning of this subsection, $(G(Y), DG(Y) Y') \in \mathrm{RSM}^{(p,p')\text{-}\mathrm{var}}_{\mathrm X}$. Denote the corresponding Gubinelli derivative by $G(Y)' = DG(Y) Y'$. Then, by Proposition~\ref{prop:G(Y) in D} (where $q < q'$ is required), we obtain  
\begin{equation*}
\begin{aligned}
\big\|\|P^{G(Y)}\|_{\frac{q''}{2}\text{-var};[s,t]}\big|\mathcal F_{s}\big\|_{k} &\lesssim_{k,K} \big\|\|\delta Y\|_{q\text{-var};[s,t]}\big|\mathcal F_{s}\big\|^{\frac{4-q}{2}}_{k}\|\delta M^{Y}\|_{\mathrm{BMO};[s,t]} \\
&\qquad  + \big\|\|P^Y\|_{\frac{q''}{2}\text{-var};[s,t]}\big|\mathcal F_{s}\big\|_{k} + |\delta \mathbf X|_{p\text{-var};[s,t]} \big\|\|\delta  Y\|_{q\text{-var};[s,t]}\big|\mathcal F_{s}\big\|^{\frac{q}{q'}}_{k},
\end{aligned}
\end{equation*}
and $\|\delta M^{G(Y)}\|_{\mathrm{BMO}} \lesssim_{K} \|\delta M^{Y}\|_{\mathrm{BMO}}$. This, together with Proposition~\ref{prop:estim of integral} yields
\begin{align*}\nonumber
&\big\|\|R^{\Lambda}\|_{\frac{p}{2}\text{-}\mathrm{var};[s,t]} \big| \mathcal F_{s}\big\|_{k} \\ \nonumber
& \lesssim_{k} |\delta \mathbf X|_{p\text{-}\mathrm{var};[s,t]} \Big( \left\|G(Y)'\right\|_{\infty} + \big\|\|\delta G(Y)'\|_{q'\text{-var};[s,t]}\big|\mathcal F_{s}\big\|_{k} + \big\|\|\delta Y\|_{q\text{-var};[s,t]} \big|\mathcal F_{s}\big\|^{\frac{4-q}{2}}_{k} \|\delta M^{Y}\|_{\mathrm{BMO};[s,t]} \\
& \qquad + \big\|  \|P^Y\|_{\frac{q''}{2}\text{-}\mathrm{var};[s,t]}\big|\mathcal F_{s}\big\|_{k} +  |\delta \mathbf X|_{p\text{-var};[s,t]} \big\|\|\delta Y\|_{q\text{-var};[s,t]} \big|\mathcal F_{s}\big\|^{\frac{q}{q'}}_{k} + \|\delta M^{Y}\|_{\mathrm{BMO};[s,t]} \Big)\\ \nonumber
& \lesssim_{K} |\delta \mathbf X|_{p\text{-}\mathrm{var};[s,t]} \Big( 1 + \big\| \|\delta Y\|_{q\text{-var};[s,t]}\big|\mathcal F_{s}\big\|_{k} 
 + \big\| \|\delta Y'\|_{q'\text{-var};[s,t]}\big|\mathcal F_{s}\big\|_{k}\\ \nonumber
& \qquad + \big\|\|\delta Y\|_{q\text{-var};[s,t]}\big|\mathcal F_{s}\big\|^{\frac{4-q}{2}}_{k}  \|\delta M^{Y}\|_{\mathrm{BMO};[s,t]} + \big\| \|P^Y\|_{\frac{q''}{2}\text{-}\mathrm{var};[s,t]}\big|\mathcal F_{s}\big\|_{k} \\
&\qquad +  |\delta \mathbf X|_{p\text{-var};[s,t]} \big\|\|\delta Y\|_{q\text{-var};[s,t]} \big|\mathcal F_{s}\big\|^{\frac{q}{q'}}_{k} + \|\delta M^{Y}\|_{\mathrm{BMO};[s,t]}\Big).
\end{align*}

Noting $|a|^{(4-q)/2}|b| \le |a| + |b|^{2/(q-2)}$ and $|a|^{q/q'} \le 1 + |a|$, we get 
\begin{equation}\label{e:Lambda RHS}
\big\| \|R^{\Lambda}\|_{\frac{p}{2}\text{-var};[s,t]} \big| \mathcal F_{s} \big\|_{k} \lesssim_{k,K} \text{ RHS of \eqref{e:(Gamma,-H(Y))}}.
\end{equation}
Then, the desired estimate \eqref{e:(Gamma,-H(Y))} follows from  \eqref{e:Lambda RHS} and the fact $\delta \Lambda_{s,t} = G(Y_{s})\delta \mathrm X_{s,t} + R^{\Lambda}_{s,t}$. 
\end{proof}

\begin{remark}\label{rem:on the exponent}
The condition $0<(4-q)/2<1$, equivalently $q\in(2,4)$, appearing in \eqref{e:R^G(Y)} of Proposition~\ref{prop:G(Y) in D}, is crucial for the proof of Corollary~\ref{cor:int of the solution}. Combined with Young's inequality, this condition leads to the exponent $2/(q-2)$ in the \emph{BMO} norm of $M^{Y}$ in \eqref{e:(Gamma,-H(Y))}. The additional requirement that $2/(q-2) < 2$, namely $q>3$, in turn is the critical threshold in the proof of Proposition~\ref{prop:BMO and ess sup}. 
\end{remark}

The following two propositions provide quantitative estimates for the difference of two RSMs. Let $\mathbf X = (\mathrm X, \mathbb X)$ and $\bar{\mathbf X} = (\bar{\mathrm X},\bar{\mathbb X})$ belong to $\mathscr C^{p\text{-}\mathrm{var}}([s,t],V)$, and 
\begin{equation}\label{e:Y in RSM}
(Y,Y') \in \mathrm{RSM}^{(p,p')\text{-}\mathrm{var}}_{\mathrm X}([s,t],V_1),\quad (\bar Y,\bar Y') \in \mathrm{RSM}^{(p,p')\text{-}\mathrm{var}}_{\bar{\mathrm X}}([s,t],V_1).
\end{equation}

\begin{proposition}\label{prop:R^G - R^G}
Assume \eqref{cond31} and \eqref{e:Y in RSM} hold. 
Let $G\in C^{\gamma}_{b}(V_1,V_2)$ for $\gamma \in (2,3]$ satisfying $\frac{\gamma - 1}{q} \ge \frac{2}{q''}$ or equivalently $q\le q'(\gamma-2)$ (recalling $\frac{2}{q''} = \frac{q + q'}{qq'}$). Then for every $k \ge 1$, we have 
\begin{equation}\label{e:E. R - R}
\Big\| \|P^{G(Y)} - P^{G(\bar Y)} \|_{\frac{q''}{2}\text{-}\mathrm{var};[s,t]} \big| \mathcal F_{s} \Big\|_{k,\infty} \lesssim_{k,\gamma} \|G\|_{C^{\gamma}_b}\Big( L_{1} + L_{2} + L_{3} + L_4 \Big),
\end{equation}
where $L_i$, $i = 1,2,3,4$, are defined as follows:  
\begin{equation*}
\begin{cases}\displaystyle
L_1 := \|Y - \bar Y\|_{\mathcal L^{\infty}([s,t]\times\Omega)}\cdot ( \big\| \|\delta Y\|_{q\text{-}\mathrm{var}} \big\|^{\gamma - 1}_{\infty;1,\infty;[s,t]} + \big\| \|\delta \bar Y\|_{q\text{-}\mathrm{var}} \big\|^{\gamma - 1}_{\infty;1,\infty;[s,t]} ),\\ \displaystyle
L_2 := \big\| \|\delta (Y - \bar Y)\|_{q\text{-}\mathrm{var}} \big\|_{\infty;1,\infty;[s,t]} \cdot (\big\| \| \delta Y\|_{q\text{-}\mathrm{var}} \big\|_{\infty;1,\infty;[s,t]} + \big\| \|\delta \bar Y \|_{q\text{-}\mathrm{var}} \big\|_{\infty;1,\infty;[s,t]}),\\ \displaystyle
L_3 := \big\| \| P^Y - P^{\bar Y}\|_{\frac{q''}{2}\text{-}\mathrm{var}} \big\|_{\infty;k,\infty} + \|Y - \bar Y\|_{\mathcal L^{\infty}([s,t] \times \Omega)} \cdot \big\| \|P^{\bar Y}\|_{\frac{q''}{2}\text{-}\mathrm{var}} \big\|_{\infty;k,\infty},\\ \displaystyle
L_4 := \Big(\|Y - \bar Y\|_{\mathcal L^{\infty}([s,t]\times\Omega)} \big( \big\| \|\delta Y\|_{q\text{-}\mathrm{var}} \big\|^{\gamma - 2}_{\infty;1,\infty;[s,t]} + \big\| \|\delta \bar Y\|_{q\text{-}\mathrm{var}} \big\|^{\gamma - 2}_{\infty;1,\infty;[s,t]}\big) \\ \displaystyle
\quad + \big\| \|\delta (Y - \bar Y)\|_{q\text{-}\mathrm{var}} \big\|_{\infty;1,\infty;[s,t]} \Big) \|\delta M^{\bar Y}\|_{\mathrm{BMO};[s,t]}  + \big\| \|\delta Y\|_{q\text{-}\mathrm{var}} \big\|_{\infty;1,\infty;[s,t]} \|\delta (M^Y - M^{\bar Y})\|_{\mathrm{BMO};[s,t]}.
\end{cases}
\end{equation*}
\end{proposition}

\begin{remark}
The above proposition shows that to bound the $(k,\infty)$ norm of the \emph{RSM} remainder after composition, it suffices to calculate $(1,\infty)$ norms of the \emph{RSM}, $(k,\infty)$ norms of its remainder, and \emph{BMO} norms of its martingale part. Thanks to this proposition, the problem of degenerating integrability, as shown in \cite{fhl21}, is overcome. 
\end{remark}

\begin{proof}
For $(s,t) \in \Delta$, we have  
\begin{equation}\label{e:R^H(Y) difference}
\begin{aligned}
&(P^{G(Y)} - P^{G(\bar Y)})_{s,t} \\
& = \delta G(Y)_{s,t} - DG(Y_{s})Y'_{s} \delta \mathrm X_{s,t} - \delta M^{G(Y)}_{s,t} - \big( \delta G(\bar Y)_{s,t}- DG(\bar Y_{s})\bar Y'_{s} \delta \bar{\mathrm X}_{s,t} - \delta M^{G(\bar Y)}_{s,t}\big).
\end{aligned}
\end{equation}
Define 
$
\hat Y_t := Y_t - \bar Y_t$, $\hat M := M^Y - M^{\bar Y}$, and $ \hat P_{s,t} := P^Y_{s,t} - P^{\bar Y}_{s,t}
$.
Then, plugging  $Y'_{s}\delta \mathrm X_{s,t} = \delta Y_{s,t} - P^Y_{s,t} - \delta M^{Y}_{s,t}$, $\bar Y'_{s} \delta \bar{\mathrm X}_{s,t} = \delta \bar Y_{s,t} - P^{\bar Y}_{s,t} - \delta M^{\bar Y}_{s,t}$, $\delta M_{s,t}^{G(Y)}=\int_s^t DG(Y_r) dM_r^Y$, and $\delta M_{s,t}^{G(\bar Y)}=\int_s^t DG(\bar Y_r) dM_r^{\bar Y}$ into \eqref{e:R^H(Y) difference}, we obtain 
\begin{align}\nonumber
&(P^{G(Y)} - P^{G(\bar Y)})_{s,t}\\ \nonumber
&= \int_{0}^{1}[DG(Y_{s} + \lambda \cdot \delta Y_{s,t}) - DG(Y_s)]d\lambda \cdot \delta Y_{s,t} \\ \nonumber
&\quad - \int_{0}^{1}[DG(\bar Y_{s} + \lambda \cdot \delta \bar Y_{s,t}) - DG(\bar Y_s)]d\lambda \cdot \delta \bar Y_{s,t} + DG(Y_s) P^Y_{s,t} - DG(\bar Y_s) P^{\bar Y}_{s,t} \\ \nonumber
&\quad + DG(Y_s)\delta M^Y_{s,t} - DG(\bar Y_s)\delta M^{\bar Y}_{s,t} - (\delta M^{G(Y)}_{s,t} - \delta M^{G(\bar Y)}_{s,t}) \\ \label{e:identity E.R - R}
& = \int_{0}^{1} [ DG(Y_{s} + \lambda \cdot \delta Y_{s,t}) - DG(\bar Y_{s} + \lambda \cdot \delta \bar Y_{s,t}) - DG(Y_{s})  + DG(\bar Y_{s})] d\lambda \cdot \delta Y_{s,t}\\ \nonumber
&\quad + \int_{0}^{1} [  DG(\bar Y_{s} + \lambda \cdot \delta \bar Y_{s,t}) - DG(\bar Y_{s})] d\lambda \cdot \delta \hat Y_{s,t} + DG(Y_s) \hat P_{s,t} + \left(DG(Y_{s}) - DG(\bar Y_s)\right) P^{\bar Y}_{s,t}\\ \nonumber
&\quad - \int_{s}^{t} \delta (DG(Y) - DG(\bar Y))_{s,r} d M^{\bar Y}_r - \int_{s}^{t} \delta DG(Y)_{s,r} d \hat M_r. 
\end{align}
For the first term on the RHS of the above equation, we have, 
\begin{align} \nonumber
&\Big| \int_{0}^{1} [ DG(\bar Y_{s} + \lambda \delta \bar Y_{s,t}) - DG(Y_{s} + \lambda \delta Y_{s,t}) - DG(\bar Y_{s}) + DG(Y_{s}) ] d\lambda \Big| \\  \nonumber
&\le \Big| \int_{0}^{1} \int_{0}^{1} [ D^2 G(\bar Y_s + \lambda \cdot\delta \bar Y_{s,t} + \tau (\hat Y_{s} + \lambda \cdot \delta \hat Y_{s,t}) ) - D^2 G(\bar Y_s + \tau \hat Y_s) ]d\tau \cdot (\hat Y_s + \lambda \cdot \delta \hat Y_{s,t}) d\lambda\Big| \\  \nonumber
&\quad + \Big| \int_{0}^{1}\int_{0}^{1} [D^{2} G(\bar Y_{s} + \tau \hat Y_{s}) d\tau \cdot (\hat Y_{s} + \lambda\cdot \delta \hat Y_{s,t} - \hat Y_s) d\lambda \Big|\\  \label{e:identity DH - DH - DH + DH}
& \le \|G\|_{C^{\gamma}_b}( |\delta \hat Y_{s,t}|^{\gamma - 2} + |\delta \bar Y_{s,t}|^{\gamma - 2}) \|\hat Y\|_{\infty} + \|G\|_{C^{2}_{b}} |\delta \hat Y_{s,t}| ,
\end{align}
and for the second one, 
\begin{equation}\label{e:DH - DH}
\Big| \int_{0}^{1} [ DG(\bar Y_{s} + \lambda \delta \bar Y_{s,t}) - DG(\bar Y_{s})]d\lambda \Big| \le \|G\|_{C^{2}_{b}} |\delta \bar Y_{s,t}|.
\end{equation}
For the $d M^{\bar Y}$ and $d \hat M$ terms, denote 
\begin{equation*}
\Xi_{s,t} := \int_{s}^{t} \delta DG(Y)_{s,r} d \hat M_r + \int_{s}^{t} \delta (DG(Y) - DG(\bar Y))_{s,r} d M^{\bar Y}_r.
\end{equation*}
By Lemma~\ref{lem:gamma - 2} we get (noting $q \le q' (\gamma - 2)$) 
\begin{equation*}
\begin{aligned}
&\| \delta (DG(Y) - DG(\bar Y)) \|_{q'\text{-}\mathrm{var};[s,t]} \le \|D^2 G\|_{C^{\gamma - 2}_b}\\
&\quad \times \Big( \|\hat Y\|_{\mathcal L^{\infty}([s,t]\times\Omega)} \big( \big\| \|\delta Y\|_{q\text{-}\mathrm{var}} \big\|^{\gamma - 2}_{\infty;1,\infty;[s,t]} + \big\| \|\delta \bar Y\|_{q\text{-}\mathrm{var}} \big\|^{\gamma - 2}_{\infty;1,\infty;[s,t]}\big) + \big\| \|\delta \hat Y\|_{q\text{-}\mathrm{var}} \big\|_{\infty;1,\infty;[s,t]} \Big). 
\end{aligned}
\end{equation*}
Hence, by Lemma~\ref{lem:p-var martingale estim} and~\ref{lem:BMO of q-var}, we have 
\begin{align}\nonumber
&\big\| \| \Xi \|_{\frac{q''}{2}\text{-}\mathrm{var};[s,t]} \big| \mathcal F_s \big\|_{k,\infty} \\ \nonumber
&\lesssim_{k} \big\| \| \delta (DG(Y) - DG(\bar Y)) \|_{q'\text{-}\mathrm{var};[s,t]} \big| \mathcal F_s \big\|_{2k,\infty} \|\delta M^{\bar Y}\|_{\mathrm{BMO};[s,t]} + \big\| \| \delta DG(Y) \|_{q\text{-}\mathrm{var};[s,t]} \big| \mathcal F_s \big\|_{2k,\infty} \\ \nonumber
&\quad \times \|\delta \hat M\|_{\mathrm{BMO};[s,t]}\\ \label{e:dM and d hat M terms}
&\lesssim_{k} \big\| \| \delta (DG(Y) - DG(\bar Y)) \|_{q'\text{-}\mathrm{var}} \big\|_{\infty;1,\infty;[s,t]} \|\delta M^{\bar Y}\|_{\mathrm{BMO};[s,t]} + \big\| \| \delta DG(Y) \|_{q\text{-}\mathrm{var}} \big\|_{\infty;1,\infty;[s,t]}  \\ \nonumber
&\quad \times \|\delta \hat M\|_{\mathrm{BMO};[s,t]}\\ \nonumber
&\le \|G\|_{C^{\gamma}_b} \Big(\|\hat Y\|_{\mathcal L^{\infty}([s,t]\times\Omega)} \big( \big\| \|\delta Y\|_{q\text{-}\mathrm{var}} \big\|^{\gamma - 2}_{\infty;1,\infty;[s,t]} + \big\| \|\delta \bar Y\|_{q\text{-}\mathrm{var}} \big\|^{\gamma - 2}_{\infty;1,\infty;[s,t]}\big) + \big\| \|\delta \hat Y\|_{q\text{-}\mathrm{var}} \big\|_{\infty;1,\infty;[s,t]}\Big) \\ \nonumber
&\quad \times \|\delta M^{\bar Y}\|_{\mathrm{BMO};[s,t]} + \big\| \|\delta Y\|_{q\text{-}\mathrm{var}} \big\|_{\infty;1,\infty;[s,t]} \|\delta \hat M\|_{\mathrm{BMO};[s,t]} . 
\end{align}

Combining \eqref{e:identity E.R - R}, \eqref{e:identity DH - DH - DH + DH}, \eqref{e:DH - DH},  and \eqref{e:dM and d hat M terms}, by taking conditional expectation at time $s$ and using Lemma~\ref{lem:BMO of q-var} again, we have
\begin{equation*}
\begin{aligned}
&\big\| \| P^{G(Y)} - P^{G(\bar Y)} \|_{\frac{q''}{2}\text{-}\mathrm{var}} \big| \mathcal F_s \big\|_{\infty;k,\infty;[s,t]} \\
&\lesssim_k \|G\|_{C^{\gamma}_b}\bigg(  \|\hat Y\|_{\mathcal L^{\infty}([s,t] \times \Omega)} (\big\| \|\delta Y\|_{q\text{-}\mathrm{var}} \big\|^{\gamma - 1}_{\infty;1,\infty;[s,t]} + \big\| \|\delta \bar Y\|_{q\text{-}\mathrm{var}} \big\|^{\gamma - 1}_{\infty;1,\infty;[s,t]}) \\ 
&\quad + \big\| \|\delta \hat Y\|_{q\text{-}\mathrm{var}} \big\|_{\infty;1,\infty;[s,t]} \big( \big\| \|\delta Y\|_{q\text{-}\mathrm{var}} \big\|_{\infty;1,\infty;[s,t]} + \big\| \|\delta \bar Y\|_{q\text{-}\mathrm{var}} \big\|_{\infty;1,\infty;[s,t]} \big) \\
&\quad + \big\| \|\hat P\|_{\frac{q''}{2}\text{-}\mathrm{var}} \big\|_{\infty;k,\infty;[s,t]} + \|\hat Y\|_{\mathcal L^{\infty}([s,t]\times\Omega)} \big\| \| P^{\bar Y} \|_{\frac{q''}{2}\text{-}\mathrm{var}} \big\|_{\infty;k,\infty;[s,t]} \\
&\quad + \Big( \|\hat Y\|_{\mathcal L^{\infty}([s,t]\times\Omega)} \big( \big\| \|\delta Y\|_{q\text{-}\mathrm{var}} \big\|^{\gamma - 2}_{\infty;1,\infty;[s,t]} + \big\| \|\delta \bar Y\|_{q\text{-}\mathrm{var}} \big\|^{\gamma - 2}_{\infty;1,\infty;[s,t]} \big) + \big\| \|\delta \hat Y\|_{q\text{-}\mathrm{var}} \big\|_{\infty;1,\infty;[s,t]} \Big)  \\
&\quad \times \|\delta M^{\bar Y}\|_{\mathrm{BMO};[s,t]} + \big\| \|\delta Y\|_{q\text{-}\mathrm{var}} \big\|_{\infty;1,\infty;[s,t]} \|\delta \hat M\|_{\mathrm{BMO};[s,t]} \bigg),
\end{aligned}
\end{equation*}
which is exactly the desired estimate \eqref{e:E. R - R}. 
\end{proof}

The next proposition provides the estimate for the difference of two rough stochastic integrals.  
 
\begin{proposition}\label{prop:estim of integral'}
Assume that \ref{(p,p')} and \ref{(p,q,q')} hold, and that \eqref{e:Y in RSM} holds with $V_1$ replaced by $\mathscr L(V,V_1)$. Denote rough stochastic integrals by 
\begin{equation*}
I^{\mathbf X}_t := \int_{0}^{t}Y_{r} d\mathbf X_{r} \quad \text{ and  } \quad I^{\bar{\mathbf X}}_{t} := \int_{0}^{t}\bar Y_{r}d\bar{\mathbf X}_r,  \quad \text{ for } t \in [0,T].
\end{equation*}
Then, for $k \ge 1$, it holds that for all $(s,t) \in \Delta$, 
\begin{align}\nonumber
&\big\| \| R^{I^{\mathbf X}} - R^{I^{\bar{ \mathbf X}}} \|_{\frac{p}{2}\text{-}\mathrm{var}} \big\|_{\infty;k,\infty;[s,t]}\\ \nonumber 
&\lesssim \big\| \| P^Y - P^{\bar Y}\|_{\frac{q''}{2}\text{-}\mathrm{var}} \big\|_{\infty;k,\infty;[s,t]} \|\delta \bar{\mathrm X} \|_{p\text{-}\mathrm{var};[s,t]} + \big\| \|P^Y\|_{\frac{q''}{2}\text{-}\mathrm{var}} \big\|_{\infty;k,\infty;[s,t]} \|\delta( \mathrm X - \bar{\mathrm X} )\|_{p\text{-}\mathrm{var};[s,t]} \\ \label{e:esti for R - R 1}
&\quad + \big( \|Y' - \bar Y'\|_{\infty} + \big\| \|\delta(Y' - \bar Y')\|_{q'\text{-}\mathrm{var}} \big\|_{\infty;k,\infty;[s,t]} \big) \| \bar{\mathbb X} \|_{\frac{p}{2}\text{-}\mathrm{var};[s,t]}
\\ \nonumber 
&\quad + \big(\|Y'\|_{\infty} + \big\| \|\delta Y'\|_{q'\text{-}\mathrm{var}} \big\|_{\infty;k,\infty;[s,t]} \big) \| \mathbb X - \bar{\mathbb X} \|_{\frac{p}{2}\text{-}\mathrm{var};[s,t]} +  \|\delta (M^Y - M^{\bar Y})\|_{\mathrm{BMO};[s,t]} \|\delta \bar{\mathrm X}\|_{p\text{-}\mathrm{var};[s,t]} \\ \nonumber
&\quad + \|\delta M^Y\|_{\mathrm{BMO};[s,t]} \|\delta (\mathrm X - \bar{\mathrm X})\|_{p\text{-}\mathrm{var};[s,t]}. 
\end{align}
\end{proposition}

\begin{proof} 
Note that for $(s,t) \in \Delta$,
\begin{equation}\label{e:R-R = Xi + Y'X}
(R^{I^{\mathbf X}} - R^{I^{\bar{\mathbf X}}})_{s,t} = \delta(I^{\mathbf X} - I^{\bar{\mathbf X}})_{s,t} - Y_{s}\delta \mathrm X_{s,t} + \bar Y_{s}\delta \bar{\mathrm X}_{s,t} = A^1_{s,t} + A^{2}_{s,t} + Y'_{s} \mathbb X_{s,t} - \bar Y'_{s} \bar{\mathbb X}_{s,t}, 
\end{equation}
where
$
A^1_{s,t} := \int_{s}^{t} (Y - M^Y)_r d\mathbf X_r - \int_{s}^{t} (\bar Y - M^{\bar Y})_r d\bar{\mathbf X}_r - \big( (Y - M^Y)_{s}\delta \mathrm X_{s,t} - (\bar Y - M^{\bar Y})_{s}\delta \bar{\mathrm X}_{s,t} + Y'_{s} \mathbb X_{s,t} - \bar Y'_{s} \bar{\mathbb X}_{s,t} \big)
$ and $A^2_{s,t} := \int_{s}^{t} \delta M^{Y}_{s,r} d\mathbf X_r - \int_{s}^{t} \delta M^{\bar Y}_{s,r} d\bar{\mathbf X}_r$. 
Define
$
\Xi_{s,t} := (Y - M^Y)_{s}\delta\mathrm X_{s,t} - (\bar Y - M^{\bar Y})_{s}\delta \bar{\mathrm X}_{s,t} + Y'_{s} \mathbb X_{s,t} - \bar Y'_{s} \bar{\mathbb X}_{s,t}
$. 
It follows that
$A^1_{s,t} = \int_{s}^{t} (Y - M^Y)_{r} d\mathbf X_r - \int_{s}^{t} (\bar Y - M^{\bar Y})_{r} d\bar{\mathbf X}_r - \Xi_{s,t}$.

Next, we apply the sewing lemma to estimate $A^1_{s,t}$. For $s\le u\le t$, note that 
\begin{equation*}
\begin{aligned}
\delta \Xi_{s,u,t} &= - P^Y_{s,u} \delta\mathrm X_{u,t} - \delta Y'_{s,u}\mathbb X_{u,t} + P^{\bar Y}_{s,u}  \delta\bar{\mathrm X}_{u,t} + \delta \bar Y'_{s,u}\bar{\mathbb X}_{u,t}\\
& = -(P^Y - P^{\bar Y})_{s,u} \delta \bar{\mathrm X}_{u,t} - P^Y_{s,u} \delta(\mathrm X -  \bar{\mathrm X})_{u,t} - \delta (Y' - \bar Y')_{s,u} \bar{\mathbb X}_{u,t} - \delta Y'_{s,u} \big(\mathbb X_{u,t} - \bar{\mathbb X}_{u,t} \big). 
\end{aligned}
\end{equation*}
Hence, we obtain 
\begin{equation}\label{e:For sto sewing 1}
\begin{aligned}
\big| \delta \Xi_{s,u,t} \big| 
&\le \|P^Y - P^{\bar Y}\|_{\frac{q''}{2}\text{-}\mathrm{var};[s,t]}
\|\delta \bar {\mathrm X} \|_{p\text{-}\mathrm{var};[s,t]} + \|P^Y\|_{\frac{q''}{2}\text{-}\mathrm{var};[s,t]}
\|\delta( \mathrm X -\bar{\mathrm X} )\|_{p\text{-}\mathrm{var};[s,t]} \\ 
& \quad + \|\delta(Y' - \bar Y')\|_{q'\text{-}\mathrm{var};[s,t]} \|\bar{\mathbb X}\|_{\frac{p}{2}\text{-}\mathrm{var};[s,t]} + \|\delta Y'\|_{q'\text{-}\mathrm{var};[s,t]} \|\mathbb X - \bar{\mathbb X}\|_{\frac{p}{2}\text{-}\mathrm{var};[s,t]}.
\end{aligned}
\end{equation}
Note $A^1_{s,t} = \int_{s}^{t} (Y - M^Y)_{r }d\mathbf X_{r} - \int_{s}^{t} (\bar Y - M^{\bar Y})_{r} d\bar{\mathbf X}_{r} - \Xi_{s,t}$. Since \ref{(p,q,q')} holds, by the sewing lemma (see, e.g., \cite[Theorem~2.5]{FrizZhang-2018}), \eqref{e:For sto sewing 1} implies that a.s.   
\begin{equation}\label{e:RHS RHS}
\big| A^1_{s,t} \big| \lesssim \text{ RHS of \eqref{e:For sto sewing 1}}.
\end{equation}
Moreover, noting that $A^2_{s,t} = \int_s^t \delta \bar{\mathrm X}_{s,r} d(M^{Y} - M^{\bar Y})_r + \int_s^t \delta (\mathrm X - \bar{\mathrm X})_{s,r} dM^{Y}_r$ from integration by parts, by L\'epingle's inequality for martingale transforms, we have 
\begin{equation}\label{e:esti for A^2}
\begin{aligned}
&\big\| \|A^2\|_{\frac{p}{2}\text{-}\mathrm{var};[s,t]} \big | \mathcal F_s \big\|_{k,\infty} \\
&\lesssim_k \|\delta (M^Y - M^{\bar Y})\|_{\mathrm{BMO};[s,t]} \|\delta \bar{\mathrm X}\|_{p\text{-}\mathrm{var};[s,t]} + \|\delta M^Y\|_{\mathrm{BMO};[s,t]} \|\delta (\mathrm X - \bar{\mathrm X})\|_{p\text{-}\mathrm{var};[s,t]}. 
\end{aligned}
\end{equation}

In addition, we have
\begin{equation}\label{e:esti for Y'X - Y'X}
\begin{aligned}
|Y'_{s}\mathbb X_{s,t} - \bar Y'_{s}\bar{\mathbb X}_{s,t}| &\le | Y'_{s} \bar{\mathbb X}_{s,t} - \bar Y'_{s} \bar{\mathbb X}_{s,t} | + | Y'_{s}\mathbb X_{s,t}  - Y'_{s} \bar{\mathbb X}_{s,t} | \\
&\le 2 \big( \|Y' - \bar Y'\|_{\infty} \| \bar{\mathbb X} \|_{\frac{p}{2}\text{-}\mathrm{var};[s,t]} + \|Y'\|_{\infty} \| \mathbb X - \bar{\mathbb X} \|_{\frac{p}{2}\text{-}\mathrm{var};[s,t]} \big). 
\end{aligned}
\end{equation}

Finally, by \eqref{e:R-R = Xi + Y'X} we have $(R^{I^{\mathbf X}} - R^{I^{\bar{ \mathbf X}}})_{s,t} = A^1_{s,t} + A^2_{s,t} + Y'_{s}\mathbb X_{s,t} - \bar Y'_{s}\bar{\mathbb X}_{s,t}$. Together with \eqref{e:RHS RHS}, \eqref{e:esti for A^2}, and \eqref{e:esti for Y'X - Y'X}, this implies
\eqref{e:esti for R - R 1}.  
\end{proof}

%------------------

\subsection{Seminorms}\label{subsec:seminorms}

We now introduce some seminorms that will be used  for the fixed point argument in the next section.
As discussed in the introduction, one key difficulty in the analysis of rough BSDEs, compared with rough SDEs, lies in establishing the regularity of $Z$. To address this issue, we introduce the \textrm{RSM} seminorm $\thicknm{\bm \cdot}_{\mathrm X;q,q';k,\infty}$. This quantity arises naturally in rough stochastic integral estimates (see Remark~\ref{rem:for Psi}) and in the proof of the comparison principle (Theorem~\ref{thm:comparison}). 

For $(p',q,q')$ satisfying \eqref{cond31}, $\mathbf X = (\mathrm X,\mathbb X) \in \mathscr{C}^{p\text{-}\mathrm{var}}([0,T],\R^d)$, $(Y,Y') \in \mathrm{RSM}^{(p,p')\text{-}\mathrm{var}}_{\mathrm X}([0,T],\R)$, and $k \ge 1$, define the seminorm (recall $q''/2 := qq'/(q+q')$)  
\begin{equation}\label{e:Psi}
\thicknm{(Y,Y')}_{\mathrm X;q,q';k,\infty} :=  \left\|\|\delta Y'\|_{q'\text{-}\mathrm{var}}\right\|_{\infty;k,\infty} + \|Y'\|_{\infty} + \big\| \|P^{Y}\|_{\frac{q''}{2}\text{-}\mathrm{var}} \big\|_{\infty;k,\infty} + \|\delta M^{Y}\|_{\mathrm{BMO}}.
\end{equation}
Here, recall  $(\|\bm\cdot\|_{q\text{-}\mathrm{var}})_{s,t} := \|\bm\cdot\|_{q\text{-}\mathrm{var};[s,t]}: \Delta_{[0,T]}\to \R$ is a two-parameter function.
For an arbitrary interval $I\subset [0,T]$, the notation $\thicknm{(Y,Y')}_{\mathrm X;q,q';k,\infty;I}$ is defined analogously by restricting the underlying objects to $I$ with, for example, $\|Y'\|_{\infty}$ replaced by $\|Y'\|_{\mathcal L^{\infty}(I \times \Omega)}$.

For another $\bar{\mathbf X} = (\bar{\mathrm X},\bar{\mathbb X}) \in \mathscr C^{p\text{-}\mathrm{var}}([s,t],\R^d)$ and $(\bar Y,\bar Y') \in \mathrm{RSM}^{(p,p')\text{-}\mathrm{var}}_{\bar{\mathrm X}}([s,t],\R)$, define 
\begin{equation*}
\begin{aligned}
&\thicknm{Y,Y';\bar Y,\bar Y'}_{\mathrm X,\bar{\mathrm X};q,q';k,\infty} \\
&:=  \big\| \|\delta (Y' - \bar Y')\|_{q'\text{-}\mathrm{var}}\big\|_{\infty;k,\infty} + \|Y' - \bar Y'\|_{\infty} + \big\| \|P^{Y} - P^{\bar Y}\|_{\frac{q''}{2}\text{-}\mathrm{var}} \big\|_{\infty;k,\infty} + \|\delta (M^{Y} - M^{\bar Y})\|_{\mathrm{BMO}}.
\end{aligned}
\end{equation*}
The notation $\thicknm{Y,Y';\bar Y,\bar Y'}_{\mathrm X,\bar{\mathrm X};q,q';k,\infty;I}$ is defined by restricting the underlying objects to $I$.

\begin{lemma}\label{lem:Y <= R^Y'}
Assume \eqref{cond31} holds and $p \ge q''/2$. Let $(s,t) \in \Delta$, and suppose $\mathbf X \in \mathscr C^{p\text{-}\mathrm{var}}([s,t],\R^d)$ and $(Y,Y') \in \mathrm{RSM}^{(p,p')\text{-}\mathrm{var}}_{\mathrm X}([s,t],\R)$. Then, for any $k \ge 1$, it holds 
\begin{equation}\label{e:Y q-var infty}
\big\|\|\delta Y \|_{p\text{-}\mathrm{var}} \big\|_{\infty;k,\infty;[s,t]} \lesssim_k (1 + \|\delta \mathrm X\|_{p\text{-}\mathrm{var};[s,t]}) \, \thicknm{(Y,Y')}_{\mathrm X;q,q';k,\infty;[s,t]},
\end{equation}
\begin{equation}\label{e:Y q-var infty infty}
\big\|\|\delta Y \|_{p\text{-}\mathrm{var}} \big\|_{\infty;k,\infty;[s,t]} + \left\|Y\right\|_{\mathcal L^{\infty}([s,t]\times\Omega)}  \lesssim_k \|Y_{t}\|_{\infty} + (1 + \|\delta \mathrm X\|_{p\text{-}\mathrm{var};[s,t]})\,  \thicknm{(Y,Y')}_{\mathrm X;q,q';k,\infty;[s,t]}.
\end{equation}
Moreover, for another $\bar{\mathbf X} \in \mathscr C^{p\text{-}\mathrm{var}}([s,t],\R^d)$ and $(\bar Y,\bar Y') \in \mathrm{RSM}^{(p,p')\text{-}\mathrm{var}}_{\bar{\mathrm X}}([s,t],\R)$, it holds 
\begin{equation}\label{e:Y - Y q-var infty infty}
\begin{aligned}
&\big\|\|\delta (Y - \bar Y) \|_{p\text{-}\mathrm{var}} \big\|_{\infty;k,\infty;[s,t]} + \|Y - \bar Y\|_{\mathcal L^{\infty}([s,t]\times\Omega)} \lesssim_k \|(Y -\bar Y)_{t}\|_{\infty}\\
&\quad  + (1 + \|\delta \mathrm X\|_{p\text{-}\mathrm{var};[s,t]})\,  \thicknm{Y,Y';\bar Y,\bar Y'}_{\mathrm X,\bar{\mathrm X};q,q';k,\infty;[s,t]} + \|\bar Y'\|_{\mathcal L^{\infty}([s,t] \times \Omega)} \|\delta (\mathrm X - \bar{\mathrm X})\|_{p\text{-}\mathrm{var};[s,t]}. 
\end{aligned}
\end{equation}
\end{lemma}

\begin{proof}
By the equation $\delta Y = Y' \delta \mathrm X + P^{Y} + \delta M^{Y}$, for $(u,v) \in \Delta_{[s,t]}$ we have
\begin{equation}\label{e:bound for Y p-var'}
\big\|\|\delta Y\|_{p\text{-}\mathrm{var};[u,v]}|\mathcal F_{u}\big\|_{k}\le \|Y'\|_{\mathcal L^{\infty}([u,v] \times \Omega)}\|\delta \mathrm X\|_{p\text{-}\mathrm{var};[u,v]} + \big\|\| P^Y \|_{p\text{-}\mathrm{var};[u,v]}|\mathcal F_{u}\big\|_{k} + \big\|\|\delta M^{Y} \|_{p\text{-}\mathrm{var};[u,v]}|\mathcal F_{u}\big\|_{k}.
\end{equation} 
For the RHS of \eqref{e:bound for Y p-var'}, noting $p \ge q''/2$ so that $\| P^Y \|_{p\text{-}\mathrm{var};[u,v]} \le \|P^Y\|_{q''/2\text{-}\mathrm{var};[u,v]}$; and by L\'epingle's inequality (Lemma~\ref{lem:Lepingle}) we get $\|\|\delta M^{Y} \|_{p\text{-}\mathrm{var};[u,v]}|\mathcal F_{u}\|_{k} \lesssim_k \|\delta M^Y\|_{\mathrm{BMO};[s,t]}$. Hence, \eqref{e:Y q-var infty} holds.
Furthermore, for $u \in [s,t]$ noting that $\|Y_u\|_{\infty} \le \|Y_{t}\|_{\infty} + \|\delta Y_{u,t} | \mathcal F_u\|_{k,\infty}$ from the inequality $|Y_u| = |\E_u[Y_u]| \le |Y_t| + |\E_{u}[\delta Y_{u,t}]|$, \eqref{e:Y q-var infty} yields \eqref{e:Y q-var infty infty}. \eqref{e:Y - Y q-var infty infty} can be proved similarly. 
\end{proof}

The following Lemma~\ref{lem:int of RSM'}, in contrast to Proposition~\ref{prop:G(Y) in D} where $G\in C^{2}_b$ is assumed, concerns a weaker condition $G\in C^{\gamma - 1}_b$, and gives estimates on the composition of $G$ and an RSM. Since $G$ may not be twice differentiable, the interpolation argument cannot be applied as we did in Proposition~\ref{prop:G(Y) in D}. However, the lower regularity condition enables us to prove the comparison principle in Theorem~\ref{thm:comparison}  under the condition $H\in C^\gamma_b$, in which case $G = DH$.  

\begin{lemma}\label{lem:int of RSM'}
Assume \eqref{cond31} holds with $p' \le q'$. Let $\mathbf X \in \mathscr C^{p\text{-}\mathrm{var}}([0,T],\R^d)$, and let $\gamma$ satisfy $\gamma - 2 \ge p/p'$, and $G \in C^{\gamma - 1}_b$. Assume $(Y,Y'),(\bar Y,\bar Y') \in \mathrm{RSM}^{(p,p')\text{-}\mathrm{var}}_{\mathrm X}([0,T],\R)$. Let $(\alpha,\alpha')$ be defined in \eqref{e:def of alpha}. Denote $C_G := \|G\|_{C^{\gamma - 1}_b}$. 
Then, for every $k \ge 1$,   
\begin{equation}\label{e:Psi alpha}
\begin{aligned}
\thicknm{(\alpha,\alpha') }_{\mathrm X;q,q' ; k , \infty} \lesssim_{k,\gamma} C_G \big(1 + |\delta \mathbf X|^{\gamma - 1}_{p\text{-}\mathrm{var};[0,T]}\big) \big(1 + \thicknm{(Y,Y')}_{\mathrm X;q,q';k,\infty}^{\gamma - 1} + \thicknm{(\bar Y,\bar Y')}_{\mathrm X;q,q';k,\infty}^{\gamma - 1}\big).
\end{aligned}
\end{equation}
Moreover, assume further that $(q,q')$ satisfies \ref{(p,q,q')}. Then, it follows that for all $(s,t)\in \Delta$, 
\begin{equation}\label{e:Psi alpha'}
\begin{aligned}
&\Big\| \Big\|\delta \Big( \int_{s}^{\bm\cdot}\alpha_{r}d\mathbf X_r \Big) \Big\|_{p\text{-}\mathrm{var};[s,t]} \Big| \mathcal F_s \Big\|_{k,\infty} \\
&\lesssim_{k,\gamma} C_G |\delta \mathbf X|_{p\text{-}\mathrm{var};[s,t]} \big(1 + |\delta \mathbf X|^{\gamma - 1}_{p\text{-}\mathrm{var};[0,T]} \big) \big(1 + \thicknm{(Y,Y')}_{\mathrm X;q,q';k,\infty}^{\gamma - 1} + \thicknm{(\bar Y,\bar Y')}_{\mathrm X;q,q';k,\infty}^{\gamma - 1} \big).
\end{aligned}
\end{equation}
\end{lemma}

\begin{proof}
Without loss of generality, we assume $\gamma \le 3$. By Lemma~\ref{lem:int of RSM} we have $(\alpha,\alpha') \in \mathrm{RSM}^{(p,p')\text{-}\mathrm{var}}_{\mathrm X}$. We now prove \eqref{e:Psi alpha}. 

Let $Y^{\lambda}_t := (1-\lambda)Y_t + \lambda \bar Y_t$, $(Y^{\lambda})'_t := (1-\lambda)Y'_t + \lambda \bar Y'_t$, and denote $G^{\lambda}, (G^{\lambda})'$ by \eqref{e:def of G^lambda}. By \eqref{e:def of alpha} in Lemma~\ref{lem:int of RSM}, we have $\|\alpha'\|_{\infty} \le C_G (\|Y'\|_{\infty}+ \|\bar Y'\|_{\infty})$. Furthermore for $\|\delta \alpha'\|_{q'\text{-}\mathrm{var}}$, by the proof leading to \eqref{e:G^lambda < Y} with $p'$ replaced by $q'$, we have 
\begin{equation}\label{e:G landa'}
\begin{aligned}
&\|\delta (G^{\lambda})'\|_{q'\text{-}\mathrm{var};[t,T]} \\
&\le C_G \Big( \|\delta Y'\|_{q'\text{-}\mathrm{var};[t,T]} + \|\delta \bar Y'\|_{q'\text{-}\mathrm{var};[t,T]} + \big(\|\delta Y\|^{\gamma - 2}_{p\text{-}\mathrm{var};[t,T]} + \|\delta \bar Y\|^{\gamma - 2}_{p\text{-}\mathrm{var};[t,T]}\big) (\|Y'\|_{\infty} + \|\bar Y'\|_{\infty}) \Big).
\end{aligned}
\end{equation}
Then applying Minkowski's inequality, noting $\alpha'_{\bm\cdot} = \int_{0}^{1} (G^{\lambda}_{\bm\cdot})' d\lambda$, we get 
\begin{equation*}
\begin{aligned}
&\|\delta \alpha'\|_{q'\text{-}\mathrm{var};[t,T]} \le C_G \times  \text{ RHS of \eqref{e:G landa'}}.
\end{aligned}
\end{equation*}
This, together with \eqref{e:Y q-var infty} in  Lemma~\ref{lem:Y <= R^Y'} yields 
\begin{equation}\label{e:alpha'}
\begin{aligned}
&\|\alpha'\|_{\infty} + \big\| \|\delta \alpha'\|_{q'\text{-}\mathrm{var}} \big\|_{\infty;k;\infty} \lesssim_{\gamma} C_G 
(1 + |\delta \mathbf X|^{\gamma - 2}_{p\text{-}\mathrm{var}})
\big( 1 + \thicknm{(Y,Y')}_{\mathrm X;q,q';k,\infty}^{\gamma - 1} + \thicknm{(\bar Y,\bar Y')}_{\mathrm X;q,q';k,\infty}^{\gamma - 1} \big).
\end{aligned}
\end{equation}
Then for $M^{\alpha}$, noting $\|\delta M^{\alpha}\|_{\mathrm{BMO}} \le \int_{0}^{1} \|\delta M^{G^{\lambda}}\|_{\mathrm{BMO}} \le C_G \int_{0}^{1} \|\delta M^{Y^{\lambda}}\|_{\mathrm{BMO}} \, d\lambda$, we have  
\begin{equation}\label{e:M^alpha}
\begin{aligned}
\|\delta M^{\alpha}\|_{\mathrm{BMO}} &\le C_{G} \Big( \|\delta M^{Y}\|_{\mathrm{BMO}} + \|\delta M^{\bar Y}\|_{\mathrm{BMO}} \Big) \le C_G \Big( \thicknm{(Y,Y')}_{\mathrm X;q,q';k,\infty} + \thicknm{(\bar Y,\bar Y' )}_{\mathrm X;q,q';k,\infty} \Big).
\end{aligned}
\end{equation}

Now we estimate $P^{\alpha}$. Denote $P^{\lambda} := P^{G^{\lambda}}$ and define $K^{\lambda;i}, i=1,2,3,4$, by \eqref{e:R^lambda}. For $K^{\lambda;1}$, $K^{\lambda;2}$ and $K^{\lambda;3}$, one can repeat the procedure in Step~1 of Lemma~\ref{lem:int of RSM}, but compute the $\frac{q''}{2}$-variation instead, and obtain (noting $\frac{q''}{2} \ge \frac{pp'}{p+p'} \ge \frac{p}{\gamma - 1}$ due to $q \ge p$ and $q' \ge p'$)
\begin{equation}\label{e:K1+K2+K3}
\begin{aligned}
\sum_{i=1}^{3}\|K^{\lambda;i}\|_{\frac{q''}{2}\text{-}\mathrm{var};[t,T]} &\le C_G \Big( \|\delta Y^{\lambda}\|^{\gamma - 2}_{p\text{-}\mathrm{var};[t,T]} \|\delta M^{Y^{\lambda}}\|_{p\text{-}\mathrm{var};[t,T]} + \|P^{Y^{\lambda}}\|_{\frac{q''}{2}\text{-}\mathrm{var};[t,T]} \\
&\quad + \|(Y^{\lambda})'\|_{\infty;[t,T]} \|\delta Y^{\lambda}\|^{\gamma - 2}_{p\text{-}\mathrm{var};[t,T]} \|\delta \mathrm X\|_{p\text{-}\mathrm{var};[t,T]} \Big).
\end{aligned}
\end{equation}
To estimate $K^{\lambda;4}$, similar to \eqref{e:K lambda 4}, by Lemma~\ref{lem:p-var martingale estim} with $r = \frac{q''}{2}$, $r_1 = \frac{p}{\gamma - 2}$, $k_1 = k_0 = 2k$, 
\begin{equation}\label{e:K lambda 4'}
\begin{aligned}
\big\| \| K^{\lambda;4} \|_{\frac{q''}{2}\text{-}\mathrm{var};[t,T]}  \big| \mathcal F_{t}\big\|_{k} & \lesssim_{k,\gamma} \big\| \|\delta DG(Y^{\lambda}_{\bm\cdot})\|_{\frac{p}{\gamma - 2}\text{-}\mathrm{var};[t,T]} \big| \mathcal F_{t} \big\|_{k_1} \big\| \delta \langle M^{Y^{\lambda}} \rangle^{\frac{1}{2}}_{t,T} \big| \mathcal F_{t} \big\|_{k_0} \\
&\le C_G \big\| \|\delta Y^{\lambda}\|_{p\text{-}\mathrm{var};[t,T]} \big| \mathcal F_t \big\|^{\gamma - 2}_{k} \big\| \delta \langle M^{Y^{\lambda}} \rangle^{\frac{1}{2}}_{t,T} \big| \mathcal F_{t} \big\|_{k_0}. 
\end{aligned}
\end{equation}  

The $k$-th conditional moment of the first term on the RHS of \eqref{e:K1+K2+K3} is bounded via H\"older's inequality, noting  $\frac{1}{k_1} + \frac{1}{k_0} = \frac{1}{k}$:
\begin{equation*}
\begin{aligned}
\big\| \|\delta Y^{\lambda}\|^{\gamma - 2}_{p\text{-}\mathrm{var};[t,T]} \|\delta M^{Y^{\lambda}}\|_{p\text{-}\mathrm{var};[t,T]}\big|\mathcal F_t \big\|_k &\le \big\| \|\delta Y^{\lambda}\|^{\gamma - 2}_{p\text{-}\mathrm{var};[t,T]} \big|\mathcal F_t\big\|_{k_1} \big\| \|\delta M^{Y^{\lambda}}\|_{p\text{-}\mathrm{var};[t,T]}\big|\mathcal F_t\big\|_{k_0} \\
&\le \big\| \|\delta Y^{\lambda}\|_{p\text{-}\mathrm{var};[t,T]} \big|\mathcal F_t\big\|^{\gamma - 2}_{k_1} \big\| \|\delta M^{Y^{\lambda}}\|_{p\text{-}\mathrm{var};[t,T]}\big|\mathcal F_t\big\|_{k_0}\\
&\lesssim_k \big\| \|\delta Y^{\lambda}\|_{p\text{-}\mathrm{var}} \big\|^{\gamma - 2}_{\infty;k,\infty} \big\| \|\delta M^{Y^{\lambda}}\|_{p\text{-}\mathrm{var};[t,T]}\big|\mathcal F_t\big\|_{k_0},
\end{aligned}
\end{equation*}
where the last inequality is due to $\| \|\delta Y^{\lambda}\|_{p\text{-}\mathrm{var}} \|_{\infty;k_1,\infty} \lesssim_k \| \|\delta Y^{\lambda}\|_{p\text{-}\mathrm{var}} \|_{\infty;k,\infty}$ from Lemma~\ref{lem:BMO of q-var}.
Thus, combining \eqref{e:K1+K2+K3} and \eqref{e:K lambda 4'}, and applying Lemma~\ref{lem:Lepingle} to $M^{Y^{\lambda}}$, we get 
\begin{equation}\label{e:R^lambda'}
\begin{aligned}
\big\| \|P^{\lambda} & \|_{\frac{q''}{2}\text{-}\mathrm{var};[t,T]}  \big| \mathcal F_{t} \big\|_{k} 
\lesssim_{k,\gamma} C_G \Big( 
\big\| \|\delta Y^{\lambda}\|_{p\text{-}\mathrm{var};[t,T]} \big| \mathcal F_t \big\|^{\gamma - 2}_{k} \big\| \delta \langle M^{Y^{\lambda}}\rangle^{\frac{1}{2}}_{t,T} \big| \mathcal F_t \big\|_{k_0} \\
& + \big\| \|P^{Y^{\lambda}} \|_{\frac{q''}{2}\text{-}\mathrm{var};[t,T]} \big| \mathcal F_t \big\|_{k} + \|(Y^{\lambda})'\|_{\infty} \big\| \|\delta Y^{\lambda}\|_{p\text{-}\mathrm{var};[t,T]} \big| \mathcal F_t \big\|^{\gamma - 2}_{k} \|\delta \mathrm X\|_{p\text{-}\mathrm{var};[t,T]} \Big).
\end{aligned}
\end{equation}

Since $(Y^{\lambda},(Y^{\lambda})',P^{Y^{\lambda}},M^{Y^{\lambda}})$ is a convex combination of $(Y,Y',P^Y,M^{Y})$ and $(\bar Y,\bar Y',P^{\bar Y},M^{\bar Y})$, and since the sum $\| \|\delta Y\|_{p\text{-}\mathrm{var}} \|_{\infty;k,\infty} + \| \|\delta \bar Y\|_{p\text{-}\mathrm{var}} \|_{\infty;k,\infty}$ is dominated by the product of $(1 + |\delta \mathbf X|_{p\text{-}\mathrm{var}})$ and $(\thicknm{(Y,Y')}_{\mathrm X;q,q';k,\infty} + \thicknm{(\bar Y,\bar Y')}_{\mathrm X;q,q';k,\infty})$ (see Lemma~\ref{lem:Y <= R^Y'}), estimate \eqref{e:R^lambda'} implies that 
\begin{equation*}
\big\| \|P^{\lambda}\|_{\frac{q''}{2}\text{-}\mathrm{var}} \big\|_{\infty;k,\infty} \lesssim_{k,\gamma} C_G (1 + |\delta \mathbf X|^{\gamma - 1}_{p\text{-}\mathrm{var};[0,T]}) \big( 1 +  \thicknm{(Y,Y')}_{\mathrm X;q,q';k,\infty}^{\gamma - 1} + \thicknm{(\bar Y,\bar Y')}_{\mathrm X;q,q';k,\infty}^{\gamma - 1} \big).
\end{equation*}
Therefore, applying Minkowski's inequality to $P^{\alpha}_{s,t} = \int_{0}^{1} P^{\lambda}_{s,t} d\lambda$ yields  
\begin{equation*}
\big\| \|P^{\alpha}\|_{\frac{q''}{2}\text{-}\mathrm{var}} \big\|_{\infty;k,\infty} \lesssim_{k,\gamma} C_G (1 + |\delta \mathbf X|^{\gamma - 1}_{p\text{-}\mathrm{var};[0,T]}) \big( 1 +  \thicknm{(Y,Y')}_{\mathrm X;q,q';k,\infty}^{\gamma - 1} + \thicknm{(\bar Y,\bar Y')}_{\mathrm X;q,q';k,\infty}^{\gamma - 1} \big).
\end{equation*}
This together with \eqref{e:alpha'} and \eqref{e:M^alpha} implies the desired estimate \eqref{e:Psi alpha}. 

Finally, assume further \ref{(p,q,q')} (hence \ref{(p,p')} holds from $p' \le q'$) and set $\Gamma_{t} := \int_{0}^{t}\alpha_r d\mathbf X_r$. Noting  $\delta\Gamma_{s,t} = R^{\Gamma}_{s,t} + \alpha'_{s} \delta \mathrm X_{s,t}$, the estimate \eqref{e:Psi alpha'} follows directly from \eqref{e:Psi alpha} and Proposition~\ref{prop:estim of integral}.
\end{proof}

%--------------------------

The next proposition shows that the seminorm $\thicknm{\bm\cdot}_{\mathrm X;q,q';k,\infty}$ over a large interval is dominated by the sum of the corresponding seminorms over small intervals.

\begin{proposition}\label{prop:Psi <= Psi + Psi}
Assume \eqref{cond31} holds. Let $(s,t) \in \Delta$, and suppose $\mathbf X \in \mathscr C^{p\text{-}\mathrm{var}}([s,t],\R^d)$ and $(Y,Y') \in \mathrm{RSM}^{(p,p')\text{-}\mathrm{var}}_{\mathrm X}([s,t],\R)$. Then, we have 
\begin{equation}\label{e:Psi <= Psi + Psi}
\thicknm{(Y,Y')}_{\mathrm X;q,q';k,\infty;[s,t]} \le  (1 + |\delta \mathbf X|_{q\text{-}\mathrm{var};[s,t]})\, \thicknm{(Y,Y')}_{\mathrm X;q,q';k,\infty;[s,u]} + \thicknm{(Y,Y')}_{\mathrm X;q,q';k,\infty;[u,t]}.
\end{equation}
\end{proposition}

\begin{proof}
Using the subadditivity of $q'$-variation, we get that for $s \le a\le u \le b \le t$,
\begin{equation*}
\| \|Y'\|_{q'\text{-}\mathrm{var};[a,b]} | \mathcal F_{a} \|_{k,\infty} \le \| \|Y'\|_{q'\text{-}\mathrm{var};[a,u]} | \mathcal F_{a} \|_{k,\infty} + \| \|Y'\|_{q'\text{-}\mathrm{var};[u,b]} | \mathcal F_{u} \|_{k,\infty},
\end{equation*}
which implies  $$\| \|Y'\|_{q'\text{-}\mathrm{var}} \|_{\infty;k,\infty;[s,t]} \le \| \|Y'\|_{q'\text{-}\mathrm{var}} \|_{\infty;k,\infty;[s,u]} + \| \|Y'\|_{q'\text{-}\mathrm{var}} \|_{\infty;k,\infty;[u,t]}.$$
For the pathwise remainder $P^Y$, by Definition~\ref{def:rough semimartingale} we have that for $u\in[r,v] \subset [a,b]$,
\begin{equation*}
P^Y_{r,v} = P^Y_{r,u} + P^Y_{u, v} + \delta Y'_{r,u} \delta \mathrm X_{u, v},
\end{equation*}
and hence by H\"older's inequality, 
\begin{equation*}
\begin{aligned}
&\big\| \|P^Y\|_{\frac{q''}{2}\text{-}\mathrm{var};[a,b]} \big| \mathcal F_{a} \big\|_{k,\infty} \\
&\le \big\| \|P^Y\|_{\frac{q''}{2}\text{-}\mathrm{var};[a,u]} \big|\mathcal F_a \big\|_{k,\infty} + \big\| \|P^Y\|_{\frac{q''}{2}\text{-}\mathrm{var};[u,b]} \big| \mathcal F_u \big\|_{k,\infty}  + \big\| \| Y'\|_{q'\text{-}\mathrm{var};[a,u]}  | \mathcal F_{a} \big\|_{k,\infty} |\mathbf X|_{q\text{-}\mathrm{var};[s,t]}.
\end{aligned}
\end{equation*}
Therefore, we have
\begin{equation*}
\begin{aligned}
&\big\| \| P^Y\|_{\frac{q''}{2}\text{-}\mathrm{var}} \big\|_{\infty;k,\infty;[s,t]} \\
&\le \big\| \|P^Y\|_{\frac{q''}{2}\text{-}\mathrm{var}} \big\|_{\infty;k,\infty;[s,u]} + \big\| \|P^Y\|_{\frac{q''}{2}\text{-}\mathrm{var}} \big\|_{\infty;k,\infty;[u,t]} + \big\| \|Y'\|_{q'\text{-}\mathrm{var}} \big\|_{\infty;k,\infty;[s,u]} |\mathbf X|_{q\text{-}\mathrm{var};[s,t]}.
\end{aligned}
\end{equation*}
Finally, noting the subadditivity of BMO norms, the desired \eqref{e:Psi <= Psi + Psi} follows.
\end{proof}

\begin{lemma}\label{lem:Banach property of RSMs}
Assume \eqref{cond31}. Let $(s,t) \in \Delta$, and suppose $\mathbf X \in \mathscr C^{p\text{-}\mathrm{var}}([s,t],\R^d)$. Denote 
\begin{equation*}
\|(Y,Y')\|_{\mathrm X;q,q';k,\infty;[s,t]} := \|Y_{t}\|_{\infty} +\, \thicknm{(Y,Y')}_{\mathrm X;q,q';k,\infty;[s,t]}.
\end{equation*}
Then, the space $(\mathrm{RSM}^{(p,p')\text{-}\mathrm{var}}_{\mathrm X}([s,t],\R), \|\bm\cdot\|_{\mathrm X;q,q';k,\infty;[s,t]})$ forms a Banach space. 
\end{lemma}

\begin{proof}
Let $\{(Y^n,(Y^n)')\}_{n \ge 1}$ be a Cauchy sequence in this space. By Lemma~\ref{lem:Banach property of L C} and the Banach property of \textrm{BMO} martingales (see, e.g., \cite[Chapter~2.1]{Kazamaki1994}) there exists $(Y',P,M)$ such that 
\begin{equation}\label{e:conver of for terms}
\begin{aligned}
&\lim_{n\to\infty} \|(Y^n)' - Y'\|_{\mathcal L^{\infty}([s,t] \times \Omega)} + \big\| \|\delta ((Y^n)' - Y')\|_{q'\text{-}\mathrm{var}} \big\|_{\infty;k,\infty;[s,t]} \\
&\quad + \big\| \|P^{Y^n} - P\|_{\frac{q''}{2}\text{-}\mathrm{var}} \big\|_{\infty;k,\infty;[s,t]} + \|\delta(M^{Y^n} - M)\|_{\mathrm{BMO};[s,t]} = 0.
\end{aligned}
\end{equation}

Furthermore, let $Y_r := \lim_{n \to \infty} Y^n_{t} - Y'_r \delta \mathrm X_{r,t} - P_{r,t} - \delta M_{r,t}$, we see that $\delta Y_{u,v} = Y'_{u} \delta \mathrm X_{u,v} + P_{u,v} + \delta M_{u,v}$. Hence, to prove $(Y,Y') \in \mathrm{RSM}^{(p,p')\text{-}\mathrm{var}}_{\mathrm X}$, it suffices to check that a.s., $\|\delta Y\|_{p\text{-}\mathrm{var};[s,t]} + \|\delta Y'\|_{p'\text{-}\mathrm{var};[s,t]} + \|P\|_{pp'/(p+p')\text{-}\mathrm{var};[s,t]} < \infty$. 

Indeed, by the convergence in \eqref{e:conver of for terms}, after passing to a subsequence if necessary, $( (Y^{n})' , P^{Y^n} )$ converges to $(Y',P)$ in the uniform norm a.s. Hence, by the lower semi-continuity of $p'$-variation norm (and, respectively, the $pp'/(p+p')$-variation norm) established in \cite[Lemma~5.12]{friz2010multidimensional}, we see that a.s. $Y'$ and $P$ have finite $p'$-variation and $pp'/(p+p')$-variation, respectively. Furthermore, by Lemma~\ref{lem:Lepingle}, after passing to a further subsequence, $M^{Y^n}$ converges to $M$ in the uniform norm a.s. By the equation $\delta Y_{u,v} = Y'_{u} \delta \mathrm X_{u,v} + P_{u,v} + \delta M_{u,v}$, it follows that $Y^n$ also converges to $Y$ in the uniform norm a.s., which in turn implies that $Y$ has finite $p$-variation a.s. Therefore, $(Y,Y') \in \mathrm{RSM}^{(p,p')\text{-}\mathrm{var}}_{\mathrm X}([s,t],\R)$ with $P^Y = P$ and $M^Y = M$. Finally, applying \eqref{e:conver of for terms} again we obtain 
\begin{equation*}
\lim_{n \to \infty}\, \thicknm{Y^n,(Y^n)';Y,Y'}_{\mathrm X,\mathrm X;q,q';k,\infty;[s,t]} = 0. 
\end{equation*}
Clearly $\lim_{n \to \infty}\|Y_t - Y^n_t\|_{\infty} = 0$, which completes the proof.  
\end{proof}

Given two RSMs $(Y, Y')$ and $(\bar Y, \bar Y')$ controlled by different rough paths $\mathbf X$ and $\bar{\mathbf X}$, consider $(H(Y),H(Y)')$ and $(\bar H(\bar Y), \bar H (\bar Y)')$ being the corresponding  composition of different nonlinear functions. The following Lemma~\ref{lem:Gamma and I} establishes an upper bound for the difference between the rough stochastic integral of $(H(Y),H(Y)')$ and $(\bar H(\bar Y), \bar H (\bar Y)')$. In particular, this estimate is locally Lipschitz continuous with respect to $\thicknm{Y, Y' ;\bar Y, \bar Y'}_{\mathrm X, \bar{\mathrm X};q,q';k,\infty}$. 

To facilitate the statement of Lemma~\ref{lem:Gamma and I}, let $R_1 > 0$ be a constant such that 
\begin{equation*}
\|\xi\|_{\infty} + \|\bar \xi\|_{\infty} + \|H\|_{C^{\gamma}_b} + \|\bar H\|_{C^{\gamma}_b} + |\delta \mathbf X |_{p\text{-}\mathrm{var};[0,T]} + |\delta \bar{\mathbf X} |_{p\text{-}\mathrm{var};[0,T]} \le R_1, 
\end{equation*}
and denote $\Theta_{1} := (\Theta,C_{f},R_1)$. Recall the parameter $q''/2 = qq'/(q+q')$.

\begin{lemma}\label{lem:Gamma and I}
Given $(s,t) \in \Delta$, assume \ref{(gamma)}, $H, \bar H\in C^{\gamma}_b$, and $\mathbf X, \bar{\mathbf X} \in \mathscr{C}^{p\text{-}\mathrm{var}}([0,T],\R^d)$. 
Suppose that $(Y,Y')$ and $(\bar Y,\bar Y')$ satisfy \eqref{e:Y in RSM} with $p' = p$ and $Y_{t},\bar Y_{t} \in L^{\infty}_t$. Denote 
\begin{equation*}
I^{\mathbf X}_{v} := \int_{s}^{v} H(Y_r) d\mathbf X_r,\ \bar I^{\bar{\mathbf X}}_{v} := \int_{s}^{v} \bar H (\bar  Y_r) d\bar {\mathbf X}_r,\quad v \in [s,t]. 
\end{equation*}
Assume further that for some constant $C>0$,  
\begin{equation}\label{e:upper bound uni in n}
\begin{aligned}
\thicknm{(Y,Y')}_{\mathrm X;q,q';k,\infty;[s,t]} + \thicknm{ (\bar Y,\bar Y' ) }_{\bar {\mathrm X} ,q, q';k,\infty;[s,t]} \le C,
\end{aligned}
\end{equation}

Then, the following estimates hold, with the constant depending implicitly on $(k,\gamma,\Theta_{1},C):$ 
\begin{equation}\label{e:R^Gamma <= R^H(Y)}
\begin{aligned}
&\big\| \|R^{I^{\mathbf X}} - R^{\bar I^{\bar {\mathbf X}}}\|_{\frac{p}{2}\text{-}\mathrm{var}} \big\|_{\infty;k,\infty;[s,t]} + \big\| \|\delta(I^{\mathbf X} - \bar I^{\bar{\mathbf X}})\|_{p\text{-}\mathrm{var}} \big\|_{\infty;k,\infty;[s,t]} \lesssim |\delta \mathbf X|_{p\text{-}\mathrm{var};[s,t]} \Big(\|Y_t - \bar Y_t\|_{\infty} 
\\ 
& \quad \quad + \thicknm{Y ,Y'; \bar Y, (\bar Y)'}_{\mathrm X,\bar{\mathrm X};q,q';k,\infty;[s,t]} + \|H - \bar H\|_{C^{\gamma - 1}_{b}}\Big) + |\delta( \mathbf X - \bar{\mathbf X} )|_{p\text{-}\mathrm{var};[s,t]}, \\
\end{aligned}
\end{equation}
\end{lemma}

\begin{proof}
Without loss of generality, we assume $\gamma \le 3$. By Proposition~\ref{prop:estim of integral'}, we have 
\begin{equation}\label{e:R^Gamma <= R^H(Y)'}
\begin{aligned}
\big\| \|R^{I^{\mathbf X}} - R^{\bar I^{\bar{\mathbf X}}}\|_{\frac{p}{2}\text{-}\mathrm{var}} \big\|_{\infty;k,\infty;[s,t]}
&\lesssim_{\Theta_1} |\delta (\mathbf X - \bar{\mathbf X}) |_{p\text{-}\mathrm{var};[s,t]}\\
&\quad + |\delta \mathbf X|_{p\text{-}\mathrm{var};[s,t]}\, \thicknm{H(Y),H(Y)'; \bar H(\bar Y), \bar H (\bar Y)'}_{\mathrm X,\bar{\mathrm X};q,q';k,\infty;[s,t]}, 
\end{aligned}
\end{equation}
where
\begin{equation*}
H(Y)' := DH(Y) Y'\ \text{ and }\   \bar H(\bar Y)' := D \bar H(\bar Y)  \bar Y'.
\end{equation*}
Thus, in order to prove \eqref{e:R^Gamma <= R^H(Y)}, it suffices to estimate  $\thicknm{H(Y),H(Y)'; \bar H(\bar Y), \bar H(\bar Y)'}_{\mathrm X,\bar{\mathrm X};q,q';k,\infty}$, which will be split into three steps. 

{\bf Step 1.} In this step, we estimate $P^{H(Y)} - P^{ \bar H(\bar Y)}$. We first provide some upper bounds for both the uniform norm and the $\| \|\bm\cdot\|_{q\text{-}\mathrm{var}} \|_{\infty;k,\infty}$ norm of $Y - \bar Y$, $Y$, and $\bar Y$. By  Lemma~\ref{lem:Y <= R^Y'} we have 
\begin{equation}\label{e:Y <= R^Y}
\begin{aligned}
&\|Y - \bar Y\|_{\mathcal L^{\infty}([s,t] \times \Omega)} + \big\| \|\delta (Y - \bar Y)\|_{q\text{-}\mathrm{var}} \big\|_{\infty;k,\infty;[s,t]} \\
& \lesssim_{\Theta_{1}} \|Y_t - \bar Y_t\|_{\infty} + \thicknm{Y ,Y'; \bar Y, \bar Y'}_{\mathrm X,\bar{\mathrm X};q,q';k,\infty;[s,t]} + |\delta (\mathbf X - \bar{\mathbf X}) |_{p\text{-}\mathrm{var};[s,t]}. 
\end{aligned}
\end{equation}
Moreover, since $\|Y\|_{\mathcal L^{\infty}([s,t] \times \Omega)} + \, \| \|\delta Y\|_{q\text{-}\mathrm{var}} \|_{\infty;k,\infty;[s,t]}$ is dominated by $\|Y_t\|_{\infty} + \thicknm{(Y,Y')}_{\mathrm X;q,q';k,\infty;[s,t]}$, by upper bound given by \eqref{e:upper bound uni in n}, we have
\begin{equation}\label{e:Y^n 1/q <= (Y^n,Y^n')}
\begin{aligned}
&\|Y\|_{\mathcal L^{\infty}([s,t] \times \Omega)} + \|\bar Y\|_{\mathcal L^{\infty}([s,t] \times \Omega)} +  \big\| \|\delta Y\|_{q\text{-}\mathrm{var}} \big\|_{\infty;k,\infty;[s,t]} +  \big\| \|\delta \bar Y\|_{q\text{-}\mathrm{var};[s,t]} \big\|_{\infty;k,\infty;[s,t]} \\
&\lesssim_{\Theta_{1}} \|Y_t\|_{\infty} + \|\bar Y_t\|_{\infty} +  \thicknm{(Y,Y')}_{\mathrm X;q,q';k,\infty;[s,t]} + \thicknm{(\bar Y,\bar Y')}_{\bar{\mathrm X};q,q';k,\infty;[s,t]}\lesssim_{\Theta_{1},C} 1.
\end{aligned}
\end{equation}

Now we estimate $P^{H(Y)} - P^{H(\bar Y)}$  and $P^{H(\bar Y)} - P^{\bar H(\bar Y)}$, respectively. By Proposition~\ref{prop:R^G - R^G} (noting $\gamma \le 3$) and the inequality $\big\| \|\delta Y\|_{q\text{-}\mathrm{var}} \big\|_{\infty;k,\infty} \le 1 + \big\| \|\delta Y\|_{q\text{-}\mathrm{var}} \big\|^{\gamma - 1}_{\infty;k,\infty}$, noting $\frac{2}{q''} \le \max\{\frac{\gamma - 1}{q}, \frac{2}{q}\}$, 
\begin{align}\nonumber 
&\big\| \|P^{H(Y)} - P^{H(\bar Y)}\|_{\frac{q''}{2}\text{-}\mathrm{var}} \big\|_{\infty;k,\infty;[s,t]} \lesssim_{\gamma,\Theta_1} \\ \nonumber 
&\quad \Big( 1 + \big\| \|\delta Y\|_{q\text{-}\mathrm{var}} \big\|^{\gamma - 1}_{\infty;k,\infty;[s,t]} + \big\| \|\delta \bar Y\|_{q\text{-}\mathrm{var}} \big\|^{\gamma - 1}_{\infty;k,\infty;[s,t]} + \big\| \|P^{\bar Y}\|_{\frac{q''}{2}\text{-}\mathrm{var}} \big\|_{\infty;k,\infty;[s,t]} + \|M^{\bar Y}\|_{\mathrm{BMO};[s,t]} \\ \label{e:R^H - R^H Ft} 
&\quad + \big\| \|\delta \bar Y\|_{q\text{-}\mathrm{var}} \big\|^{\gamma - 2}_{\infty;k,\infty;[s,t]} \|M^{\bar Y}\|_{\mathrm{BMO};[s,t]} \Big) \times \Big( \|Y - \bar Y\|_{\mathcal L^{\infty}([s,t] \times \Omega)}    \\ \nonumber 
&\quad + \big\| \|\delta (Y - \bar Y)\|_{q\text{-}\mathrm{var}} \big\|_{\infty;k,\infty;[s,t]} + \big\| \|P^Y - P^{\bar Y}\|_{\frac{q''}{2}\text{-}\mathrm{var};[s,t]} \big\|_{\infty;k,\infty;[s,t]} + \|\delta (M^Y - M^{\bar Y})\|_{\mathrm{BMO};[s,t]} \Big).
\end{align}
Then, by \eqref{e:Y^n 1/q <= (Y^n,Y^n')} and \eqref{e:R^H - R^H Ft}, it follows that 
\begin{align}\label{e:first estim of R^H - R^H}
&\big\| \|P^{H(Y)} - P^{H(\bar Y)}\|_{\frac{q''}{2}\text{-}\mathrm{var}} \big\|_{\infty;k,\infty;[s,t]}  
\lesssim_{\gamma,\Theta_{1},C} \|Y - \bar Y\|_{\mathcal L^{\infty}([s,t] \times \Omega)} \\ \nonumber 
&\ + \big\| \|\delta (Y - \bar Y)\|_{q\text{-}\mathrm{var}} \big\|_{\infty;k,\infty;[s,t]} + \big\| \|P^Y - P^{\bar Y}\|_{\frac{q''}{2}\text{-}\mathrm{var}} \big\|_{\infty;k,\infty;[s,t]} + \|\delta (M^Y - M^{\bar Y})\|_{\mathrm{BMO};[s,t]}. 
\end{align}
Therefore, combining \eqref{e:Y <= R^Y} and \eqref{e:first estim of R^H - R^H}, we get 
\begin{equation}\label{e:R^H - R^H Ft'}
\text{LHS of \eqref{e:first estim of R^H - R^H} } \lesssim_{\gamma,\Theta_1,C} \|Y_t - \bar Y_t\|_{\infty} + \thicknm{Y ,Y'; \bar Y, \bar Y'}_{\mathrm X,\bar{\mathrm X};q,q';k,\infty;[s,t]} + |\delta (\mathbf X - \bar{\mathbf X} )|_{p\text{-}\mathrm{var};[s,t]}.
\end{equation}
Denote
$
\hat H := H - \bar H\ \text{ and }\ \hat H(\bar Y)' := D\hat H(\bar Y) \bar Y'.
$
Then, Lemma~\ref{lem:int of RSM'} (letting $p' = q'p/q$ therein so that $\gamma - 2 \ge p/p'$) yields, in view of \eqref{e:upper bound uni in n}, 
\begin{equation}\label{e:H - H^star}
\big\| \|P^{\hat H(\bar Y)}\|_{\frac{q''}{2}\text{-}\mathrm{var}} \big\|_{\infty;k,\infty;[s,t]} \lesssim_{\gamma,\Theta_{1},C}  \|\hat H\|_{C^{\gamma - 1}_b}. 
\end{equation}

Thus, by \eqref{e:R^H - R^H Ft'} and \eqref{e:H - H^star} we get, noting $P^{H(Y)} - P^{\bar H(\bar Y)} = P^{H(Y)} - P^{H(\bar Y)} + P^{\hat H(\bar Y)}$, 
\begin{equation}\label{e:R^H - R^H Ft''}
\begin{aligned}
&\big\| \|P^{H(Y)} - P^{H(\bar Y)}\|_{\frac{q''}{2}\text{-}\mathrm{var}} \big\|_{\infty;k,\infty;[s,t]} + \big\| \|P^{\hat H(\bar Y)}\|_{\frac{q''}{2}\text{-}\mathrm{var}} \big\|_{\infty;k,\infty;[s,t]}  \\
&\lesssim_{\gamma,\Theta_1,C} \|Y_t - \bar Y_t\|_{\infty} + \thicknm{Y ,Y'; \bar Y, \bar Y'}_{\mathrm X,\bar{\mathrm X};q,q';k,\infty;[s,t]} + |\delta (\mathbf X - \bar{\mathbf X} )|_{p\text{-}\mathrm{var};[s,t]} + \|\hat H\|_{C^{\gamma - 1}_b}. 
\end{aligned}
\end{equation}

\

{\bf Step 2.} In this step,  we estimate $M^{H(Y)} - M^{\bar H(\bar Y)}$.  Note that for all $(u,v) \in \Delta_{[s,t]}$, 
\begin{equation*}
\delta (M^{H(Y)} - M^{\bar H(\bar Y)})_{u,v} = \int_{u}^{v} DH(Y_r) d (M^{Y}_r - M^{\bar Y}_r) + \int_{u}^{v} [DH(Y_r) - DH(\bar Y_r) + D\hat H (\bar Y_r)] dM^{\bar Y}_r. 
\end{equation*}
Then, it follows that  
\begin{equation}\label{e:MHY <= MY}
\begin{aligned} 
&\big\| \delta \big\langle M^{H(Y)} - M^{\bar H(\bar Y)} \big\rangle^{\frac{1}{2}}_{u,v} \big| \mathcal F_{u} \big\|_{2} \\ 
&= \Big\| \delta \Big\langle \int_s^{\bm\cdot} DH(Y_r) d (M^{Y}_r - M^{\bar Y}_r) + \int_0^{\bm\cdot} [DH(Y_r) - DH(\bar Y_r) + D\hat H(\bar Y_r)] dM^{\bar Y}_r \Big\rangle^{\frac{1}{2}}_{u,v} \Big| \mathcal F_{u} \Big\|_{2}\\ 
&\le \|DH\|_{\infty} \big\| \delta \big\langle M^{Y} - M^{\bar Y} \big\rangle^{\frac{1}{2}}_{u,v} \big| \mathcal F_{u}\big\|_{2} + (\|D^2 H\|_{\infty} \|Y - \bar Y\|_{\infty} + \|D\hat H\|_{\infty}) \big\| \delta \big\langle M^{\bar Y} \big\rangle_{u,v}^{\frac{1}{2}} \big| \mathcal F_{u} \big\|_2.
\end{aligned}
\end{equation}
Hence, by the boundedness of $D H$ and $D^2 H$, \eqref{e:upper bound uni in n}, and \eqref{e:Y <= R^Y}, the estimate \eqref{e:MHY <= MY} yields 
\begin{equation}\label{e:M^H^n - M^H^m}
\begin{aligned}
&\|\delta (M^{H(Y)} - M^{\bar H(\bar Y)})\|_{\mathrm{BMO};[s,t]}\\ 
&\lesssim_{\Theta_1} \|\delta (M^{Y} - M^{\bar Y})\|_{\mathrm{BMO};[s,t]} + (\|Y - \bar Y\|_{\mathcal L^{\infty}([s,t] \times \Omega)} + \|D\hat H\|_{\infty}) \|\delta M^{\bar Y}\|_{\mathrm{BMO};[s,t]}\\ 
&\lesssim_{\Theta_{1},C} \|\delta (M^{Y} - M^{\bar Y})\|_{\mathrm{BMO};[s,t]} + \|Y - \bar Y\|_{\mathcal L^{\infty}([s,t] \times \Omega)} + \|D\hat H\|_{\infty}\\
&\lesssim_{\Theta_1} \|Y_t - \bar Y_t\|_{\infty} + \thicknm{Y ,Y'; \bar Y, \bar Y'}_{\mathrm X,\bar{\mathrm X};q,q';k,\infty;[s,t]} + |\delta (\mathbf X - \bar{\mathbf X} )|_{p\text{-}\mathrm{var};[s,t]} + \|D\hat H\|_{\infty}.
\end{aligned}
\end{equation}

{\bf Step 3.} In this step, we estimate $H(Y)' - \bar H(\bar Y)'$. By the triangular inequality, we have  
\begin{align}\nonumber
&\|H(Y)' - \bar H(\bar Y)'\|_{\mathcal L^{\infty}([s,t] \times \Omega)} + \big\| \|\delta (H(Y)' - \bar H(\bar Y)')\|_{q'\text{-}\mathrm{var}} \big\|_{\infty;k,\infty;[s,t]}\\ \nonumber
&\le \|(DH(Y) - DH(\bar Y) ) Y'\|_{\mathcal L^{\infty}([s,t] \times \Omega)} + \|D\hat H(\bar Y)  Y'\|_{\mathcal L^{\infty}([s,t] \times \Omega)} + \| D\bar H(\bar Y) (Y' - \bar Y')\|_{\mathcal L^{\infty}([s,t] \times \Omega)}\\ \label{e:HY - HY <= Y - Y}
&\  + \big\| \|\delta( (DH(Y) - DH(\bar Y)) Y' )\|_{q'\text{-}\mathrm{var}} \big\|_{\infty;k,\infty;[s,t]} \\ \nonumber
&\  + \big\| \|\delta ( D \hat H(\bar Y) Y' )\|_{q'\text{-}\mathrm{var}} \big\|_{\infty;k,\infty;[s,t]} + \big\| \|\delta (D \bar H(\bar Y) (Y' - \bar Y') )\|_{q'\text{-}\mathrm{var}} \big\|_{\infty;k,\infty;[s,t]}.
\end{align}
For the first three terms on the RHS of \eqref{e:HY - HY <= Y - Y}, by the boundedness of $\|D^2 H\|_{\infty} + \|D^2 \bar H\|_{\infty}$,
\begin{align}\nonumber 
&
\|(DH(Y) - DH(\bar Y) ) Y'\|_{\mathcal L^{\infty}([s,t] \times \Omega)} +  \|D\hat H(\bar Y) Y'\|_{\mathcal L^{\infty}([s,t] \times \Omega)} + \| D \bar H(\bar Y) (Y' - \bar Y')\|_{\mathcal L^{\infty}([s,t] \times \Omega)}\\ \label{e:H - H* Y uniform} 
&\le \|D^{2}H\|_{\infty} 
\|Y - \bar Y\|_{\mathcal L^{\infty}([s,t] \times \Omega)} 
\|Y'\|_{\mathcal L^{\infty}([s,t] \times \Omega)} + \|D\hat H\|_{\infty}\|Y'\|_{\mathcal L^{\infty}([s,t] \times \Omega)} \\ \nonumber 
&\quad + \|D \bar H\|_{\infty} \|(Y' - \bar Y')\|_{\mathcal L^{\infty}([s,t] \times \Omega)}.
\end{align}
For the other three terms on the RHS of \eqref{e:HY - HY <= Y - Y}, by the inequality $|\delta( a\cdot b)_{u,v}| \le |\delta a_{u,v}| \cdot |b_{u}| + |\delta b_{u,v}| \cdot |a_{v}|$ and Lemma~\ref{lem:gamma - 2}, we get for all $(u,v) \in \Delta_{[s,t]}$,
\begin{align}\nonumber 
&\big\| \| \delta \big((DH(Y) - DH(\bar Y)) Y' \big)\|_{q'\text{-}\mathrm{var};[u,v]} \big| \mathcal F_{u}\big\|_{k,\infty} \\ \nonumber 
&\quad + \big\| \|\delta \big( D\hat H(\bar Y) Y' \big) \|_{q'\text{-}\mathrm{var};[u,v]} \big| \mathcal F_{u} \big\|_{k,\infty} +  \big\| \| \delta \big( D \bar H(\bar Y) (Y' - \bar Y') \big) \big\|_{q'\text{-}\mathrm{var};[u,v]} \big| \mathcal F_{u}\big\|_{k,\infty} \\ \label{e:H - H* Y hol}
&\lesssim \|DH\|_{C^{\gamma - 1}_b} \bigg[ \Big( \big(\big\| \|\delta Y\|_{q\text{-}\mathrm{var};[u,v]} \big| \mathcal F_{u} \big\|^{\gamma - 2}_{k,\infty} + \big\| \|\delta \bar Y\|_{q\text{-}\mathrm{var};[u,v]} \big| \mathcal F_u \big\|^{\gamma - 2}_{k,\infty}\big) \|Y - \bar Y\|_{\mathcal L^{\infty}([s,t] \times \Omega)} \\ \nonumber 
&\quad + \big\| \|\delta (Y - \bar Y)\|_{q\text{-}\mathrm{var};[u,v]} \big| \mathcal F_u \big\|_{k,\infty} \Big)\|Y'\|_{\mathcal L^{\infty}([s,t] \times \Omega)} + \|Y - \bar Y\|_{\mathcal L^{\infty}([s,t] \times \Omega)} \big\| \|\delta Y'\|_{q'\text{-}\mathrm{var};[u,v]} \big| \mathcal F_u \big\|_{k,\infty} \bigg] \\ \nonumber
&\quad + \|D \hat H\|_{C^1_b} \Big(\big\| \|\delta \bar Y\|_{q'\text{-}\mathrm{var};[u,v]} \big| \mathcal F_u \big\|_{k,\infty} \|Y'\|_{\mathcal L^{\infty}([s,t] \times \Omega)} + \big\| \|\delta Y'\|_{q'\text{-}\mathrm{var};[u,v]} \big| \mathcal F_u \big\|_{k,\infty} \\ \nonumber
&\qquad \qquad \qquad \times \|\bar Y\|_{\mathcal L^{\infty}([s,t] \times \Omega)} \Big) + \|D \bar H\|_{C^1_b} \Big(\big\| \|\delta \bar Y\|_{q'\text{-}\mathrm{var};[u,v]} \big| \mathcal F_u \big\|_{k,\infty} \|Y' - \bar Y'\|_{\mathcal L^{\infty}([s,t] \times \Omega)} \\ \nonumber
&\qquad \qquad \qquad \qquad \qquad \qquad \qquad \qquad \qquad  + \big\| \|\delta (Y' - \bar Y')\|_{q'\text{-}\mathrm{var};[u,v]} \big| \mathcal F_u \big\|_{k,\infty} \|\bar Y\|_{\mathcal L^{\infty}([s,t] \times \Omega)} \Big). 
\end{align}
Then, combining \eqref{e:HY - HY <= Y - Y}, \eqref{e:H - H* Y uniform}, and \eqref{e:H - H* Y hol}, and applying Young's inequality $|a|^{\gamma-2} + |a| \le 2 + 2 |a|^{(\gamma - 2) \vee 1}$, noting that $\frac{\gamma - 2}{q} \ge \frac{1}{q'}$ and $q \le q'$, we have 
\begin{equation}\label{e:H - H* Y hol'}
\begin{aligned}
&\|H(Y)' - \bar H(\bar Y)'\|_{\mathcal L^{\infty}([s,t] \times \Omega)} + \big\| \|\delta (H(Y)' - \bar H(\bar Y)' )\|_{q'\text{-}\mathrm{var}} \big\|_{\infty;k,\infty;[s,t]} \\
&\lesssim_{\Theta_{1}} \Big(1 + \|Y_{t}\|_{\infty} + \|\bar Y_{t}\|_{\infty} + \big\| \|\delta Y\|_{q\text{-}\mathrm{var}} \big\|^{(\gamma - 2)\vee 1}_{\infty;k,\infty;[s,t]} + \big\| \|\delta \bar Y\|_{q\text{-}\mathrm{var}} \big\|^{(\gamma - 2)\vee 1}_{\infty;k,\infty;[s,t]}\Big) \\
&\quad \times \Big( 1 + \|Y'\|_{\mathcal L^{\infty}([s,t] \times \Omega)} + \big\| \|\delta Y'\|_{q'\text{-}\mathrm{var}} \big\|_{\infty;k,\infty;[s,t]} \Big)\\
&\quad \times \Big( \|Y - \bar Y\|_{\mathcal L^{\infty}([s,t] \times \Omega)} + \big\| \|\delta (Y - \bar Y)\|_{q\text{-}\mathrm{var}} \big\|_{\infty;k,\infty;[s,t]} + \|D \hat H\|_{C^{1}_{b}}\\ 
&\quad\quad + \|Y' - \bar Y'\|_{\mathcal L^{\infty}([s,t] \times \Omega)} + \big\| \| \delta(Y' - \bar Y') \|_{q'\text{-}\mathrm{var}} \big\|_{\infty;k,\infty;[s,t]} \Big)\\
&\lesssim_{\gamma,\Theta_1,C} \Big( \|Y - \bar Y\|_{\mathcal L^{\infty}([s,t] \times \Omega)} + \big\| \|\delta (Y - \bar Y)\|_{q\text{-}\mathrm{var}} \big\|_{\infty;k,\infty;[s,t]} + \|D \hat H\|_{C^{1}_{b}}\\ 
&\quad\quad + \|Y' - \bar Y'\|_{\mathcal L^{\infty}([s,t] \times \Omega)} + \big\| \| \delta(Y' - \bar Y') \|_{q'\text{-}\mathrm{var}} \big\|_{\infty;k,\infty;[s,t]} \Big)\\
&\lesssim_{\Theta_{1}} \|Y_t - \bar Y_t\|_{\infty} + \thicknm{Y ,Y'; \bar Y, \bar Y'}_{\mathrm X,\bar{\mathrm X};q,q';k,\infty;[s,t]} + |\delta ( \mathbf X - \bar{\mathbf X} )|_{p\text{-}\mathrm{var};[s,t]} + \|D\hat H\|_{C^{1}_{b}},
\end{aligned}
\end{equation}
where the second inequality is due to \eqref{e:upper bound uni in n} and \eqref{e:Y^n 1/q <= (Y^n,Y^n')}; and the last is due to \eqref{e:Y <= R^Y}.

Finally, combining \eqref{e:R^H - R^H Ft''}, \eqref{e:M^H^n - M^H^m}, and \eqref{e:H - H* Y hol'}, which were established in the above three steps, we obtain 
\begin{equation}\label{e:HY <= Y}
\begin{aligned}
&\thicknm{H(Y),H(Y)' ; \bar H(\bar Y), \bar H(\bar Y)'}_{\mathrm X,\bar{\mathrm X};q,q';k,\infty;[s,t]} \\
&\lesssim_{\gamma,\Theta_{1},C} \|Y_t - \bar Y_t\|_{\infty} + \thicknm{Y ,Y'; \bar Y, \bar Y'}_{\mathrm X,\bar{\mathrm X};q,q';k,\infty;[s,t]} + |\delta (\mathbf X - \bar{\mathbf X}) |_{p\text{-}\mathrm{var};[s,t]} + \|\hat H\|_{C^{\gamma - 1}_{b}}. 
\end{aligned}
\end{equation}
Hence, by \eqref{e:R^Gamma <= R^H(Y)'} and \eqref{e:HY <= Y}, $\| \|R^{I^{\mathbf X}} - R^{\bar I^{\bar{\mathbf X}}}\|_{\frac{p}{2}\text{-}\mathrm{var}} \|_{\infty;k,\infty;[s,t]}$ is bounded by the RHS of \eqref{e:R^Gamma <= R^H(Y)}. 

\

To conclude the proof, we prove the estimate~\eqref{e:R^Gamma <= R^H(Y)}. Denote $(H(Y)\delta \mathrm X)_{u,v} := H(Y_u)\delta \mathrm X_{u,v}$ and $(\bar H(\bar Y)\delta \bar{\mathrm X})_{u,v} := \bar H(\bar Y_u) \delta \bar{\mathrm X}_{u,v}$. Note that 
\begin{equation}\label{e:Gubinelli deri of Gamma}
\begin{aligned}
&\big\| \|H(Y)\delta \mathrm X - \bar H(\bar Y)\delta \bar{\mathrm X} \|_{p\text{-}\mathrm{var}} \big\|_{\infty;k,\infty;[s,t]} \\ 
&\le |\delta \bar{\mathbf X} |_{p\text{-}\mathrm{var};[s,t]} \|H(Y) - \bar H(\bar Y)\|_{\mathcal L^{\infty}([s,t] \times \Omega)} + |\delta (\mathbf X - \bar{\mathbf X}) |_{p\text{-}\mathrm{var};[s,t]} \|H(Y)\|_{\mathcal L^{\infty}([s,t] \times \Omega)} \\ 
&\lesssim_{\Theta_1} |\delta \bar{\mathbf X} |_{p\text{-}\mathrm{var};[s,t]} \Big( \|H(Y_t) - \bar H(\bar Y_t)\|_{\infty} + \thicknm{H(Y),H(Y)'; \bar H(\bar Y), \bar H(\bar Y)'}_{\mathrm X, \bar{\mathrm X};q,q';k,\infty;[s,t]} \Big)	\\
&\quad + |\delta (\mathbf X - \bar{\mathbf X}) |_{p\text{-}\mathrm{var};[s,t]} \\ 
&\lesssim_{\gamma,\Theta_1,C} |\delta \bar {\mathbf X} |_{p\text{-}\mathrm{var};[s,t]} \Big( \|Y_t - \bar Y_t\|_{\infty} + \thicknm{Y ,Y'; \bar Y, \bar Y'}_{\mathrm X,\bar{\mathrm X};q,q';k,\infty;[s,t]} + \|\hat H\|_{C^{\gamma - 1}_{b}} \Big) \\
&\quad + |\delta (\mathbf X - \bar{\mathbf X}) |_{p\text{-}\mathrm{var};[s,t]},
\end{aligned}
\end{equation}
where the second and the last inequalities follow from the fact that 
\[\|H(Y) - \bar H(\bar Y)\|_{\mathcal L^{\infty}([s,t] \times \Omega)} \le \|H(Y_t) - \bar H(\bar Y_t)\|_{\infty} + \thicknm{H(Y),H(Y)'; \bar H(\bar Y),\bar H(\bar Y)'}_{\mathrm X,\bar{\mathrm X};q,q';k,\infty;[s,t]}\]
and \eqref{e:HY <= Y}, respectively. Furthermore, noting $\delta(I^{\mathbf X} - \bar I^{\bar{\mathbf X}})_{u,v} = (R^{I^{\mathbf X}} - R^{\bar I^{\bar{\mathbf X}}})_{u,v} + H(Y_u)\delta \mathrm X_{u,v} - \bar H(\bar Y_u)\delta \bar{\mathrm X}_{u,v}$, we have  
\begin{equation*}
\begin{aligned}
&\big\| \| \delta (I^{\mathbf X} - \bar I^{\bar{\mathbf X}}) \|_{p\text{-}\mathrm{var}} \big\|_{\infty;k,\infty;[s,t]}\\ 
&\le \big\| \|(R^{I^{\mathbf X}} - R^{\bar I^{\bar{\mathbf X}}})\|_{p\text{-}\mathrm{var}} \big\|_{\infty;k,\infty;[s,t]} +  \big\| \|H(Y) \delta \mathrm X - \bar H(\bar Y) \delta \bar{\mathrm X} \|_{p\text{-}\mathrm{var}} \big\|_{\infty;k,\infty;[s,t]}. 
\end{aligned}
\end{equation*}
Therefore, since $\| \| R^{I^{\mathbf X}} - R^{\bar I^{\bar{\mathbf X}}}\|_{\frac{p}{2}\text{-}\mathrm{var}} \|_{\infty;k,\infty;[s,t]}$ is bounded by the RHS of \eqref{e:R^Gamma <= R^H(Y)}, by  \eqref{e:Gubinelli deri of Gamma}, we obtain the desired \eqref{e:R^Gamma <= R^H(Y)}. 
\end{proof}

%--------------------

\section{Rough BSDE}\label{sec:Rough BSDE}

In this section, we establish the well-posedness of the rough BSDE~\eqref{e:BSDE}. In Section~\ref{subsec:comparison}, we establish a comparison principle for rough BSDEs. In Section~\ref{subsec:a priori estimate}, we first derive an RSM \emph{a priori} estimate for the solution, by clarifying how the \textrm{RSM} seminorm $\thicknm{\bm\cdot}_{\mathrm X;q,q';k,\infty}$, defined in \eqref{e:Psi}, depends on the \textrm{BMO} norm of $M^Y$ and the essential supremum of $Y$.  
These key estimates are subsequently used in Section~\ref{subsec:Rough BSDE with low regularity H} to study the existence and stability of solutions. 

Throughout this section, $p\in(2,3)$ is a fixed number, $p',q,q'$ are numbers such that $p \le p'$, $p \le q$, and $q \le q'$.

We first introduce the notion of solutions to rough BSDE~\eqref{e:BSDE}. In particular, no integrability condition on the solution is required in the definition.

\begin{definition}\label{def:def of BSDEs}
Let $\mathbf X\in \mathscr{C}^{p\text{-}\mathrm{var}}([0,T],\R^d)$ and $H \in C^{2}(\R,\R^d)$. For a continuous adapted process $Y$ and a progressively measurable process $Z$, we call $(Y,Z)$ a solution to Eq.~\eqref{e:BSDE} with the parameter $(\xi,f,H,\mathbf X)$, if $(Y,-H(Y)) \in \mathrm{RSM}^{(p,p')\text{-}\mathrm{var}}_{\mathrm X}([0,T],\R)$ with some $(p,p')$ satisfying \ref{(p,p')}, $\int_{0}^{T} |Z_r|^2 dr < \infty$ a.s., and the equation~\eqref{e:BSDE} holds a.s.  
\end{definition}

We note that $M^Y_t = \int_0^t Z_r dW_r$. This relation is in force throughout this section, as follows from Remark~\ref{rem:given (Y,Y') in RSm}.

We also note that if $(Y,Z)$ is a solution to Eq.~\eqref{e:BSDE} with some $p'\ge p$, the regularity of the solution can be improved to $(Y,-H(Y)) \in \mathrm{RSM}^{(p,p)\text{-}\mathrm{var}}_{\mathrm X}$. Indeed, by a standard localization argument (see, e.g., \cite[Lemma~4.7]{dause2026controlled}), one may apply Corollary~\ref{cor:int of the solution} with $G = H$, and assume that the RHS of \eqref{e:(Gamma,-H(Y))} is finite. Since $(Y,Z)$ satisfies Eq.~\eqref{e:BSDE}, it holds that    
\begin{equation*}
\delta Y_{s,t} = -\delta \Lambda_{s,t} - \int_{s}^{t} f(r,Y_r,Z_r) dr + \delta M^{Y}_{s,t}, \quad P^{Y}_{s,t} = -R^{\Lambda}_{s,t} - \int_{s}^{t} f(r,Y_r,Z_r) dr.
\end{equation*}
Note that by \eqref{e:(Gamma,-H(Y))} the processes $\Lambda$ and $R^{\Lambda}$ have finite $p$-variation and finite $p/2$-variation, respectively. Thus, $Y$ and $H(Y)$ have finite $p$-variation, and the pathwise remainder $P^Y$ has finite $p/2$-variation. 
This implies that $(Y,-H(Y))$ belongs to $\mathrm{RSM}^{(p,p)\text{-}\mathrm{var}}_{\mathrm X}$.

\subsection{Comparison Theorem}\label{subsec:comparison}
In this subsection, we prove the comparison principle for rough BSDEs solved in the sense of Definition~\ref{def:def of BSDEs}, comparing two solutions corresponding to different terminal values and $f$-terms. As a corollary, the uniqueness of solutions is obtained.  

We first relate the norm $\|\bm\cdot| \mathcal F_s\|_{1,\infty}$ to the $\mathrm{VMO}^{p\text{-}\mathrm{var}}$ process. The John--Nirenberg inequality for $\mathrm{VMO}^{p\text{-}\mathrm{var}}$ processes, proved in \cite{le2022quantitative}, is essential for the proof of the comparison.   

\begin{lemma}\label{lem:cond for VMO process}
Assume $S$ is an adapted continuous process. For $p \ge 1$ and a control function $w$, assume further that $\| \delta S_{s,t} | \mathcal F_{s} \|_{1,\infty} \le w(s,t)^{1/p}$ for all $(s,t) \in \Delta$. 

Then, $S$ is a $\mathrm{VMO}^{p\text{-}\mathrm{var}}$ process, and (recall $\varrho(\bm\cdot)$ is the modulus of mean oscillation)
\begin{equation*}
\|\varrho(S)\|_{p\text{-}\mathrm{var};[0,T]} \le 2 w(0,T)^{\frac{1}{p}}. 
\end{equation*}
\end{lemma}

\begin{proof}
By \cite[Proposition~2.2]{le2022quantitative}, we have $\varrho_{s,t}(S) \le 2 \sup_{u\in[s,t]} \| \delta S_{u,t} | \mathcal F_{u}\|_{1,\infty}$. Consequently, 
\begin{equation*}
\|\varrho(S)\|_{p\text{-}\mathrm{var}} \le 2 \sup_{\pi\in \mathcal P_{[0,T]}} \Big\{ \sum_{[s,t]\in\pi} \sup_{u\in[s,t]} \|\delta S_{u,t} | \mathcal F_{u}\|^{p}_{1,\infty} \Big\}^{\frac{1}{p}} \le 2 \sup_{\pi\in \mathcal P_{[0,T]}} \Big\{ \sum_{[s,t]\in\pi} w(s,t) \Big\}^{\frac{1}{p}} \le 2 w(0,T)^{\frac{1}{p}},
\end{equation*}
which completes the proof. 
\end{proof}

\begin{theorem}[Comparison theorem]\label{thm:comparison}
Assume \ref{(gamma)} holds, $H:\R \to \R^d$ is a function such that $DH\in C^{\gamma - 1}_{b}(\R,\R^d)$, and assume $\mathbf X \in \mathscr{C}^{p\text{-}\mathrm{var}}([0,T],\R^d)$. Let $\xi$ and $\overline \xi$ be two terminal values belonging to $L^{\infty}$, and let $f$ and $\overline f$ be two functions satisfying \ref{(f)}. Suppose $(Y,Z)$ and $(\overline Y, \overline Z)$ are solutions to Eq.~\eqref{e:BSDE} with terminal values $\xi$ and $\overline \xi$ and generators $f$ and $\overline{f}$, respectively, such that 
\begin{equation*}
\thicknm{(Y,-H(Y))}_{\mathrm X;q,q';1,\infty} +\thicknm{(\overline Y,-H(\overline Y))}_{\mathrm X;q,q';1,\infty} < \infty.
\end{equation*}
Then, if $\xi \le \overline \xi$ and $f(r,\overline Y_r,\overline Z_r) \le \overline{f}(r,\overline Y_r,\overline Z_r)$ a.s. for every $r \in [t,T]$, we have $Y_{t} \le \overline Y_{t}$ a.s. 
Moreover, if $\xi < \overline \xi$ a.s., then the comparison is strict; namely, $Y_t < \overline Y_t$ a.s.   
\end{theorem}

\begin{remark}\label{rem:on strict comparison}
The strict comparison property is not immediate for solutions constructed via smooth approximation approaches, such as those in \cite[Theorem~3]{DF} and \cite[Definition~3.1]{liang2025multidimensional}. 
By contrast, the \emph{RSM} framework allows us to establish the strict inequality directly. 
\end{remark}

\begin{proof}
Assume $d=1$ without loss of generality. Let $R \in (0,\infty)$ be a constant such that  
\begin{equation*}
\|DH\|_{C^{\gamma - 1}_b} + C_{f} +  \thicknm{(Y,-H(Y))}_{\mathrm X;q,q';1,\infty} +\thicknm{(\overline Y,-H(\overline Y))}_{\mathrm X;q,q';1,\infty } + |\delta \mathbf X |_{p\text{-}\mathrm{var}} + T \le R.
\end{equation*}For $t\in[0,T]$, define 
\begin{equation*}
\begin{cases}\displaystyle
\alpha_{t} := \int_{0}^{1} DH(Y_t + \lambda(\overline Y_t - Y_t) ) d\lambda,\\ \displaystyle
\alpha'_t := - \int_{0}^{1} D^2 H(Y_{t} + \lambda(\overline Y_t - Y_t) ) \big( (1-\lambda) H(Y_t) + \lambda H(\overline Y_t) \big) d\lambda,\\ \displaystyle
\beta_{t} := \frac{f(t,Y_t,Z_t) - f(t,\overline Y_t,Z_t)}{Y_t - \overline Y_t}\mathbf 1_{\{Y_t - \overline Y_t \neq 0\}},\ \text{ and }\ \theta_{t} := \frac{f(t,\overline Y_t,Z_t) - f(t,\overline Y_t,\overline Z_t)}{Z_t - \overline Z_t}\mathbf 1_{\{Z_t - \overline Z_t \neq 0\}}. 
\end{cases}
\end{equation*}

One can assume $q' > q$ without loss of generality, since one can always choose a constant $\tilde q' > q'$ so that $(\gamma,p,q,\tilde q')$ still satisfies condition~\ref{(gamma)}.
According to \ref{(f)} and the condition $DH\in C^{\gamma - 1}_b$, we see that $\alpha$, $\beta$, $\theta$, and $\alpha'$ are bounded adapted processes. Furthermore, letting $p' := \frac{q'p}{q}$, then $\gamma - 1 \ge \frac{p+p'}{p'}$. Consequently, since $DH \in C^{\gamma - 1}_b(\R,\R^d)$ and $\thicknm{(Y,-H(Y))}_{\mathrm X;q,q';1,\infty} +\thicknm{(\overline Y,-H(\overline Y))}_{\mathrm X;q,q';1,\infty} \le R$, by Lemma~\ref{lem:int of RSM} we have that $(\alpha,\alpha')\in \mathrm{RSM}^{(p,p')\text{-}\mathrm{var}}_{\mathrm X}$. Moreover, it follows from \eqref{e:Psi alpha'} in Lemma~\ref{lem:int of RSM'} that for some constant $\hat C_{\gamma,R} > 0$, 
\begin{equation}\label{e:alpha <= X Psi alpha}
\Big\|\Big\|\delta \Big( \int_{u}^{\bm\cdot}\alpha_{r}d\mathbf X_r \Big) \Big\|_{p\text{-}\mathrm{var};[u,v]}\Big| \mathcal F_u \Big\|_{1,\infty} \le \hat C_{\gamma,R} |\delta \mathbf X|_{p\text{-}\mathrm{var};[u,v]} \ \text{ for all }\ (u,v) \in \Delta.
\end{equation}

Therefore, $(Y-\overline Y,Z-\overline Z)$ solves the following linear BSDE for $t\in[0,T]$: 
\begin{equation*}
\begin{aligned}
(Y-\overline Y)_{t} = \hat \xi &+ \int_{t}^{T} [\beta_{r}(Y - \overline Y)_r + \theta_{r}(Z - \overline Z)_r + \hat f_r] dr + \int_{t}^{T} \alpha_{r} (Y - \overline Y)_{r} d\mathbf X_{r} - \int_{t}^{T} (Z - \overline Z)_{r} dW_{r}, 
\end{aligned}
\end{equation*}
where $\hat \xi := \xi - \overline \xi$ and $\hat f_r := (f - \overline f)(r,\overline Y_r,\overline Z_r)$. For $(t,r) \in \Delta$, define 
\begin{equation*}
S^{t}_{r} := \int_{t}^{r} 
\Big(\beta_s - \frac{1}{2} |\theta_s|^2 \Big) ds + 
\frac{1}{2} \mathrm{tr}\Big( \int_{t}^{r} \alpha_s \alpha^{\top}_s d[\mathbf X]_s \Big)
+ \int_{t}^{r} \theta_{s} dW_{s} + \int_{t}^{r} \alpha_{s} d\mathbf X_s .
\end{equation*}
Then, an application of the rough It\^o formula in \cite[Theorem~4.19]{dause2026controlled} to $(Y_{\bm\cdot} - \overline Y_{\bm\cdot}) e^{S^{t}_{\bm\cdot}}$ yields
\begin{equation}\label{e:Ito to Y - Y}
(Y_t - \overline Y_t) = e^{S^{t}_{T}} \hat \xi + \int_{t}^{T} e^{S^t_r} \hat f_r dr - \int_{t}^{T} [ e^{S^t_r} (Z_r - \overline Z_r)  + \theta_{r} e^{S^{t}_r} (Y_r - \overline Y_r)] dW_{r}.
\end{equation}

Note $\| \int_{u}^{v}\alpha_{r} d\mathbf X_{r} | \mathcal F_{u}\|_{1,\infty} \le \|\|\delta ( \int_{u}^{\bm\cdot}\alpha_{r}d\mathbf X_r)\|_{p\text{-}\mathrm{var};[u,v]}\big| \mathcal F_u \|_{1,\infty}$ for all $(u,v) \in \Delta$. By \eqref{e:alpha <= X Psi alpha} and Lemma~\ref{lem:cond for VMO process}, $\int_{t}^{\bm\cdot}\alpha_{r}d\mathbf X_r$ is a $\mathrm{VMO}^{p\text{-}\mathrm{var}}$ process with $\| \varrho(\int_{t}^{\bm\cdot}\alpha_{r}d\mathbf X_r)\|_{p\text{-}\mathrm{var};[t,T]} \le 2\hat C_{\gamma,R} |\delta \mathbf X|_{p\text{-}\mathrm{var};[t,T]}$. 
Furthermore, according to the estimate for Young integrals shown in \cite[Theorem~6.8]{friz2010multidimensional}, it holds $\|\int_{u}^{v}\alpha_r \alpha^{\top}_r d[\mathbf X]_r | \mathcal F_{u}\|_{1,\infty} \le \tilde{C}_{\gamma,R} |\delta \mathbf X|_{p\text{-}\mathrm{var};[u,v]}$ for some constant $\tilde{C}_{\gamma,R} > 0$. Hence, by Lemma~\ref{lem:cond for VMO process} again, $\int_{t}^{\bm\cdot}\alpha_r \alpha^{\top}_r d[\mathbf X]_r$ is a $\mathrm{VMO}^{p\text{-}\mathrm{var}}$ process with $\| \varrho(\int_{t}^{\bm\cdot}\alpha_r \alpha^{\top}_r d[\mathbf X]_r)\|_{p\text{-}\mathrm{var};[t,T]} \le 2\tilde{C}_{\gamma,R} |\delta \mathbf X|_{p\text{-}\mathrm{var};[t,T]}$. 
Similarly, by the boundedness of $\beta,\theta$, one can prove that $\int_{t}^{\bm\cdot} (\beta_s - \frac{1}{2}|\theta_s|^2) ds + \int_{t}^{\bm\cdot} \theta_s dW_s$ is  a $\mathrm{VMO}^{p\text{-}\mathrm{var}}$ process, with the $p$-variation of its modulus of mean oscillation dominated by $R$. 

Consequently, $S^t_{\bm\cdot}$ is a $\mathrm{VMO}^{p\text{-}\mathrm{var}}$ process. Hence, by the John--Nirenberg inequality shown in \cite[Theorem~3.4]{le2022quantitative}, for some constant $C_{\gamma,R} > 0$ it holds that for all $r\in[t,T]$, 
$
\|e^{S^{t}_{r}}\|_{2} \le C_{\gamma,R}.
$
Therefore, by Eq.~\eqref{e:Ito to Y - Y} $\int_{t}^{\bm\cdot} [ e^{S^{t}_{r}}(Z_r - \overline Z_r) + \theta_r e^{S^{t}_{r}} (Y_r - \overline Y_r) ] dW_r$ is uniformly integrable and thus a martingale. Taking conditional expectation of \eqref{e:Ito to Y - Y}, we have 
\begin{equation*}
(Y - \overline Y)_{t} = \E_{t} \Big[e^{S^{t}_{T}} \hat \xi + \int_{t}^{T} e^{S^{t}_{r}} \hat f_r dr \Big].
\end{equation*}
This yields the (strict) comparison result, noting that $e^{S_r^t} > 0$  and $\hat f_r \ge 0$ a.s.
\end{proof}

\begin{remark}\label{rem:comparison to sup solution}
The theory of BSDEs also includes supersolutions, which correspond to dynamics of the form \eqref{e:BSDE} augmented by an additional increasing process. We refer to \cite{el1997backward} for further discussion of supersolutions of BSDEs. As for the comparison theorem, we note that Theorem~\ref{thm:comparison} extends without difficulty to the case where $\overline Y_t$ is a supersolution.
\end{remark}

\begin{corollary}\label{cor:uniqueness}
Assume \ref{(gamma)} holds,  $H:\R\to\R^d$ is a function such that $DH\in C^{\gamma - 1}_{b}(\R,\R^d)$, and $\mathbf X \in \mathscr{C}^{p\text{-}\mathrm{var}}([0,T],\R^d)$. Suppose that $\xi \in L^{\infty}$ and that $f$ satisfies \ref{(f)}. Then, Eq.~\eqref{e:BSDE} admits at most one solution $(Y,Z)$ satisfying $\thicknm{(Y,-H(Y))}_{\mathrm X;q,q';1,\infty} < \infty$.
\end{corollary}

\begin{proof}
Given any two solutions $(Y^{(1)},Z^{(1)})$ and $(Y^{(2)},Z^{(2)})$, Theorem~\ref{thm:comparison} yields that $Y^{(1)} = Y^{(2)}$. Consequently, by  Eq.~\eqref{e:BSDE} we get $\int_{0}^{t}(Z^{(1)}_{r} -Z^{(2)}_{r})dW_{r} = \int_{0}^{t}\big(f(r,Y^{(1)}_r,Z^{(1)}_r) - f(r,Y^{(2)}_r,Z^{(2)}_r) \big) dr$. 
This implies $Z^{(1)} = Z^{(2)}$, since a continuous local martingale of finite variation is constant.
\end{proof}

\subsection{The Rough Semimartingale \emph{a priori} Estimate}\label{subsec:a priori estimate}

In this subsection, assuming that $(Y,Z)$ is a solution to Eq.~\eqref{e:BSDE}, and that $\thicknm{(Y,-H(Y))}_{\mathrm X;q,q';k,\infty}$ is finite, we establish upper bounds for $\|Y\|_{\infty}$ and $\thicknm{(Y,-H(Y))}_{\mathrm X;p,p;k,\infty}$. 

Throughout this subsection, we assume \ref{(p,q,q')} holds, $\mathbf X\in \mathscr{C}^{p\text{-}\mathrm{var}}([0,T], \mathbb R^d)$, $H\in C^{2}_{b}(\mathbb R,\mathbb R^d)$, $\xi \in L^{\infty}$, and assumption~\ref{(f)} is satisfied. Let $(Y,Z)$ be a solution to Eq.~\eqref{e:BSDE} in the sense of Definition~\ref{def:def of BSDEs}, and thus (see the discussion under Definition~\ref{def:def of BSDEs}) $(Y,-H(Y)) \in \mathrm{RSM}^{(p,p)\text{-}\mathrm{var}}_{\mathrm X}$. 

\begin{proposition}\label{prop:bounded by BMO}
Assume $q<q'$. For any $k>1$, assume $(Y,Z)$ is a solution to Eq.~\eqref{e:BSDE} with 
\begin{equation}\label{e:cond that Psi < infty}
\thicknm{(Y,-H(Y))}_{\mathrm X;q,q';k,\infty} < \infty, 
\end{equation}
Assume $\|H\|_{C^{2}_b} + C_f + |\delta \mathbf X |_{p\text{-}\mathrm{var}} + T \le K$ for some positive constant $K$.  Then, there exists a constant $C_{k,K} > 0$ depending only on $(k,K)$ such that with $\theta := {\frac{2}{q-2}} > 1$,
\begin{equation}\label{e:bound-Psi}
{ \thicknm{(Y,-H(Y))}_{\mathrm X;p,p;k,\infty} \le C_{k,K} \Big(1+ \|\xi\|_{\infty} + 
\|\delta M^{Y}\|^{\theta}_{\mathrm{BMO}} \Big). }
\end{equation}
\end{proposition}

\begin{proof}
For $t\in[0,T]$, denote
\begin{equation}\label{e:def of Gamma}
I^{\mathbf X}_{t} := \int_{0}^{t}H(Y_{r})d\mathbf X_{r},\quad \mathcal I_{t} := \int_{0}^{t}f(r,Y_r,Z_r) dr.
\end{equation}
We present the proof in four steps to facilitate readability.

{\textbf{Step 1.}} 
First, we establish upper bounds for the $p$-variations of $M^{Y}$ and $Y$. Note that
\begin{equation*}
\delta M^{Y}_{s,T} = \int_{s}^{T}Z_{r} dW_r = \xi - Y_{s} + \delta I^{\mathbf X}_{s,T} + \delta \mathcal I_{s,T}, \quad s\in[0,T]. 
\end{equation*}
By Lemma~\ref{lem:Lepingle} and Doob's maximal inequality, noting $k > 1$, we have
\begin{equation}\label{e:tilde Z}
\begin{aligned}
\big\| \|\delta M^{Y}\|_{p\text{-var};[s,T]} \big|\mathcal F_{s} \big\|_{k} &\lesssim_{k} \big\| \| M^{Y}_{\bm\cdot} - M^Y_s \|_{\infty;[s,T]} \big|\mathcal F_{s} \big\|_{k} 
\lesssim_{k} \|  \delta M^{Y}_{s,T} |\mathcal F_{s} \|_{k}\\
& \le \|\xi - Y_{s} |\mathcal F_{s} \|_{k} + \| \delta I^{\mathbf X}_{s,T} | \mathcal F_{s}\|_{k}+ \|\delta \mathcal I_{s,T} | \mathcal F_{s} \|_{k}.
\end{aligned}
\end{equation}
Note $\xi - Y_{s} = \xi - \E_{s}[\xi] + \E_{s}[\xi - Y_s]$ and $\E_{s}[\xi - Y_s] = - \E_{s}[\delta (I^{\mathbf X} + \mathcal I)_{s,T}]$. Hence, we have
\begin{equation}\label{e:xi <= Gamma}
\begin{aligned}
\|\xi - Y_{s} |\mathcal F_{s} \|_{k} \le \|\xi - \E_{s}[\xi] |\mathcal F_{s} \|_{k} + \| \E_{s}[\delta (I^{\mathbf X} + \mathcal I)_{s,T}] |\mathcal F_{s}\|_{k}\le 2\|\xi\|_{\infty} + \| \delta (I^{\mathbf X} + \mathcal I)_{s,T} |\mathcal F_{s}\|_{k}.
\end{aligned}
\end{equation}
Then, combining \eqref{e:tilde Z} and \eqref{e:xi <= Gamma}, we have 
\begin{equation}\label{e:M^Y <= Gamma}
\big\| \|\delta M^{Y}\|_{p\text{-var};[s,T]} \big|\mathcal F_{s} \big\|_{k} \lesssim_{k} \|\xi\|_{\infty} + \big\| \|\delta I^{\mathbf X}\|_{p\text{-var};[s,T]} \big| \mathcal F_{s} \big\|_{k} + \big\| \|\delta \mathcal I\|_{p\text{-var};[s,T]} \big| \mathcal F_{s} \big\|_{k}.
\end{equation}
This, together with the fact that $\delta Y = -\delta I^{\mathbf X} - \delta \mathcal I + \delta M^Y$, yields
\begin{equation}\label{e:estim of Yn+1 (1)}
\big\| \|\delta Y\|_{p\text{-var};[s,T]} \big| \mathcal F_{s} \big\|_{k} 
\lesssim_{k} \|\xi\|_{\infty} + \big\| \|\delta I^{\mathbf X}\|_{p\text{-var};[s,T]} \big| \mathcal F_{s} \big\|_{k} + \big\|  \|\delta \mathcal I\|_{p\text{-var};[s,T]} \big| \mathcal F_{s} \big\|_{k}.
\end{equation}

{\textbf{Step 2.}} 
In this step, we estimate the terms involving $\|\delta I^{\mathbf X}\|_{p\text{-var}}$ and $\|\delta \mathcal I\|_{p\text{-var}}$ on the RHS of \eqref{e:M^Y <= Gamma} and \eqref{e:estim of Yn+1 (1)}. By Corollary~\ref{cor:int of the solution} (where $q<q'$ is required), we have for all $(s,t)\in\Delta$, (recall $q''/2 := qq'/(q+q')$) 
\begin{equation}\label{e:R^tildeY'}
\begin{aligned}
&\big\| \|R^{I^{\mathbf X}}\|_{\frac{p}{2}\text{-var};[s,t]} \big| \mathcal F_{s} \big\|_{k} + \big\| \|\delta I^{\mathbf X}\|_{p\text{-var};[s,t]} \big| \mathcal F_{s} \big\|_{k}\\
&\lesssim_{k,K}|\delta \mathbf X|_{p\text{-}\mathrm{var};[s,t]} \Big( 1 + \big\| \|\delta Y\|_{q\text{-var};[s,t]} \big| \mathcal F_{s}\big\|_{k} +  \|\delta M^{Y}\|^{\theta}_{\mathrm{BMO};[s,t]} + \big\| \|P^Y\|_{\frac{q''}{2}\text{-}\mathrm{var};[s,t]} \big| \mathcal F_{s} \big\|_{k}\Big).
\end{aligned}
\end{equation}
In addition, by Assumption~\ref{(f)} and the fact $|Y_r| \le \|\xi\|_{\infty} + \|\delta Y\|_{q\text{-}\mathrm{var};[s,T]}$ for $r\in[s,t]$, we get
\begin{align}\nonumber
\big\| \|\delta \mathcal I\|_{\frac{p}{2}\text{-}\mathrm{var};[s,t]} \big| \mathcal F_{s} \big\|_{k} & \le \Big\| \int_{s}^{t} |f(r,Y_r,Z_r)| dr \Big|\mathcal F_{s}\Big\|_{k} \\ \label{e:hat Gamma <= M}
& \le \Big\| \int_{s}^{t} |f(r,0,0)| dr \Big|\mathcal F_{s}\Big\|_{k} + \Big\| \int_{s}^{t} |f(r,Y_r,Z_r) - f(r,0,0)| dr \Big|\mathcal F_{s}\Big\|_{k}\\ \nonumber
&\lesssim_{k,K} (t - s) \Big(1 + \|\xi\|_{\infty} + \big\| \|\delta Y\|_{q\text{-}\mathrm{var};[s,T]} \big| \mathcal F_s \big\|_{k}\Big) + (t - s)^{\frac{1}{2}}\|\delta M^{Y}\|_{\mathrm{BMO};[s,t]}.
\end{align}

{\bf Step 3.} Now, we prove the desired bound \eqref{e:bound-Psi} for small $T$.
Plugging both \eqref{e:R^tildeY'} and \eqref{e:hat Gamma <= M} into \eqref{e:M^Y <= Gamma} and \eqref{e:estim of Yn+1 (1)}, 
noting that  $\|\delta M^Y\|_{\mathrm{BMO}} \le 1 + \|\delta M^{Y}\|^{\theta}_{\mathrm{BMO}}$ by Young's inequality, we get
\begin{equation}\label{e:bound of tilde Z}
\begin{aligned}
&\big\| \|\delta M^{Y}\|_{p\text{-var};[s,T]} \big| \mathcal F_{s}\big\|_{k} + \big\| \|\delta Y\|_{p\text{-var};[s,T]} \big| \mathcal F_{s}\big\|_{k} \\
&\lesssim_{k,K} 1 + \|\xi\|_{\infty} + \big(  |T-s|^{\frac{1}{2}} + |\delta \mathbf X|_{p\text{-}\mathrm{var};[s,T]} \big)\\
&\qquad \times \Big( 1 + \big\| \|\delta Y\|_{q\text{-var};[s,T]} \big| \mathcal F_{s} \big\|_{k} +  \|\delta M^{Y}\|^{\theta}_{\mathrm{BMO}} + \big\| \|P^Y\|_{\frac{q''}{2}\text{-}\mathrm{var};[s,T]} \big| \mathcal F_{s}\big\|_{k}\Big).
\end{aligned}
\end{equation}
Next, noting $\|\delta Y'\|_{p\text{-var}}=\|\delta H(Y)\|_{p\text{-var}} \le \|DH\|_\infty \|\delta Y\|_{p\text{-var}}$, we have 
\begin{equation*}
\|\|\delta Y'\|_{p\text{-var};[s,T]}|\mathcal F_{s}\|_{k}=\|\|\delta H(Y)\|_{p\text{-var};[s,T]}|\mathcal F_{s}\|_{k}\lesssim_{k,K} \text{RHS of \eqref{e:bound of tilde Z}}.
\end{equation*}
Moreover, by Remark~\ref{rem:Gamma-RSM}, we have that $(I^{\mathbf X},H(Y))\in\mathscr{D}^{(p,p)\text{-var}}_{\mathrm X}$ a.s. with $P^Y = - R^{I^{\mathbf X}} - \delta \mathcal I$. Hence, combining \eqref{e:R^tildeY'}, \eqref{e:hat Gamma <= M}, and \eqref{e:bound of tilde Z}, we get 
\begin{equation}\label{e:three in one}
\begin{aligned}
 \|H(Y)\|_{\infty}  &+ \big\| \|\delta H(Y)\|_{p\text{-var};[s,T]} \big| \mathcal F_{s} \big\|_{k} + \big\| \|P^Y\|_{\frac{p}{2}\text{-var};[s,T]} \big| \mathcal F_{s}\big\|_{k} \\
& + \big\| \|\delta M^{Y}\|_{p\text{-var};[s,T]} \big| \mathcal F_{s} \big\|_{k}\lesssim_{k,K} \text{RHS of \eqref{e:bound of tilde Z}}. 
\end{aligned}
\end{equation}
Taking the essential supremum over $[0,T]\times \Omega$ in \eqref{e:three in one}, the finiteness of $\thicknm{(Y,-H(Y))}_{\mathrm X;q,q';k,\infty} $ assumed in \eqref{e:cond that Psi < infty} 
yields that $\thicknm{(Y,-H(Y))}_{\mathrm X;p,p;k,\infty} < \infty$.

Furthermore, by \eqref{e:bound for Y p-var'} in Lemma~\ref{lem:Y <= R^Y'}, 
\begin{equation}\label{e:Y <= 1 + Psi}
\big\| \|\delta Y\|_{q\text{-var};[s,T]} \big| \mathcal F_{s} \big\|_{k} \le \big\| \|\delta Y\|_{p\text{-var};[s,T]} \big| \mathcal F_{s} \big\|_{k}\le (1 + \|\delta \mathrm X\|_{p\text{-}\mathrm{var};[s,T]}) \times \text{LHS of \eqref{e:three in one}}.
\end{equation}
Therefore, by \eqref{e:three in one} and \eqref{e:Y <= 1 + Psi}, noting $\big\| \|P^Y\|_{\frac{q''}{2}\text{-}\mathrm{var};[s,T]} \big| \mathcal F_{s}\big\|_{k} \le \big\| \|P^Y\|_{\frac{p}{2}\text{-}\mathrm{var};[s,T]} \big| \mathcal F_{s}\big\|_{k}$, we get
\begin{equation}\label{e:LHS of three in one}
\begin{aligned}
\text{LHS of \eqref{e:three in one}}&\lesssim_{k,K}  (1 + |\delta \mathbf X|_{p\text{-}\mathrm{var};[0,T]}) (1 + \|\xi\|_{\infty}) + (|T|^{\frac{1}{2}} +  |\delta \mathbf X|_{p\text{-}\mathrm{var};[0,T]})\|\delta M^{Y}\|^{\theta}_{\mathrm{BMO}}\\
&\qquad + (|T|^{\frac{1}{2}} + |\delta \mathbf X|_{p\text{-}\mathrm{var};[0,T]}) (1 + |\delta \mathbf X|_{p\text{-}\mathrm{var};[0,T]}) \times \text{LHS of \eqref{e:three in one}}.
\end{aligned}
\end{equation}
By taking essential supremum over $(s,\omega)\in [0,T]\times \Omega$ in \eqref{e:LHS of three in one}, and applying Lemma~\ref{lem:Lepingle},
\begin{align} \nonumber
\thicknm{(Y,-H(Y))}_{\mathrm X;p,p;k,\infty}
&\le C_{k,K} \Big((1 + |\delta \mathbf X|_{p\text{-}\mathrm{var};[0,T]}) (1 + \|\xi\|_{\infty}) + (|T|^{\frac{1}{2}} +  |\delta \mathbf X|_{p\text{-}\mathrm{var};[0,T]})\|\delta M^{Y}\|^{\theta}_{\mathrm{BMO}}\\ \label{Psi <= RHS}
&\quad + (|T|^{\frac{1}{2}} +  |\delta \mathbf X|_{p\text{-}\mathrm{var};[0,T]}) (1 + |\delta \mathbf X|_{p\text{-}\mathrm{var};[0,T]}) \,  \thicknm{(Y,-H(Y))}_{\mathrm X;p,p;k,\infty} \Big),
\end{align}
where $C_{k,K}>0$ is a constant independent of $\xi$. Denote $\varepsilon := \big(C_{k,K}(1 + |\delta \mathbf X|_{p\text{-}\mathrm{var};[0,T]}) \big)^{-1}$, and choose $T$ sufficiently small such that $ (|T|^{\frac{1}{2}} + |\delta\mathbf X|_{p\text{-}\mathrm{var};[0,T]}) \le \frac{1}{2} \varepsilon$. Then, it follows from \eqref{Psi <= RHS} that 
\begin{equation}\label{e:Psi <= M}
\thicknm{(Y,-H(Y))}_{\mathrm X;p,p;k,\infty}\le (2 C_{k,K} +1) (1 + \|\xi\|_{\infty}) + \|\delta M^{Y}\|^{\theta}_{\mathrm{BMO}}.
\end{equation}

{\bf Step 4.}
Now we establish an estimate for $\thicknm{(Y,-H(Y))}_{\mathrm X;p,p;k,\infty}$  on an arbitrary interval $[0,T]$. 
Let $C_{k,K}$ be the constant chosen in \eqref{Psi <= RHS}, and let $\varepsilon$ be defined as above. Since $\hat w(s,t) := |t-s| + \|\delta \mathrm X\|^p_{p\text{-}\mathrm{var};[s,t]} + \|\mathbb X\|^{p/2}_{p/2\text{-}\mathrm{var};[s,t]}$ defines a control function, it follows from  \cite[Lemma~1.5]{FrizZhang-2018} (see also \cite[Lemma~2.4]{li2024reflected}) that  there exists a partition  
$
\pi = \{[t_i,t_{i+1}];\ i=0,1,\dots,N-1\}\in\mathcal P_{[0,T]}
$
such that $\hat w(t_i,t_{i+1}) \le \min\{|\varepsilon/6|^{2},|\varepsilon/6|^{p},|\varepsilon/6|^{p/2}\}$ for all $i$. Consequently, $|t_{i+1} - t_{i}|^{\frac{1}{2}} + |\delta \mathbf X|_{p\text{-}\mathrm{var};[t_{i},t_{i+1}]} \le \varepsilon/2$. Moreover, by the superadditivity of the control function $\hat w$, the size $N$ of the partition is bounded by a constant depending only on $T$, $|\delta \mathbf X|_{p\text{-}\mathrm{var};[0,T]}$ and $\varepsilon$ (and hence only on $K$ and $\varepsilon$).

Then, by repeating the procedure in Step~1--Step~3 on each interval $[t_i,t_{i+1}]\in \pi$ in place of $[0,T]$, we obtain, as in \eqref{e:Psi <= M}, 
\begin{equation}\label{e:Psi <= M,[u,v]}
\thicknm{(Y,-H(Y))}_{\mathrm X;p,p;k,\infty;[t_{i},t_{i+1}]} \le (2 C_{k,K} +1) (1 + \|Y_{t_{i+1}}\|_{\infty}) + \|\delta M^{Y}\|^{\theta}_{\mathrm{BMO}}.
\end{equation}
For $\|Y_{t_{i+1}}\|_{\infty}$, by \eqref{e:Y q-var infty infty} in Lemma~\ref{lem:Y <= R^Y'}, there exists a constant $\hat C_{k,K} > 0$ independent of $i$ such that  
\begin{equation}\label{e:Y <= M}
\|Y_{t_{i+1}}\|_{\infty} \le \hat C_{k,K} \big(\|\xi\|_{\infty} + \thicknm{(Y,-H(Y))}_{\mathrm X;p,p;k,\infty;[t_{i+1},T]}\big).
\end{equation}
Combining \eqref{e:Psi <= M,[u,v]} and \eqref{e:Y <= M}, it holds, for some constant $\tilde{C}_{k,K} > 0$ independent of $i$, 
\begin{equation}\label{e:Psi <= 1 + xi + Psi + M^Y}
\begin{aligned}
&\thicknm{(Y,-H(Y))}_{\mathrm X;p,p;k,\infty;[t_{i},t_{i+1}]} \\
&\le  \tilde{C}_{k,K} (1 + \|\xi\|_{\infty}) + \tilde{C}_{k,K}\, \thicknm{(Y,-H(Y))}_{\mathrm X;p,p;k,\infty;[t_{i+1},T]} + \|\delta M^{Y}\|^{\theta}_{\mathrm{BMO}}.
\end{aligned}
\end{equation}
Furthermore, by Proposition~\ref{prop:Psi <= Psi + Psi}, we see that 
\begin{equation*}
\thicknm{(Y,-H(Y))}_{\mathrm X;p,p;k,\infty;[t_{i},T]} \lesssim_{K} \thicknm{(Y,-H(Y))}_{\mathrm X;p,p;k,\infty;[t_{i},t_{i+1}]} + \thicknm{ (Y,-H(Y))}_{\mathrm X;p,p; k,\infty;[t_{i+1},T]}.
\end{equation*} 
Then by \eqref{e:Psi <= 1 + xi + Psi + M^Y}, there exists a constant $\overline{C}_{k,K} > 0$ such that for $i=0,1,\dots, N - 1$,
\begin{equation}\label{e:iter 1}
\thicknm{(Y,-H(Y))}_{\mathrm X;p,p;k,\infty;[t_{i},T]} \le \overline{C}_{k,K}\big(1 + \|\xi\|_{\infty} + \thicknm{(Y,-H(Y))}_{\mathrm X;p,p;k,\infty;[t_{i+1},T]} + \|\delta M^{Y}\|^{\theta}_{\mathrm{BMO}} \big).
\end{equation}
In particular, for $i = N - 1$, this yields
\begin{equation*}
\thicknm{(Y,-H(Y))}_{\mathrm X;p,p;k,\infty;[t_{N - 1},T]} \le (2C_{k,K} + 1)(1+\|\xi\|_{\infty}) + \|\delta M^{Y}\|^{\theta}_{\mathrm{BMO}}.
\end{equation*}
The desired \eqref{e:bound-Psi} then follows by iterating \eqref{e:iter 1}.
\end{proof}

In the previous result, we show that the \textrm{RSM} seminorm  $\thicknm{(Y,-H(Y))}_{\mathrm X;p,p;k,\infty}$ is dominated by $1+\|\xi\|_\infty+\|\delta M^Y\|_{\text{BMO}}^{\theta}$ for some $\theta > 1$. Next in Proposition~\ref{prop:BMO and ess sup}, we show that  $\|\delta M^{Y}\|_{\mathrm{BMO}}$ can be bounded by the variant depending on $\|Y\|_\infty$ and system parameters. 

\begin{proposition}\label{prop:BMO and ess sup}
Let $(Y,Z)$ be a solution to Eq.~\eqref{e:BSDE} satisfying \eqref{e:cond that Psi < infty} with $k > 1$. Assume $\|Y\|_{\infty} + \|H\|_{C^{2}_b} + C_f + |\delta \mathbf X |_{p\text{-}\mathrm{var};[0,T]} + T \le \hat K$, where $\hat K > 0$ is a constant.
Then, there exists a constant $C_{k,\hat K} > 0$ depending only on $(k,\hat K)$ s.t.  
\begin{equation}\label{e:BMO <= 1}
\|\delta M^{Y}\|_{\mathrm{BMO}} \le C_{k,\hat K}.
\end{equation}
\end{proposition}

\begin{proof} 
Without loss of generality, one can choose a sufficiently larger $q'$ so that $q < q'$ and conditions \ref{(p,q,q')} and \eqref{e:cond that Psi < infty} still hold. By Proposition~\ref{prop:bounded by BMO} (where $q < q'$ is required due to the application of Corollary~\ref{cor:int of the solution}), we have 
$\thicknm{(Y,-H(Y))}_{\mathrm X;p,p;k,\infty}<\infty$. Therefore, we may assume further without loss of generality that $q>3$, hence in $q \in (3,4)$, and that $q'$ is chosen sufficiently close to $q$, such that \ref{(p,q,q')} and \eqref{e:cond that Psi < infty} hold. Note that this implies $\theta = {\frac{2}{q-2}} \in (1,2)$. Denote the process $I^{[\mathbf X]}_t := \mathrm{tr} ( \int_{0}^{t} H(Y_{r})H(Y_{r})^{\top} d[\mathbf X]_r)$. By the rough It\^o formula shown in \cite[Theorem~4.19]{dause2026controlled}, we have 
\begin{equation}\label{e:Ito}
\begin{aligned}
|Y_{t}|^2 + \|Z\|^2_{L^{2}([t,T])} &= |\xi|^2 + 2\int_{t}^{T} Y_r H(Y_r) d\mathbf X_r - 2\int_{t}^{T} Y_rZ_r dW_r \\
&\quad + 2\int_{t}^{T}Y_r f(r,Y_r,Z_r) dr - \delta 
I^{[\mathbf X]}_{t,T}. 
\end{aligned}
\end{equation}
Taking the conditional expectation $\E_t[\bm \cdot]$ on \eqref{e:Ito}, for each $t\in[0,T]$ we get 
\begin{equation}\label{e:Ito to Y with smooth X}
\begin{aligned}
\E_{t}\Big[\|Z\|^2_{L^{2}([t,T])}\Big] \le \E_{t}\Big[|\xi|^2 &+ 2\int_{t}^{T} Y_r H(Y_r) d\mathbf X_r + 2\int_{t}^{T}Y_r f(r,Y_r,Z_r) dr - \delta 
I^{[\mathbf X]}_{t,T}\Big].
\end{aligned}
\end{equation}

Let $\kappa (\bm\cdot)$ be a smooth cut-off function s.t. $\kappa(y) = 1$ if $|y| \le \hat K$ and $\kappa(y) = 0$ if $|y| > \hat K + 1$. Denote 
\begin{equation*}
\tilde{H}(y) := \kappa(y) y\cdot  H(y),\quad \tilde{f}(t,y,z) := \kappa(y)y \cdot f(t,\kappa(y)y,z).
\end{equation*}
Then, by the condition $\|H\|_{C^{2}_{b}}<\infty$ and Assumption~\ref{(f)}, it follows for all $t\in[0,T]$, 
\begin{equation}\label{e:for tilde f}
\|\tilde{H}\|_{C^{2}_{b}}\lesssim_{\hat K} 1, \quad |\tilde{f}(t,y,z)| \lesssim_{\hat K} 1 + |z|.
\end{equation}
Furthermore, since $\hat K$ is an upper bound of $\|Y\|_{\infty}$, the inequality \eqref{e:Ito to Y with smooth X} leads to:
\begin{equation}\label{e:Ito to Y with smooth X'}
\begin{aligned}
\E_{t}\left[\delta \langle M^{Y} \rangle_{t,T}\right] = \E_{t}\Big[\|Z\|^2_{L^{2}([t,T])}\Big] \le \|\xi \|^{2}_{\infty} + 2\E_{t} \Big[ \delta (\tilde{I}^{\mathbf X} + \tilde{\mathcal I})_{t,T} - \delta 
I^{[\mathbf X]}_{t,T} \Big],
\end{aligned}
\end{equation}
where 
\begin{equation*}
\tilde{I}^{\mathbf X}_{t} := \int_{0}^{t} \tilde{H}(Y_{r})d\mathbf X_{r},\quad \tilde{\mathcal I}_{t} := \int_{0}^{t} \tilde{f}
(r,Y_r,Z_r) dr,\quad t\in[0,T].
\end{equation*}

To estimate $\tilde{I}^{\mathbf X}$, using Corollary~\ref{cor:int of the solution}, we have for $t\in[0,T]$,
\begin{equation}\label{e:1st estim for tilde Gamma}
\| \delta \tilde{I}^{\mathbf X}_{t,T} | \mathcal F_{t} \|_{k} \lesssim_{k,\hat K} |\delta \mathbf X|_{p\text{-}\mathrm{var}} \Big(1 + \big\| \|\delta Y\|_{q\text{-var};[t,T]} \big| \mathcal F_{t} \big\|_{k} +  \|\delta M^{Y}\|^{\theta}_{\mathrm{BMO}} + \big\| \|P^Y\|_{\frac{q''}{2}\text{-}\mathrm{var};[t,T]} \big| \mathcal F_{t}\big\|_{k} \Big).
\end{equation}
In addition, by \eqref{e:Y q-var infty infty} in Lemma~\ref{lem:Y <= R^Y'} and Proposition~\ref{prop:bounded by BMO} (where it requires $k > 1$), we have 
\begin{equation}\label{e:bound for Y p-var}
\begin{aligned}
\big\| \|\delta Y\|_{q\text{-var};[t,T]} \big| \mathcal F_{t} \big\|_{k}  &\lesssim_{k,\hat K} 1 + \thicknm{(Y,-H(Y))}_{\mathrm X;q,q';k,\infty} \lesssim_{k,\hat K} 1 + \|\delta M^{Y}\|^{\theta}_{\mathrm{BMO}}. 
\end{aligned}
\end{equation}
Moreover, by Proposition~\ref{prop:bounded by BMO} again, we also have 
\begin{equation}\label{e:R^{J^Y}}
\big\| \|P^Y\|_{\frac{q''}{2}\text{-}\mathrm{var};[t,T]} \big| \mathcal F_{t}\big\|_{k} \lesssim_{k,\hat K} 1 + \|\delta M^{Y}\|^{\theta}_{\mathrm{BMO}}.
\end{equation}
Hence, combining \eqref{e:1st estim for tilde Gamma}, \eqref{e:bound for Y p-var}, and \eqref{e:R^{J^Y}}, it follows that 
\begin{equation}\label{e:tilde H <= Z}
\big\| \delta \tilde{I}^{\mathbf X}_{t,T} \big| \mathcal F_{t} \big\|_{k} \lesssim_{k,\hat K} 1  +  \|\delta M^{Y}\|^{\theta}_{\mathrm{BMO}}.
\end{equation}

On the other hand, for the term $\tilde{\mathcal I}$, by  \eqref{e:for tilde f} and H\"older's inequality, we obtain that  
\begin{equation}\label{e:hat Gamma <= Z}
\begin{aligned}
\big\| \delta \tilde{\mathcal I}_{t,T} \big| \mathcal F_{t} \big\|_k = \Big\| \int_{t}^{T} \tilde{f}(r,Y_r,Z_r) dr \Big|\mathcal F_{t}\Big\|_{k}  \lesssim_{k,\hat K} 1 + \Big\| \|Z\|_{L^{2}([t,T])} \Big|\mathcal F_{t}\Big\|_{k} \lesssim_k 1 + \|\delta M^{Y}\|_{\mathrm{BMO}}.
\end{aligned}
\end{equation}

In addition, by the estimate for Young integrals (see, e.g., \cite[Theorem~5]{diehl2017young}) and \eqref{e:bound for Y p-var}, 
\begin{equation}\label{e:bound for [X] term}
\begin{aligned}
\delta 
I^{[\mathbf X]}_{t,T}
&\lesssim_{\hat K} \big(1 + \E_{t}[\|\delta Y\|_{q\text{-var};[t,T]}] \big) \|\delta [\mathbf X]\|_{\frac{p}{2}\text{-var};[t,T]} \lesssim_{k,\hat K} 1 + \|\delta M^{Y}\|^{\theta}_{\mathrm{BMO}}.
\end{aligned}
\end{equation}

Combining \eqref{e:Ito to Y with smooth X'}, \eqref{e:tilde H <= Z}, \eqref{e:hat Gamma <= Z}, and \eqref{e:bound for [X] term}, 
we see that, with $\theta < 2$,
\begin{equation*}
\|\delta M^{Y}\|^2_{\mathrm{BMO}}  = \esssup_{(t,\omega)\in[0,T]\times\Omega} \E_{t}\left[\delta \langle M^{Y} \rangle_{t,T}\right] \lesssim_{k,\hat K}  1 + \|\delta M^{Y}\|^{\theta}
_{\mathrm{BMO}},
\end{equation*}
and then obtain \eqref{e:BMO <= 1} via
Young's inequality.
\end{proof}

The following Theorem~\ref{thm:BMO bound} gives an estimate for $\thicknm{(Y,-H(Y))}_{\mathrm X;p,p;k,\infty}$. Recall that $\Theta$ is defined in \eqref{e:Theta}, and define $\Theta_2 := (\Theta, C_{f}, R_2)$. 

\begin{theorem}\label{thm:BMO bound} 
Assume $H\in C^{\gamma}_b$ for some $\gamma > p$, and assume further that  \ref{(fdet)} holds.
Suppose that for some constant $R_2 > 0$,
\begin{equation*}
\|\xi\|_{\infty} + \|H\|_{C^{2}_b} + |\delta \mathbf X |_{p\text{-}\mathrm{var};[0,T]} + T \le R_2.
\end{equation*}
Let $(Y,Z)$ be a solution to Eq.~\eqref{e:BSDE} satisfying \eqref{e:cond that Psi < infty} with $k > 1$. 
Then, there is a constant $C_{\gamma,\Theta_2} > 0$ depending only on $(\gamma,\Theta_2)$ such that  
\begin{equation}\label{e:rough BSDE bound}
\left\|Y\right\|_{\infty} + \|\delta M^{Y}\|_{\mathrm{BMO}} \le C_{\gamma,\Theta_2}.
\end{equation}
Moreover, we have that for some constant $C_{k,\gamma,\Theta_2} > 0$ depending only on $(k,\gamma,\Theta_2)$, 
\begin{equation}\label{e:bound (Y,-H(Y))}
\thicknm{(Y,-H(Y))}_{\mathrm X;p,p;k,\infty} \le C_{k,\gamma,\Theta_2}.
\end{equation}
\end{theorem}

\begin{proof}
Without loss of generality, one can choose a sufficiently larger $q'$ so that $q < q'$ and conditions \ref{(p,q,q')} and \eqref{e:cond that Psi < infty} still hold. 

We first establish an upper bound for $\|Y\|_{\infty}$. Consider the following (deterministic) RDEs:
\begin{equation*}
\begin{cases}\displaystyle
\overline Y_{t} = \overline{C} + \int_{t}^{T}f(r,\overline Y_r,0)dr + \int_{t}^{T} H(\overline Y_r) d\mathbf X_{r}\\ \displaystyle
\underline{Y}_{t} = \underline{C} + \int_{t}^{T}f(r,\underline{Y}_r,0)dr + \int_{t}^{T} H(\underline{Y}_r) d\mathbf X_{r},
\end{cases}
\end{equation*}
where  $\overline{C} := \|\xi\|_{\infty}$ and $\underline{C} := -\|\xi\|_{\infty}$.
Noting $H \in C^{\gamma}_b$, by Proposition~\ref{prop:RDE with drift}, the above two equations admit unique solutions $\overline Y$ and $\underline{Y}$, respectively, such that $(\overline Y,-H(\overline Y)),\,(\underline{Y},-H(\underline{Y})) \in \mathscr{D}^{(p,p)\text{-}\mathrm{var}}_{\mathrm X}$.
Since $\overline Y$ (resp. $\underline{Y}$) is deterministic, $(\overline Y,0)$ (resp. $(\underline{Y},0)$) uniquely solves BSDE~\eqref{e:BSDE} with $\xi$ replaced by $\overline{C}$ (resp. $\underline{C}$). Thus, we have 
\begin{equation*}
\begin{aligned}
&\thicknm{(\overline Y,-H(\overline Y))}_{\mathrm X;p,p;k,\infty} + \thicknm{(\underline{Y},-H(\underline{Y}))}_{\mathrm X;p,p;k,\infty} \\
&\le \|(\overline Y,-H(\overline Y))\|_{\mathscr{D}^{(p,p)\text{-}\mathrm{var}}_{\mathrm X}} + \|(\underline{Y},-H(\underline{Y}))\|_{\mathscr{D}^{(p,p)\text{-}\mathrm{var}}_{\mathrm X}} + 2\|H(\bm\cdot)\|_{\infty} < \infty. 
\end{aligned}
\end{equation*}
Therefore, we can apply Theorem \ref{thm:comparison}, noting $\underline{C} \le \xi \le \overline{C}$, to get
\begin{equation}\label{e:Y <= Y <= Y} 
\underline{Y}_t \le Y_t \le \overline Y_t.
\end{equation}
In addition, by the \emph{a priori} estimate for RDEs shown in Proposition~\ref{prop:RDE with drift}, we have 
\begin{equation*}
|\overline Y_t| \vee |\underline{Y}_t| \le C_{\gamma,\Theta_2}, \text{ for all } t\in[0,T],
\end{equation*}
for some constant $C_{\gamma,\Theta_2}$. This together with \eqref{e:Y <= Y <= Y} yields $\|Y\|_\infty\le C_{\gamma,\Theta_2}$. 

Finally, by Propositions~\ref{prop:BMO and ess sup} and then \ref{prop:bounded by BMO}  (noting $k > 1$ and $q < q'$ are required), the boundedness of $\|Y\|_{\infty}$ implies \eqref{e:rough BSDE bound} and \eqref{e:bound (Y,-H(Y))}.
\end{proof}

\begin{remark}\label{rem:deterministic}
In the above theorem, the function $f$ can be random, provided that it is uniformly bounded from above and below by two deterministic uniformly Lipschitz continuous functions. In contrast, Theorem~\ref{thm:comparison} does not permit a comparison in the coefficient $H(\bm\cdot)$, which serves as the integrand in the rough integral. Consequently, it is nontrivial to extend  $H(\bm\cdot)$ to a stochastic (controlled) field, as the forward case considered in \cite{fhl21,allan2024rough}.
\end{remark}

\begin{remark}\label{rem:Why 1-d}
In this paper, we focus on the $1$-dimensional setting, but see Theorem~\ref{thm:multi-d conditional} for a conditional multidimensional result. The main reason is that the \emph{a priori} estimate (of Theorem~\ref{thm:BMO bound}) relies heavily on the comparison theorem (Theorem~\ref{thm:comparison}) which is not applicable to multi-dimensional cases in general.\footnote{The situation is similar in viscosity PDE theory.}
To see this, Proposition~\ref{prop:bounded by BMO} and \ref{prop:BMO and ess sup} show that once the essential supremum of the modulus of the solution $Y$ is bounded, the \emph{RSM} seminorm is bounded as well. However, since $Z$ is also unknown in the equation, it is difficult to obtain \emph{a priori} estimates for $Y$ using the classical \emph{RDE} approach (e.g., \cite[Proposition~8.2]{friz2020course}). We refer to Liang and Tang \cite{liang2025multidimensional} for some special multidimensional cases via Doss--Sussmann transformation.
\end{remark}

Proposition~\ref{prop:classical BSDEs} shows that the solutions to classical BSDEs satisfy condition~\eqref{e:cond that Psi < infty}.

\begin{proposition}\label{prop:classical BSDEs}
Let $\mathbf X \in \mathscr{C}^{p\text{-}\mathrm{var}}([0,T],\R^d)$ be the canonical lift of a smooth path $\mathrm X: [0,T]\to \R^d$. Suppose that $\xi \in L^{\infty}$, and that \ref{(f)} holds. 
Then, there exists a unique solution to Eq.~\eqref{e:BSDE} such that for any $k\ge 1$,
\begin{equation}\label{e:quadratic integrability}
\thicknm{(Y,-H(Y))}_{\mathrm X;p,p;k,\infty}< \infty.
\end{equation}
\end{proposition}

\begin{proof}
The well-posedness of Eq.~\eqref{e:BSDE} follows from the classical BSDE theory. Indeed, by \cite[Theorem~7.3.3]{zhang2017backward}, there exists a unique pair $(Y,Z)$ such that $Y$ is continuous, adapted and bounded, $Z$ is progressively measurable and quadratically integrable, and the equation \eqref{e:BSDE} holds. Moreover, clearly $Y$ is a semimartingale, and the rough stochastic integral $\int H(Y_r) d\mathbf X_r$ coincides with the Lebesgue integral $\int H(Y_r) \partial_{r}\mathrm X_r dr$. Thus, $(Y,-H(Y)) \in \mathrm{RSM}^{(p,p)\text{-}\mathrm{var}}_{\mathrm X}$, and $(Y,Z)$ solves Eq.~\eqref{e:BSDE} in the sense of Definition~\ref{def:def of BSDEs}. 

Furthermore, by \cite[Theorem~7.2.1]{zhang2017backward}, the process $M^{Y}_t = \int_{0}^{t}Z_r dW_r$ is a \textrm{BMO} martingale. Let $I^{\mathbf X}, \mathcal I$ be defined by \eqref{e:def of Gamma}. Then, for $(s,t)\in\Delta$, we have $P^{Y}_{s,t} = - \delta (I^{\mathbf X} + \mathcal I)_{s,t} + H(Y_{s}) \delta \mathrm X_{s,t}$, and hence
\begin{equation*}
\begin{aligned}
\Big\| \|P^{Y}\|_{\frac{p}{2}\text{-}\mathrm{var}} \Big\|_{\infty;k,\infty} &\le \sup_{s\in[0,T]} \Big\| \int_{s}^{T} |H(Y_r) - H(Y_{s})| \cdot |\partial_{r}  \mathrm X_{r}| dr + \int_{s}^{T} |f(r,Y_{r},Z_{r})| dr \Big| \mathcal F_{s} \Big\|_{k,\infty} \\
&\lesssim \|H\|_{\infty} \sup_{r\in[0,T]}|\partial_{r}\mathrm X_r| T + C_{f} (T + \|Y\|_{\infty} T + \|\delta M^Y\|_{\mathrm{BMO}} T^{\frac{1}{2}}) < \infty.
\end{aligned}
\end{equation*}
Similarly, it also holds $\|\|\delta(I^{\mathbf X} + \mathcal I)\|_{p\text{-}\mathrm{var}}\|_{\infty;k,\infty} < \infty$. Noting that $\delta Y_{s,t} = - \delta I^{\mathbf X}_{s,t}  - \delta \mathcal I_{s,t} + \delta M^{Y}_{s,t}$, and that $\|\|\delta M^Y\|_{p\text{-}\mathrm{var}}\|_{\infty;k,\infty}\lesssim_k \|\delta M^Y\|_{\mathrm{BMO}} < \infty$ (see Lemma~\ref{lem:Lepingle}), we get $\|\|\delta Y\|_{p\text{-}\mathrm{var}}\|_{\infty;k,\infty} < \infty$, and hence
$
\|\|\delta H(Y)\|_{p\text{-}\mathrm{var}}\|_{\infty;k,\infty} < \infty.
$
Consequently, the desired \eqref{e:quadratic integrability} is obtained, which, combined with Corollary~\ref{cor:uniqueness}, also implies that $(Y,Z)$ is the unique solution to Eq.~\eqref{e:BSDE}. 
\end{proof}

\subsection{Well-posedness of Rough BSDEs}\label{subsec:Rough BSDE with low regularity H} 

In this section, we establish the well-posedness of Eq.~\eqref{e:BSDE} in the sense of Definition~\ref{def:def of BSDEs}. We first define the solution map, and then prove the invariance and the contraction of the solution map on small intervals. Next, a solution is constructed via Picard's iteration, and then the local well-posedness together with the \emph{a priori} estimate of solutions, allows us to extend the unique solution globally. After proving the well-posedness of the solution, we establish the stability of the solution with respect to coefficients. Throughout this subsection, we assume \ref{(gamma)}. 

To facilitate the statement of Lemma~\ref{lem:solu map}, let $R > 0$ be a constant such that 
\begin{equation*}
\|H\|_{C^{\gamma - 1}_b} + |\delta \mathbf X|_{p\text{-}\mathrm{var};[0,T]} + T \le R,
\end{equation*}
and denote $\bar \Theta := (\Theta,C_f,R)$. 

\begin{lemma}\label{lem:solu map}
Let $(s,t) \in \Delta$ and $k \ge 1$. Assume \ref{(gamma)}, \ref{(f)}, $\mathbf X = (\mathrm X,\mathbb X) \in \mathscr C^{p\text{-}\mathrm{var}}([s,t],\R^d)$, $\xi \in L^{\infty}_{t}$, and $H \in C^{\gamma - 1}_b$. Let $(Y,Y') \in \mathrm{RSM}^{(p,p)\text{-}\mathrm{var}}_{\mathrm X}([s,t],\R)$, $Z$ be a progressively measurable process with $\int_{s}^{t} | Z_r |^2 dr < \infty$ a.s. Assume $\|Y_t\|_{\infty} < \infty$, $\thicknm{(Y,Y')}_{\mathrm X;q,q';k,\infty;[s,t]} < \infty$, and $\delta M^{Y}_{s,t} = \int_{s}^{t} Z_r dW_r$. 

Then, there exists a unique triplet $(\tilde Y,-H(Y),\tilde Z)$ such that $(\tilde Y,- H (Y)) \in \mathrm{RSM}^{p,p}_{\mathrm X}([s,t],\R)$, and $\tilde Z$ is a progressively measurable, $\int_{s}^{t} | \tilde Z_r |^2 dr < \infty$ a.s., $\thicknm{(\tilde Y,- H (Y))}_{\mathrm X;q,q';k,\infty;[s,t]} < \infty$, $\delta M^{\tilde Y}_{s,t} = \int_{s}^t \tilde Z_r dW_r$, and 
\begin{equation}\label{e:solu map} 
\tilde Y_v = \xi + \int_{v}^{t} f(r,Y_r,Z_r) dr + \int_v^t H(Y_r) d\mathbf X_r - \int_v^t \tilde Z_r dW_r,\quad v \in [s,t]. 
\end{equation}
Moreover, we have 
\begin{equation}\label{e:bound for solution map}
\begin{aligned}
&\thicknm{(\tilde Y,- H (Y))}_{\mathrm X;q,q';k,\infty;[s,t]} \lesssim_{k,\gamma,\bar \Theta} 1 + \|\xi\|_{\infty} + \|Y_t\|_{\infty} \\
&\qquad + \Big(|\delta \mathbf X|_{p\text{-}\mathrm{var};[s,t]} + (t - s)^{\frac{1}{2}}\Big) \thicknm{( Y,Y')}^{\gamma - 1}_{\mathrm X;q,q';k,\infty;[s,t]} + \thicknm{( Y,Y')}^{\,p/q'}_{\mathrm X;q,q';k,\infty;[s,t]}. 
\end{aligned}
\end{equation}
\end{lemma}

\begin{proof}
For simplicity, we denote $I^{\mathbf X}_v := \int_s^v H(Y_r) d\mathbf X_r$ and $\mathcal I_v := \int_{s}^v f(r,Y_r,Z_r) dr$, $v \in [s,t]$. 	

By Lemma~\ref{lem:BMO of q-var}, and by conditions $\thicknm{(Y,Y')}_{\mathrm X;q,q';k,\infty;[s,t]} < \infty$ and $\xi \in L^{\infty}_{t}$,  
\begin{equation}\label{e:xi Gamma Gamma}
\begin{aligned}
&\left\|\xi\right\|_2 + \Big\| \big\| \delta \mathcal I \big\|_{1\text{-}\mathrm{var};[s,t]} + \big\|\delta I^{\mathbf X} \big\|_{p\text{-var};[s,t]} \Big\|_{2} \\
&\lesssim_k \|\xi\|_{\infty} + \big\| \|\delta \mathcal I\|_{1\text{-}\mathrm{var}} \big\|_{\infty;k,\infty;[s,t]} + \big\| \| \delta I^{\mathbf X}\|_{p\text{-}\mathrm{var}} \big\|_{\infty;k,\infty;[s,t]} \\ 
&\lesssim_{k,\gamma,\bar \Theta} \|\xi\|_{\infty} + (|t - s| + |t - s|^{\frac{1}{2}}) \left(1 + \|Y\|_{\mathcal L^{\infty}([s,t] \times \Omega)} + \|\delta M^Y\|_{\mathrm{BMO};[s,t]} \right) \\[1mm]
&\quad\quad +|\delta \mathbf X|_{p\text{-}\mathrm{var};[s,t]} \left( 1 + \thicknm{(Y,Y')}^{\gamma - 1}_{\mathrm X;q,q';k,\infty;[s,t]} \right)\\
&\lesssim_{\bar \Theta} \|\xi\|_{\infty}  + \|Y_t\|_{\infty}+ (|\delta \mathbf X|_{p\text{-}\mathrm{var};[s,t]} + |t - s|^{\frac{1}{2}}) \left( 1 + \thicknm{(Y,Y')}^{\gamma - 1}_{\mathrm X;q,q';k,\infty;[s,t]} \right),
\end{aligned}
\end{equation}
where we apply the following inequalities,
\begin{equation}\label{e:mathcal I <=}
\begin{aligned}		
\big\|\delta \mathcal I \big\|_{1\text{-}\mathrm{var};[v,t]} &\le C_f \Big|\int_{v}^{t} ( 1+|Y_{r}|+|Z_{r}| )dr\Big|\\
&\lesssim_{\bar \Theta} |t - v|\Big(1 + |Y_t| + \|Y\|_{p\text{-}\mathrm{var};[v,t]}\Big) + |t -  v|^{\frac{1}{2}} \Big|\int_{v}^{t}|Z_{r}|^{2}dr\Big|^{\frac{1}{2}},
\end{aligned}
\end{equation}
and by \eqref{e:Psi alpha'} in Lemma~\ref{lem:int of RSM'} (letting $p' = pq'/q$ therein so that $\gamma - 2 \ge p/p'$), 
\begin{equation*}
\begin{aligned}
\big\| \big\|\delta I^{\mathbf X} \big\|_{p\text{-var}} \big\|_{\infty;k,\infty;[s,t]} 
\lesssim_{k,\gamma,\bar \Theta} |\delta \mathbf X|_{p\text{-}\mathrm{var};[s,t]} \left( 1 + \thicknm{(Y,Y')}^{\gamma - 1}_{\mathrm X;q,q';k,\infty;[s,t]} \right). 
\end{aligned}
\end{equation*}
Hence, since $\|\xi\|_2 + \big\|\delta (\mathcal I +  I^{\mathbf X})_{v,t} \big\|_2 \le \|\xi\|_2 + \big\| \|\delta \mathcal I\|_{1\text{-}\mathrm{var};[s,t]} + \|\delta I^{\mathbf X}\|_{p\text{-}\mathrm{var};[s,t]} \big\|_2$, \eqref{e:xi Gamma Gamma} yields 
\begin{equation*}
\|\xi\|_2 + \sup_{v \in [s,t]}\big\|\delta \mathcal I_{v,t} + \delta I^{\mathbf X}_{v,t} \big\|_2 < \infty.
\end{equation*}
Then, by the martingale representation theorem for $L^2$-integrable random variable (see, e.g., \cite[Problem~4.17]{karatzas1991brownian}), there exists a unique progressively measurable
process $\tilde Z$ such that a.s., 
\begin{align*}
\int_v^t \tilde Z_{r} dW_r = \xi + \delta \mathcal I_{s,t} + \delta I^{\mathbf X}_{s,t} - \mathbb{E}_v\left[\xi + \delta \mathcal I_{s,t} + \delta I^{\mathbf X}_{s,t}\right]\text{ for all }v\in[s,t].
\end{align*}
Let 
\begin{equation}\label{e:tilde-Y}
\tilde{Y}_{v} = \mathbb{E}_{v}\left[\xi + \delta \mathcal I_{v,t} + \delta I^{\mathbf X}_{v,t} \right],~ v \in [s,t].
\end{equation}
Then $(\tilde Y, \tilde Z)$ satisfies \eqref{e:solu map}, and hence as explained under Definition~\ref{def:def of BSDEs}, one can directly check that $(\tilde Y, -H(Y)) \in \mathrm{RSM}^{(p,p)\text{-}\mathrm{var}}_{\mathrm X}([s,t],\R)$ with $\delta M^{\tilde Y}_{s,t} = \int_{s}^t \tilde Z_r dW_r$. 

Now we prove \eqref{e:bound for solution map}. By the lower bound of the BDG inequality and the estimate \eqref{e:xi Gamma Gamma}, we get 
\begin{align}\nonumber
\|\delta M^{\tilde Y}\|_{\mathrm{BMO};[s,t]} &\lesssim \sup_{r \in [s,t]} \big\|  \sup_{v\in[r,t]} |\delta M^{\tilde Y}_{r,v}| \big|\mathcal F_{r} \big\|_{2,\infty} \\ \label{e:M^Y BMO <= sup}
&\lesssim_k \|\xi\|_{\infty} + \big\| \|\delta \mathcal I\|_{1\text{-}\mathrm{var}} \big\|_{\infty;k,\infty;[s,t]} + \big\| \| \delta I^{\mathbf X}\|_{p\text{-}\mathrm{var}} \big\|_{\infty;k,\infty;[s,t]} \\ \nonumber
&\lesssim_{k,\gamma,\bar \Theta} 
\|\xi\|_{\infty} +  \|Y_t\|_{\infty}
+ (|\delta \mathbf X|_{p\text{-}\mathrm{var};[s,t]} + |t - s|^{\frac{1}{2}}) \left( 1 + \thicknm{(Y,Y')}^{\gamma - 1}_{\mathrm X;q,q';k,\infty;[s,t]} \right), 
\end{align}
where the second inequality is due to Doob's maximal inequality so that  
\begin{equation*}
\big\|\sup_{v \in [r,t]} |\delta M^{\tilde Y}_{r,v}| \big| \mathcal F_r \big\|_{2,\infty} \lesssim \big\| \delta M^{\tilde Y}_{r,t} \big| \mathcal F_r \big\|_{2,\infty} = \big\|\delta (\tilde Y + \mathcal I + I^{\mathbf X})_{r,t} \big| \mathcal F_r \big\|_{2,\infty} 
\end{equation*}
and by the equation $\delta \tilde Y_{r,t} = \xi - \E_{r}[\xi] - \E_{r}[\delta (\mathcal I + I^{\mathbf X})_{r,t}]$ and Lemma~\ref{lem:BMO of q-var} so that 
\begin{equation*}
\big\|\delta (\tilde Y + \mathcal I + I^{\mathbf X})_{r,t} \big| \mathcal F_r \big\|_{2,\infty} \lesssim_k \|\xi\|_{\infty} + \big\| \|\delta \mathcal I\|_{1\text{-}\mathrm{var}} \big\|_{\infty;k,\infty;[s,t]} + \big\| \| \delta I^{\mathbf X}\|_{p\text{-}\mathrm{var}} \big\|_{\infty;k,\infty;[s,t]}.
\end{equation*}

Furthermore, by Proposition~\ref{prop:estim of integral} and Lemma~\ref{lem:int of RSM'}, 
\begin{equation*}
\begin{aligned}
\big\| \big\|R^{I^{\mathbf X}} \big\|_{\frac{p}{2}\text{-var}} \big\|_{\infty;k,\infty;[s,t]} \lesssim_{k,\gamma,\bar \Theta} |\delta \mathbf X|_{p\text{-}\mathrm{var};[s,t]} \left( 1 + \thicknm{(Y,Y')}^{\gamma - 1}_{\mathrm X;q,q';k,\infty;[s,t]} \right). 
\end{aligned}
\end{equation*}
Consequently, noting $P^{\tilde Y} = - R^{I^{\mathbf X}} - \delta \mathcal I$, the above estimate together with \eqref{e:mathcal I <=} implies   
\begin{equation}\label{e:R^J tilde Y}
\big\| \big\|P^{\tilde Y} \big\|_{\frac{p}{2}\text{-var}} \big\|_{\infty;k,\infty;[s,t]} \lesssim_{k,\gamma,\bar \Theta} (|\delta \mathbf X|_{p\text{-}\mathrm{var};[s,t]} + (t-s)^{\frac{1}{2}}) \left( 1 +  \|Y_t\|_{\infty}  + \thicknm{(Y,Y')}^{\gamma - 1}_{\mathrm X;q,q';k,\infty;[s,t]} \right).  
\end{equation}

Moreover, since $\| \|Y\|_{p\text{-}\mathrm{var}} \|_{\infty;k,\infty;[s,t]}$ is dominated by $\thicknm{(Y,Y')}_{\mathrm X;q,q';k,\infty;[s,t]}$ (see Lemma~\ref{lem:Y <= R^Y'}), by the standard interpolation argument (noting $p \le q'$) we get  
\begin{equation}\label{e:H Y + H Y}
\begin{aligned}
\|H(Y)\|_{\infty} + \big\| \|H(Y)\|_{q'\text{-}\mathrm{var}} \big\|_{\infty;k,\infty;[s,t]} &\lesssim \|H\|_{\infty} + \|H\|^{(q'-p)/q'}_{\infty} \|D H\|^{p/q'}_{\infty} \big\| \|Y\|_{p\text{-}\mathrm{var}} \big\|^{p/q'}_{\infty;k,\infty;[s,t]}\\
&\lesssim_{\bar \Theta} 1 + \thicknm{(Y,Y')}^{\,p/q'}_{\mathrm X;q,q';k,\infty;[s,t]}. 
\end{aligned}
\end{equation}
Then the desired estimate \eqref{e:bound for solution map} follows by combining \eqref{e:M^Y BMO <= sup}, \eqref{e:R^J tilde Y}, and \eqref{e:H Y + H Y}. 

Finally, Eq.~\eqref{e:solu map} yields that $\tilde Y_v$ must be given by \eqref{e:tilde-Y}, and the uniqueness follows.
\end{proof}

The solution map $\Phi^{\xi}_{[s,t]}$ then can be well-defined as follows. 

\begin{definition}\label{def:solu map}
Under the same assumptions as in Lemma~\ref{lem:solu map}, define the solution map 
\begin{equation*}
\Phi^{\xi}_{[s,t]}(Y,Y',Z) := (\tilde Y,-H(Y),\tilde Z). 
\end{equation*}
\end{definition}

The following Proposition~\ref{prop:invariant within ball} shows that the solution  $\Phi^{\xi}_{[s,t]}(\bm\cdot)$ is  invariant within the ball with respect to the seminorm $\thicknm{\bm\cdot}_{\mathrm X;q,q';k,\infty;[s,t]}$ when $|\delta \mathbf X|_{p\text{-}\mathrm{var};[s,t]} + (t-s)^{1/2}$ is smaller than a specific constant $\theta$. In particular, $\theta$ can be chosen uniformly with respect to the upper bound of $\xi$.

\begin{proposition}\label{prop:invariant within ball}
Assume that $(Y,Y',Z)$ satisfies the same condition as in Lemma~\ref{lem:solu map}. Assume further that $p < q'$. Denote by $(\tilde Y,-H(Y),\tilde Z) = \Phi^{\xi}_{[s,t]}(Y,Y',Z)$. Let $\bar \Theta$ be the same as in Lemma~\ref{lem:solu map}. 
Then, there exists $K > 0$ such that for every $K' \ge K$, there exists $\theta > 0$ such that if $|\delta \mathbf X|_{p\text{-}\mathrm{var};[s,t]} + (t-s)^{1/2} \le \theta$ and $\,\thicknm{(Y,Y')}_{\mathrm X;q,q';k,\infty;[s,t]} \le K'$, it holds that 
\begin{equation}\label{e:<= K'} 
\thicknm{(\tilde Y,-H(Y))}_{\mathrm X;q,q';k,\infty;[s,t]} \le K'.
\end{equation} 
In particular, we can choose $K := C(1 + \|\xi\|_{\infty})$ and $\theta := 
\frac{1}{2} 
C^{-1} | K' |^{-(\gamma - 2)}$, where $C$ is a constant depending only on $(k,\gamma,\bar \Theta)$. 
\end{proposition}

\begin{proof}
Noting $p < q'$, for every $\rho > 0$  we have 
\begin{equation*}
\thicknm{(Y,Y')}_{\mathrm X;q,q';k,\infty;[s,t]}^{\,p/q'} \lesssim_{\rho} 1 + \rho\, \thicknm{(Y,Y')}_{\mathrm X;q,q';k,\infty;[s,t]}.
\end{equation*}
Thus, by choosing $\rho$ sufficiently small, one can apply \eqref{e:bound for solution map} in Lemma~\ref{lem:solu map} to determine a constant $C_1 > 0$ such that 
\begin{equation*}
\begin{aligned}
\thicknm{(\tilde Y,- H (Y))}_{\mathrm X;q,q';k,\infty;[s,t]} \le C_1 \Big( 1 + \|\xi\|_{\infty} + \big(|\delta \mathbf X|_{p\text{-}\mathrm{var};[s,t]} + (t - s)^{\frac{1}{2}}\big) \thicknm{( Y,Y')}^{\gamma - 1}_{\mathrm X;q,q';k,\infty;[s,t]} \Big). 
\end{aligned}
\end{equation*}

Set $K := 2C_1 ( 1 + \|\xi\|_{\infty} )$ and $\theta := \frac{1}{2} C^{-1}_1 | K'|^{-(\gamma - 2)}$. If $(|\delta \mathbf X|_{p\text{-}\mathrm{var};[s,t]} + (t - s)^{\frac{1}{2}}) \le \theta$, we see that for $K' \ge K$ the condition \eqref{e:<= K'} implies 
\begin{equation*}
C_1 \big(|\delta \mathbf X|_{p\text{-}\mathrm{var};[s,t]} + (t - s)^{\frac{1}{2}} \big) \thicknm{( Y,Y')}^{\gamma - 1}_{\mathrm X;q,q';k,\infty;[s,t]} 
\le K'/2.
\end{equation*}
Consequently, noting $C_1 (1 + \|\xi\|_{\infty}) \le K/2 \le K'/2$, it holds $\thicknm{(\tilde Y,- H (Y))}_{\mathrm X;q,q';k,\infty;[s,t]} \le K'$. 
\end{proof}

The following Proposition~\ref{prop:contraction} shows that the twice composition of the solution map is a contraction on small intervals. The proof relies heavily on the  equivalence of norms proved in Lemma~\ref{lem:BMO of q-var}. 

\begin{proposition}\label{prop:contraction}
Assume \ref{(gamma)}, $\xi \in L^{\infty}_t$ and \ref{(f)} holds. Let $k \ge 1$, $H\in C^{\gamma}_{b}(\R,\R^d)$, and $\mathbf X\in \mathscr{C}^{p\text{-}\mathrm{var}}([s,t], \R^{d})$. Let $R_0 > 0$ be a constant  satisfying 
\begin{equation*}
\|\xi\|_{\infty} + \|H\|_{C^{\gamma}_b} + |\delta \mathbf X |_{p\text{-}\mathrm{var};[s,t]} \le R_0.
\end{equation*}
Denote $\Theta_0 := (\Theta,C_f, R_0)$. For $(Y,Y'), (\bar Y, \bar Y') \in \mathrm{RSM}^{(p,p)\text{-}\mathrm{var}}_{\mathrm X}$ with $Y_t = \bar Y_t$ and 
\begin{equation*}
\thicknm{(Y,Y')}_{\mathrm X;q,q';k,\infty;[s,t]} + \thicknm{(\bar Y,\bar Y')}_{\mathrm X;q,q';k,\infty;[s,t]} < \infty
\end{equation*}
Let $Z$ ($\bar Z$ resp.) be the unique progressively measurable process with $\|Z\|_{L^{2}([s,t])} < \infty$ ($\|\bar Z\|_{L^{2}([s,t])} < \infty$ resp.) a.s. such that $\delta M^{Y}_{s,t} = \int_{s}^{t} Z_r dW_r$ ($\delta M^{\bar Y}_{s,t} = \int_{s}^{t} \bar Z_r dW_r$ resp.). Denote the $i$-fold composition of $\Phi^{\xi}_{[s,t]}$ with itself by $(\Phi^{\xi}_{[s,t]})^{\circ i}$ and 
\begin{equation*}
\begin{cases}\displaystyle
(Y^{(i)},(Y^{(i)})',Z^{(i)}) := \left(\Phi^{\xi}_{[s,t]}\right)^{\circ i} (Y,Y',Z),\ \\
(\bar Y^{(i)},(\bar Y^{(i)})',\bar Z^{(i)}) := \left(\Phi^{\xi}_{[s,t]}\right)^{\circ i} (\bar Y,\bar Y',\bar Z),\ i = 0,1,2.
\end{cases}
\end{equation*}

Then, for every $\hat K > 0$ and $\varepsilon \in (0,1)$, there exists $\hat \theta > 0$ depending on $(k,\gamma,\Theta_0,\hat K,\varepsilon)$ such that if 
\begin{equation*}
\begin{cases}\displaystyle
\sup_{i=0,1,2}\,
\Big(\thicknm{(Y^{(i)},(Y^{(i)})')}_{\mathrm X;q,q';k,\infty;[s,t]} +  \thicknm{(\bar Y^{(i)},(\bar Y^{(i)})')}_{\mathrm X;q,q';k,\infty;[s,t]}\Big) \le \hat K,\\ \displaystyle
|\delta \mathbf X|_{p\text{-}\mathrm{var};[s,t]} + (t-s)^{1/2} \le \hat \theta,
\end{cases}
\end{equation*}
the following contraction property for $(\Phi^{\xi}_{[s,t]})^{\circ 2}$ holds: 
\begin{equation*}
\begin{aligned}
&\thicknm{Y^{(2)},(Y^{(2)})'; \bar Y^{(2)},(\bar Y^{(2)})'}_{\mathrm X,\mathrm X;q,q';k,\infty;[s,t]}\le (1 - \varepsilon) \, \thicknm{Y,Y'; \bar Y, \bar Y'}_{\mathrm X,\mathrm X;q,q';k,\infty;[s,t]} .
\end{aligned}
\end{equation*}
 
\end{proposition}

\begin{proof}
For simplicity of notations, denote 
\begin{equation*}
\begin{cases}\displaystyle
P^{(i)} := P^{Y^{(i)}},\ I^{(i)}_v := \int_s^v H(Y^{(i)}_r) d\mathbf X_r,\ \mathcal I^{(i)}_v := \int_{s}^v f(r, Y^{(i)}_r,  Z^{(i)}_r) dr, \\ \displaystyle
\bar P^{(i)} := P^{\bar Y^{(i)}},\ \bar I^{(i)}_v := \int_s^v H(\bar Y^{(i)}_r) d\mathbf X_r,\  \bar{\mathcal I}^{(i)}_v := \int_{s}^v f(r,\bar Y^{(i)}_r,\bar Z^{(i)}_r) dr, \, i = 0,1,2.
\end{cases}
\end{equation*}

We first estimate $P^{(i+1)} - \bar P^{(i+1)}$. On the one hand, by Lemma~\ref{lem:Gamma and I}, we have for $i = 0,1$, 
\begin{equation}\label{e:I - I* i}
\begin{aligned}
&\big\| \|R^{I^{(i)}} - R^{\bar  I^{(i)}}\|_{\frac{p}{2}\text{-}\mathrm{var}} \big\|_{\infty;k,\infty;[s,t]} + \big\| \|\delta (I^{(i)} - \bar  I^{(i)})\|_{p\text{-}\mathrm{var}} \big\|_{\infty;k,\infty;[s,t]}\\
&\lesssim_{k,\gamma,\Theta_0,\hat K} |\delta \mathbf X|_{p\text{-}\mathrm{var};[s,t]} \, \thicknm{Y^{(i)} ,(Y^{(i)})'; \bar Y^{(i)}, (\bar Y^{(i)})'}_{\mathrm X,\mathrm X;q,q';k,\infty;[s,t]}. 
\end{aligned}
\end{equation}
On the other hand, a similar estimate also holds for $\mathcal I^{(i)} - \bar{\mathcal I}^{(i)}$. Indeed, since $(Y^{(i)} - \bar Y^{(i)})_{t} = 0$ we have $\sup_{r \in [u,t]} |Y^{(i)} - \bar Y^{(i)}| \le \|Y^{(i)} - \bar Y^{(i)}\|_{p\text{-}\mathrm{var};[u,t]}$ for $u \in [s,t]$; and by Lemma~\ref{lem:Y <= R^Y'} we have 
\begin{equation*}
\big\| \|Y^{(i)} - \bar Y^{(i)}\|_{p\text{-}\mathrm{var}} \big\|_{\infty;k,\infty;[s,t]} \le \thicknm{Y^{(i)} ,(Y^{(i)})'; \bar Y^{(i)}, (\bar Y^{(i)})'}_{\mathrm X,\mathrm X;q,q';k,\infty;[s,t]}.
\end{equation*}	
Thus, according to inequalities $|f(r,y,z) - f(r,\bar y,\bar z)| \le C_f (|y - \bar y| + |z - \bar z|)$ and $\int_{u}^{t} (|Y^{(i)}_r - \bar Y^{(i)}_r| + |Z^{(i)}_r - \bar Z^{(i)}_r|) dr \le (t - s) \|Y^{(i)} - \bar Y^{(i)}\|_{p\text{-}\mathrm{var};[u,t]} + (t-s)^{1/2} \{\int_{u}^t |Z^{(i)}_r - \bar Z^{(i)}_r|^2 dr\}^{1/2}$, it holds that
\begin{equation}\label{e:Gamma-Gamma* i}
\begin{aligned}
&\big\| \|\delta(\mathcal I^{(i)} - \bar{\mathcal I}^{(i)}) \|_{1\text{-}\mathrm{var}} \big\|_{\infty;k,\infty;[s,t]}  \lesssim_{k,\gamma,\Theta_0,\hat K} (t - s)^{\frac{1}{2}} \, \thicknm{Y^{(i)} ,(Y^{(i)})'; \bar Y^{(i)}, (\bar Y^{(i)})'}_{\mathrm X,\mathrm X;q,q';k,\infty;[s,t]}. 
\end{aligned}
\end{equation}
Therefore, noting that $P^{(i+1)} - \bar P^{(i+1)} = - R^{I^{(i)}} + R^{\bar I^{(i)}} - \delta (\mathcal I^{(i)} - \bar{\mathcal I}^{(i)})$, we have 
\begin{equation}\label{e:R (i) - R (i-1)}
\begin{aligned}
&\big\| \|P^{(i+1)} - \bar P^{(i+1)}\|_{\frac{p}{2}\text{-}\mathrm{var}} \big\|_{\infty;k,\infty;[s,t]} \\
&\lesssim_{k,\gamma,\Theta_0,\hat K} \Big(|\delta \mathbf X|_{p\text{-}\mathrm{var};[s,t]} + (t - s)^{\frac{1}{2}}\Big) \, \thicknm{Y^{(i)} ,(Y^{(i)})'; \bar Y^{(i)}, (\bar Y^{(i)})'}_{\mathrm X,\mathrm X;q,q';k,\infty;[s,t]}.
\end{aligned}
\end{equation}

Denote $M^{(i)} := M^{Y^{(i)}}$ and $\bar M^{(i)} := M^{\bar Y^{(i)}}$, $i = 0,1,2$. Next, we estimate $M^{(i+1)} - \bar M^{(i+1)}$ and $Y^{(i+1)} - \bar Y^{(i+1)}$ for $i = 0,1$. 

We first estimate $M^{(i+1)} - \bar M^{(i+1)}$. Since for $r \in [s,t]$ it holds $(Y^{(i+1)} - \bar Y^{(i+1)})_r = \E_r [ \delta (I^{(i)} - \bar I^{(i)})_{r,t} + \delta (\mathcal I^{(i)} - \bar{\mathcal I}^{(i)})_{r,t} ]$ and $Y^{(i)}_{t} - \bar Y^{(i)}_{t} = 0$, we have  
\begin{equation*}
\delta (Y^{(i + 1)} - \bar Y^{(i + 1)})_{r,t} = \E_{r} [\delta (Y^{(i+1)} - \bar Y^{(i+1)})_{r,t}] = - \E_{r} \left[ \delta (I^{(i)} - \bar I^{(i)})_{r,t} + \delta (\mathcal I^{(i)} - \bar{\mathcal I}^{(i)})_{r,t} \right].
\end{equation*}
Hence, it holds that 
\begin{equation*}
\begin{aligned}
\delta (M^{(i+1)} - \bar M^{(i + 1)})_{r,t} &= \delta (Y^{(i+1)} - \bar Y^{(i+1)})_{r,t} + \delta (I^{(i)} - \bar I^{(i)})_{r,t} + \delta (\mathcal I^{(i)} - \bar{\mathcal I}^{(i)})_{r,t} \\
& = - \E_{r}\big[ \delta (I^{(i)} - \bar I^{(i)} + \mathcal I^{(i)} - \bar{\mathcal I}^{(i)})_{r,t} \big] + \delta (I^{(i)} - \bar I^{(i)} + \mathcal I^{(i)} - \bar{\mathcal I}^{(i)})_{r,t}
\end{aligned}
\end{equation*}
This, together with Lemma~\ref{lem:BMO of q-var}  yields that 
\begin{equation}\label{e:esti for M^n,m'}
\begin{aligned}
&\big\| \|\delta (M^{(i+1)} - \bar M^{(i + 1)})\|_{p\text{-}\mathrm{var}} \big\|_{\infty;k,\infty;[s,t]}  \\ 
&\lesssim_k \|\delta (M^{(i+1)} - \bar M^{(i + 1)})\|_{\mathrm{BMO}} \lesssim \big\| \delta (I^{(i)} - \bar I^{(i)}) \big\|_{\infty;2,\infty;[s,t]} + \big\| \delta (\mathcal I^{(i)} - \bar{\mathcal I}^{(i)}) \big\|_{\infty;2,\infty;[s,t]},
\end{aligned}
\end{equation}
where the last inequality is due to the BDG inequality and Doob's maximal inequality: 
\begin{equation*}
\begin{aligned}
\big\|\delta \langle M^{(i + 1)} - \bar M^{(i + 1)} \rangle^{1/2}_{r,t} \big| \mathcal F_r \big\|_{2} &\lesssim \big\|\sup_{v\in[r,t]} |\delta (M^{(i + 1)} - \bar M^{(i + 1)})_{r,v}| \big| \mathcal F_r \big\|_{2}\\
&\lesssim  \big\| \delta (M^{(i + 1)} - \bar M^{(i + 1)})_{r,t} \big| \mathcal F_r \big\|_{2} \\
&\lesssim 
\big\| \delta (I^{(i)} - \bar I^{(i)} + \mathcal I^{(i)} - \bar{\mathcal I}^{(i)})_{r,t} |\mathcal F_r \big\|_{2}.
\end{aligned}
\end{equation*}
 
In the RHS of \eqref{e:esti for M^n,m'}, note that $\| \delta (I^{(i)} - \bar I^{(i)}) \|_{\infty;2,\infty;[s,t]} $ plus $ \| \delta (\mathcal I^{(i)} - \bar{\mathcal I}^{(i)}) \|_{\infty;2,\infty;[s,t]}$ is bounded by $\| \| \delta (I^{(i)} - \bar I^{(i)}) \|_{p\text{-}\mathrm{var}} \|_{\infty;2,\infty;[s,t]}  $ plus $\| \| \delta (\mathcal I^{(i)} - \bar{\mathcal I}^{(i)}) \|_{1\text{-}\mathrm{var}} \|_{\infty;2,\infty;[s,t]}$. Furthermore, by the equivalence of norms proved in Lemma~\ref{lem:BMO of q-var}, the latter one is dominated by $\| \| \delta (I^{(i)} - \bar I^{(i)}) \|_{p\text{-}\mathrm{var}} \|_{\infty;k,\infty;[s,t]}  $ plus $\| \| \delta (\mathcal I^{(i)} - \bar{\mathcal I}^{(i)}) \|_{1\text{-}\mathrm{var}} \|_{\infty;k,\infty;[s,t]}$ .
Then, by \eqref{e:esti for M^n,m'} and estimates \eqref{e:I - I* i} and \eqref{e:Gamma-Gamma* i}, we get 
\begin{equation}\label{e:M i - M i-1}
\begin{aligned}
& \big\| \|\delta (M^{(i+1)} - \bar M^{(i+1)})\|_{p\text{-}\mathrm{var}} \big\|_{\infty;k,\infty;[s,t]} + \|\delta (M^{(i+1)} - \bar M^{(i+1)})\|_{\mathrm{BMO};[s,t]} \\
&\lesssim_{k,\gamma,\Theta_0,\hat K} (|\delta \mathbf X|_{p\text{-}\mathrm{var};[s,t]} + (t-s)^{\frac{1}{2}}) \, \thicknm{Y^{(i)},(Y^{(i)})';  \bar Y^{(i)},(\bar Y^{(i)})'}_{\mathrm X,\mathrm X;q,q';k,\infty;[s,t]}. 
\end{aligned}
\end{equation}
 Note $\delta (Y^{(i+1)} - \bar Y^{(i+1)}) = \delta (- I^{(i)} + \bar I^{(i)} - \mathcal I^{(i)} + \bar{\mathcal I}^{(i)} + M^{(i+1)} - \bar M^{(i+1)})$. The above estimate together with \eqref{e:I - I* i} and \eqref{e:Gamma-Gamma* i} yields 
\begin{equation}\label{e:Y i - Y i-1}
\begin{aligned}
&\big\| \|\delta (Y^{(i+1)} - \bar Y^{(i+1)})\|_{p\text{-}\mathrm{var}} \big\|_{\infty;k,\infty;[s,t]} \\
&\lesssim_{k,\gamma,\Theta_0,\hat K} (|\delta \mathbf X|_{p\text{-}\mathrm{var};[s,t]} + (t-s)^{\frac{1}{2}}) \, \thicknm{Y^{(i)},(Y^{(i)})';  \bar Y^{(i)},(\bar Y^{(i)})'}_{\mathrm X,\mathrm X;q,q';k,\infty;[s,t]}. 
\end{aligned}
\end{equation}

For $(Y^{(i+1)})' - (\bar Y^{(i+1)})'$, since $(Y^{(i+1)})' - (\bar Y^{(i+1)})' = - H(Y^{(i)}) + H(\bar Y^{(i)})$, by Lemma~\ref{lem:gamma - 2} (noting $H \in C^{\gamma}_b$ so that $H \in C^{2}_b$) and the inequality $|(Y^{(i+1)})' - (\bar Y^{(i+1)})'|\lesssim_{\Theta_0} |Y^{(i)} - \bar Y^{(i)}|$ we have 
\begin{equation}\label{e:Y i+1 <= Y i}
\begin{aligned}
&\|(Y^{(i+1)})' - (\bar Y^{(i+1)})'\|_{\mathcal L^{\infty}([s,t] \times \Omega)} + \big\| \|\delta ( (Y^{(i+1)})' - (\bar Y^{(i+1)})' )\|_{q'\text{-}\mathrm{var}} \big\|_{\infty;k,\infty;[s,t]} \\
&\lesssim_{\Theta_0} \Big(1 + \big\| \|\delta Y^{(i)}\|_{q\text{-}\mathrm{var}} \big\|_{\infty;k,\infty;[s,t]} + \big\| \|\delta \bar Y^{(i)}\|_{q\text{-}\mathrm{var}} \big\|_{\infty;k,\infty;[s,t]} \Big) \|Y^{(i)} - \bar Y^{(i)}\|_{\mathcal L^{\infty}([s,t] \times \Omega)} \\
&\quad + \big\| \|\delta (Y^{(i)} - \bar Y^{(i)})\|_{q'\text{-}\mathrm{var}} \big\|_{\infty;k,\infty;[s,t]}.
\end{aligned}
\end{equation}
For the RHS of the above inequality, by \eqref{e:Y q-var infty} in Lemma~\ref{lem:Y <= R^Y'} we get 
\begin{equation}\label{e:Y + bar Y <= 1}
\big\| \|\delta Y^{(i)}\|_{p\text{-}\mathrm{var}} \big\|_{\infty;k,\infty;[s,t]} + \big\| \|\delta \bar Y^{(i)}\|_{p\text{-}\mathrm{var}} \big\|_{\infty;k,\infty;[s,t]} \lesssim_{\Theta_0,\hat K} 1. 
\end{equation}
Note $Y^{(i)}_t = \bar Y^{(i)}_t$. Then, combining \eqref{e:Y i+1 <= Y i}, \eqref{e:Y + bar Y <= 1}, and \eqref{e:Y - Y q-var infty infty} in Lemma~\ref{lem:Y <= R^Y'}, we have 
\begin{equation*}
\begin{aligned}
&\|(Y^{(i+1)})' - (\bar Y^{(i+1)})'\|_{\mathcal L^{\infty}([s,t]\times\Omega)} + \big\| \|\delta ( (Y^{(i+1)})' - (\bar Y^{(i+1)})' )\|_{q'\text{-}\mathrm{var}} \big\|_{\infty;k,\infty;[s,t]} \\
&\lesssim_{\gamma,\Theta_0,\hat K} \thicknm{Y^{(i)},(Y^{(i)})'; \bar Y^{(i)},(\bar Y^{(i)})'}_{\mathrm X,\mathrm X;q,q';k,\infty;[s,t]}.
\end{aligned}
\end{equation*}
Consequently, by \eqref{e:R (i) - R (i-1)} and \eqref{e:M i - M i-1} we have 
\begin{equation}\label{e:(2) < = (1)}
\thicknm{Y^{(i+1)},(Y^{(i+1)})'; \bar Y^{(i+1)},(\bar Y^{(i+1)})'}_{\mathrm X,\mathrm X;q,q';k,\infty;[s,t]} \lesssim_{k,\gamma,\Theta_0,\hat K} \thicknm{Y^{(i)},(Y^{(i)})'; \bar Y^{(i)},(\bar Y^{(i)})'}_{\mathrm X,\mathrm X;q,q';k,\infty;[s,t]}. 
\end{equation}
However, it is not obvious to extract a small interval factor in the above estimate. Nevertheless, we can estimate $(Y^{(i+1)})' - (\bar Y^{(i+1)})'$ by $Y^{(i-1)} - \bar Y^{(i-1)}$-terms times a small interval factor. 

Indeed, since $Y^{(1)}_t - \bar Y^{(1)}_t = 0$, by the same proof as in Lemma~\ref{lem:Y <= R^Y'}, it holds that $\|Y^{(1)} - \bar Y^{(1)}\|_{\mathcal L^{\infty}([s,t] \times \Omega)} \le \| \|\delta (Y^{(1)} - \bar Y^{(1)})\|_{p\text{-}\mathrm{var}} \|_{\infty;k,\infty;[s,t]}$. Then, \eqref{e:Y i+1 <= Y i} and \eqref{e:Y + bar Y <= 1} yield, noting $p\le q\le q'$, 
\begin{equation*}
\begin{aligned}
&\|(Y^{(2)})' - (\bar Y^{(2)})'\|_{\mathcal L^{\infty}([s,t]\times\Omega)} + \big\| \|\delta ((Y^{(2)})' - (\bar Y^{(2)})') \|_{q'\text{-}\mathrm{var}}\big\|_{\infty;k,\infty;[s,t]} \\
&\lesssim_{\gamma,\Theta_0,\hat K} \big\| \|\delta (Y^{(1)} - \bar Y^{(1)})\|_{p\text{-}\mathrm{var}} \big\|_{\infty;k,\infty;[s,t]}.
\end{aligned}
\end{equation*}
This, together with the estimate of $Y^{(1)} - \bar Y^{(1)}$ shown in  \eqref{e:Y i - Y i-1} (letting $i = 0$), yields  
\begin{equation}\label{H Y 2 - H Y 1}
\begin{aligned}
&\|(Y^{(2)})' - (\bar Y^{(2)})'\|_{\mathcal L^{\infty}([s,t]\times\Omega)} + \big\| \|\delta( (Y^{(2)})' - (\bar Y^{(2)})' )\|_{q'\text{-}\mathrm{var}} \big\|_{\infty;k,\infty;[s,t]} \\
&\lesssim_{k,\gamma,\Theta_0,\hat K} (|\delta \mathbf X|_{p\text{-}\mathrm{var};[s,t]} + (t-s)^{\frac{1}{2}}) \, \thicknm{Y,Y';  \bar Y,\bar Y'}_{\mathrm X,\mathrm X;q,q';k,\infty;[s,t]}.
\end{aligned}
\end{equation}

Finally, combining \eqref{e:R (i) - R (i-1)}, \eqref{e:M i - M i-1}, and \eqref{H Y 2 - H Y 1}, we obtain that 
\begin{equation*}
\begin{aligned}
&\thicknm{Y^{(2)},(Y^{(2)})'; \bar Y^{(2)}, (\bar Y^{(2)})'}_{\mathrm X;q,q';k,\infty;[s,t]}\lesssim_{k,\gamma,\Theta_0,\hat K} (|\delta \mathbf X|_{p\text{-}\mathrm{var};[s,t]} + (t-s)^{1/2} )\\
&\qquad\qquad   \times \left(\thicknm{Y^{(1)},(Y^{(1)})'; \bar Y^{(1)}, (\bar Y^{(1)})'}_{\mathrm X;q,q';k,\infty;[s,t]}  + \thicknm{Y,Y'; \bar Y, \bar Y'}_{\mathrm X;q,q';k,\infty;[s,t]} \right).
\end{aligned}
\end{equation*}
This, together with \eqref{e:(2) < = (1)} (letting $i = 0$), yields that for some constant $C_{k,\gamma,\Theta_0,\hat K} > 0$, 
\begin{equation*}
\begin{aligned}
\thicknm{Y^{(2)},(Y^{(2)})'; \bar Y^{(2)}, (\bar Y^{(2)})'}_{\mathrm X;q,q';k,\infty;[s,t]} &\le C_{k,\gamma,\Theta_0,\hat K} (|\delta \mathbf X|_{p\text{-}\mathrm{var};[s,t]} + (t-s)^{1/2} )\\
&\quad \times \thicknm{Y,Y'; \bar Y, \bar Y'}_{\mathrm X;q,q';k,\infty;[s,t]}.  
\end{aligned}
\end{equation*}
Setting $\hat \theta := (1 - \varepsilon) C^{-1}_{k,\gamma,\Theta_0,\hat K}$, we get the desired contraction property. 
\end{proof}

In the following Theorem~\ref{thm:the existence}, we prove the well-posedness of  Eq.~\eqref{e:BSDE} via Picard's iteration. By the contraction property of the solution map on small intervals, we extend the solution to the whole interval, via the telescopic argument which is valid due to the \emph{a priori} estimate for the solution established in Theorem~\ref{thm:BMO bound}.

\begin{theorem}\label{thm:the existence}
Assume \ref{(gamma)}, $\xi \in L^{\infty}$ and \ref{(fdet)} holds. Let $H\in C^{\gamma}_{b}(\R,\R^d)$ and $\mathbf X\in \mathscr{C}^{p\text{-}\mathrm{var}}([0,T], \R^{d})$.
Then, there exists a unique solution $(Y,Z)$ to Eq.~\eqref{e:BSDE} with
\begin{equation*}
(Y,-H(Y))\in \mathrm{RSM}^{(p,p)\text{-}\mathrm{var}}_{\mathrm X}\ \text{ and }\ \ \thicknm{(Y,-H(Y))}_{\mathrm X;q,q';1,\infty}   < \infty.
\end{equation*}
Moreover, for every $k\ge 1$,
\begin{equation}\label{e:|(Y,-H(Y))|_X;p,p;k,infty < infty for every k}
\thicknm{(Y,-H(Y))}_{\mathrm X;p,p;k,\infty} < \infty. 
\end{equation}
In particular, $\|Y\|_{\infty} + \|\delta (\int_{0}^{\bm\cdot} Z_r dW_r)\|_{\mathrm{BMO}} < \infty$.
\end{theorem}

\begin{proof}
Uniqueness is already proved in Corollary~\ref{cor:uniqueness} by the comparison principle, so it suffices to construct a solution satisfying \eqref{e:|(Y,-H(Y))|_X;p,p;k,infty < infty for every k}. Since \eqref{e:|(Y,-H(Y))|_X;p,p;k,infty < infty for every k} is independent of $(q,q')$, we can assume  $p < q'$ without loss of generality for the application of Proposition~\ref{prop:invariant within ball}. In addition, we fix $k > 1$, where the condition $k > 1$ is also required in Theorem~\ref{thm:BMO bound} which will be used later. 

For $t \in [0,T]$, denote $(Y^{(i)},(Y^{(i)})',Z^{(i)})$ by 
\begin{equation*}
(Y^{(i)},(Y^{(i)})',Z^{(i)}) := \Phi^{\xi}_{[t,T]}(Y^{(i-1)},(Y^{(i-1)})',Z^{(i-1)}),\ i\ge 1;\quad (Y^{(0)},(Y^{(0)})',Z^{(0)}) := (0,0,0).
\end{equation*}
Our goal is  to construct a solution $(Y,Z)$ to \eqref{e:BSDE} by taking the limit of $(Y^{(i)},Z^{(i)})$.  For the reader's convenience, we divide the construction into three steps.

{\bf Step 1.} In this step, we prove the convergence of $\{(Y^{(i)},(Y^{(i)})',Z^{(i)})\}_{i \ge 1}$ on $[t,T]$, assuming $|\delta \mathbf X|_{p\text{-}\mathrm{var};[t,T]} + (T-t)^{1/2}$ is small enough. 

By the \emph{a priori} estimate established in Theorem~\ref{thm:BMO bound}, there exists a constant $L$ such that for any time $t \in [0,T]$ and any solution $(\hat Y,\hat Z)$ on $[t,T]$ with $\thicknm{(\hat Y,-H(\hat Y))}_{\mathrm X;q,q';k,\infty;[t,T]} < \infty$, it holds 
\begin{equation}\label{e:a pri esti}
\sup_{r\in[t,T]} \|\hat Y_r\|_{\infty} \le L.
\end{equation}
In particular, $\|\xi\|_{\infty} \le L$. Let $C > 0$ be the constant obtained in Proposition~\ref{prop:invariant within ball} (where $p < q'$ is needed, and $C$ is independent of $\xi$). Denote $K_1$ as follows: 
\begin{equation*}
K_1 := C (1 + L).
\end{equation*}
Noting $K_1 \ge C (1 + \|\xi\|_{\infty})$, Proposition~\ref{prop:invariant within ball} yields that if assuming further 
\begin{equation}\label{e:cond of T-t}
|\delta \mathbf X|_{p\text{-}\mathrm{var};[t,T]} + (T-t)^{1/2} \le  C^{-1} K^{-(\gamma - 2)}_1 /2  =: \theta_1,
\end{equation}
then for $(\tilde \psi,\tilde \psi',\tilde z) := \Phi^{\xi}_{[t,T]}(\psi,\psi',z)$ it holds:  
\begin{equation}\label{e:invariance of Phi}
\thicknm{(\tilde \psi, \tilde \psi')}_{\mathrm X;q,q';k,\infty;[t,T]} \le K_1\ \text{ whenever }\ \thicknm{( \psi,  \psi')}_{\mathrm X;q,q';k,\infty;[t,T]} \le K_1.
\end{equation}
Since $\thicknm{(Y^{(0)},(Y^{(0)})')}_{\mathrm X;q,q';k,\infty;[t,T]} = 0 \le K_1$, assuming \eqref{e:cond of T-t} holds, we have 
\begin{equation}\label{e:upper bound in i}
\sup_{i \ge 0} \  \thicknm{(Y^{(i)},(Y^{(i)})')}_{\mathrm X;q,q';k,\infty;[t,T]} \le K_1.
\end{equation}

Now we show the contraction of $(\Phi^{\xi}_{[t,T]})^{\circ 2}$. Denote by $\mathfrak B^{\xi;q,q'}_{[t,T]}(K_1)$ the set as follows:  
\begin{equation*}
\begin{aligned}
\mathfrak B^{\xi;q,q'}_{[t,T]}(K_1) := &\bigg\{(\psi,\psi',z);\ (\psi,\psi') \in \mathrm{RSM}^{(p,p)\text{-}\mathrm{var}}_{\mathrm X}([t,T]),\ \thicknm{(\psi,\psi')}_{\mathrm X;q,q';k,\infty;[t,T]} \le K_1,\ \psi_{T} = \xi, \\
&z\ \text{is progressively measurable with } \|z\|_{L^2([t,T])} < \infty\ \text{ a.s.},\  \text{ and } \delta M^{\psi}_{u,v} = \int_{u}^{v} z_r dW_r \bigg\}.
\end{aligned}
\end{equation*}
By \eqref{e:upper bound in i}, we have $(Y^{(i)},(Y^{(i)})',Z^{(i)}) \in \mathfrak B^{\xi;q,q'}_{[t,T]}(K_1)$ for every $i \ge 1$. Moreover, since the map $\Phi^{\xi}_{[t,T]}$ is invariant in $\mathfrak B^{\xi;q,q'}_{t,T}(K_1)$ under the condition \eqref{e:cond of T-t}, by Proposition~\ref{prop:contraction} there exists $\theta_2 > 0$ depending on $K_1$ such that if assuming further 
\begin{equation}\label{e:cond of T-t'}
|\delta \mathbf X|_{p\text{-}\mathrm{var};[t,T]} + (T-t)^{1/2} \le \theta_1 \land \theta_2, 
\end{equation}
then the map $(\Phi^{\xi}_{[t,T]})^{\circ 2}$ is a contraction in $\mathfrak B^{\xi;q,q'}_{[t,T]}(K_1)$.

Assume \eqref{e:cond of T-t'} holds. Then, by the completeness of $\mathfrak B^{\xi;q,q'}_{[t,T]}(K_1)$ shown in Lemma~\ref{lem:Banach property of RSMs} and the contraction property of $(\Phi^{\xi}_{[t,T]})^{\circ 2}$, there exists $(Y,Y',Z)$, the unique fixed point of $(\Phi^{\xi}_{[t,T]})^{\circ 2}$ in $\mathfrak B^{\xi;q,q'}_{[t,T]}(K_1)$, such that $(Y,Y') \in \mathrm{RSM}^{(p,p)\text{-}\mathrm{var}}_{\mathrm X}([t,T],\R)$ and 
\begin{equation*}
\lim_{i \to \infty} \, \thicknm{Y^{(i)},(Y^{(i)})'; Y,Y'}_{\mathrm X,\mathrm X;q,q';k,\infty;[t,T]} = 0. 
\end{equation*}
Clearly, since $(Y,Y',Z) \in \mathfrak B^{\xi;q,q'}_{[t,T]}(K_1)$, we have $\thicknm{(Y,Y')}_{\mathrm X;q,q';k,\infty;[t,T]} \le K_1$.

\textbf{Step 2.} In this step, we prove the existence of the solution $(Y,Z)$ on small intervals.

In the step 1, we show that $(Y,Y',Z)$ is the unique fixed point of $(\Phi^{\xi}_{[t,T]})^{\circ 2}$ in $\mathfrak B^{\xi;q,q'}_{[t,T]}(K_1)$. Now we show that it is also a fixed point of $\Phi^{\xi}_{[t,T]}$ in $\mathfrak B^{\xi;q,q'}_{[t,T]}(K_1)$. 

Denote $(\tilde Y,\tilde Y',\tilde Z) := \Phi^{\xi}_{[t,T]}(Y,Y',Z)$. Since $(Y,Y',Z) \in \mathfrak B^{\xi;q,q'}_{[t,T]}(K_1)$, by \eqref{e:invariance of Phi} $(\tilde Y,\tilde Y',\tilde Z)$ also belongs to $\mathfrak B^{\xi;q,q'}_{[t,T]}(K_1)$. Furthermore, the triplet  $(\Phi^{\xi}_{[t,T]})^{\circ 3}(Y,Y',Z)$, which is equal to $(\tilde Y,\tilde Y',\tilde Z)$, is also a fixed point of $(\Phi^{\xi}_{[t,T]})^{\circ 2}$. Therefore, the uniqueness of the fixed point of $(\Phi^{\xi}_{[t,T]})^{\circ 2}$ yields 
\begin{equation*}
(\tilde Y,\tilde Y',\tilde Z) = (Y,Y',Z). 
\end{equation*}
Consequently, $(Y,Y',Z)$ is a fixed point of $\Phi^{\xi}_{[t,T]}$, and hence $(Y,Z)$ solves Eq.~\eqref{e:BSDE} on $[t,T]$.

Since $(Y,Z)$ solves Eq.~\eqref{e:BSDE} with $\thicknm{(Y,Y')}_{\mathrm X;q,q';k,\infty;[t,T]} < \infty$, the \emph{a priori} estimate proved in Theorem~\ref{thm:BMO bound} implies $\thicknm{(Y,Y')}_{\mathrm X;p,p;k,\infty;[t,T]} < \infty$, which meets our requirement \eqref{e:|(Y,-H(Y))|_X;p,p;k,infty < infty for every k}. 

\textbf{Step 3.}
Now we apply a backward telescoping argument to prove the existence of a solution on an arbitrary interval $[0,T]$. 

We claim that there exists a partition $\pi=\{[t_i,t_{i+1}]; i=0,1,\ldots,N-1\}$ of $[0,T]$ such that, on each subinterval $[t_i,t_{i+1}]$, we can construct a solution as in Steps~1 and 2. To see this, since $\theta_1$, $K_1$, and $\theta_2$ are chosen independent of the terminal value, and since for any terminal value $Y_v$ the condition $K_1 \ge C(1 + \|Y_v\|_{\infty})$ always holds so that Proposition~\ref{prop:invariant within ball} always applies at any terminal time $v$, we can choose the partition fine enough so that the following condition holds:  
\begin{equation*}
(|\delta \mathbf X|_{p\text{-}\mathrm{var};[t_i,t_{i+1}]} + |t_{i+1} - t_{i}|^{1/2}) \le \theta_1 \land \theta_2, 
\end{equation*}
and then the arguments in Steps~1 and 2 apply for each subinterval $[t_i,t_{i+1}]$.  

Finally, the global existence of the solution follows by iterating the procedures in Steps~1 and 2 on each subinterval, proceeding backward from the terminal time $T$ to the initial time $0$, which can be done in finitely many steps.
\end{proof}

We have the following quantitative stability result for the equation. Denote $\Theta'_1 := (\Theta,C_f, R'_1)$. 

\begin{theorem}\label{thm:stability}
Assume both $(\xi,f,H,\mathbf X)$ and $(\bar \xi, \bar f, \bar H, \bar{\mathbf X})$ satisfy the conditions in Theorem~\ref{thm:the existence}. Let $(Y,Z)$ ($(\bar Y,\bar Z)$ resp.) be the unique solution to \eqref{e:BSDE} with the coefficient $(\xi,f,H,\mathbf X)$ (the coefficient $(\bar \xi,\bar f, \bar H, \bar{\mathbf X})$ resp.) such that  
\begin{equation*}
\thicknm{(Y,-H(Y))}_{\mathrm X;q,q';1,\infty}  + \thicknm{(\bar Y,- \bar H(\bar Y))}_{\bar{\mathrm X};q,q';1,\infty} < \infty.
\end{equation*}
Let $R'_1 > 0$ be a constant such that 
\begin{equation*}
\|\xi\|_{\infty} + \|\bar \xi\|_{\infty} + \|H\|_{C^{\gamma}_b} + \|\bar H\|_{C^{\gamma}_b} + |\delta \mathbf X |_{p\text{-}\mathrm{var}} + |\delta  \bar{\mathbf X} |_{p\text{-}\mathrm{var}} + T \le R'_1.
\end{equation*}

Then, for any $k \ge 1$ the following local Lipschitz inequality holds up to a multiplicative constant depending on $(k,\gamma,\Theta'_1)$,
\begin{equation}
\begin{aligned}\label{ine:stab}
&\thicknm{Y,-H(Y) ; \bar Y, - \bar H(\bar Y)}_{\mathrm X,\bar{\mathrm X};p,p;k,\infty}     \\ 
& \qquad \lesssim \|\xi - \bar \xi\|_{\infty} + \|f - \bar f\|_{\infty} + \|H - \bar H\|_{C^{\gamma - 1}_b} + |\delta (\mathbf X - \bar{\mathbf X}) |_{p\text{-}\mathrm{var}} . 
\end{aligned}
\end{equation}
In particular, $\|Y - \bar Y\|_{\infty} + \|\delta \big( \int_0 ^{\bm\cdot} (Z_r - \bar Z_r) dW_r \big)\|_{\mathrm{BMO}}$ is also bounded by the right-hand side of the above inequality.
\end{theorem}

\begin{proof}
Without loss of generality, we assume $k > 1$ to apply  Theorem~\ref{thm:BMO bound}.
Recall $Y' = -H(Y)$ and $\bar Y' = - \bar H(\bar Y)$. For simplicity of notation, denote 
\begin{equation*}
\begin{cases}\displaystyle
L := \|H - \bar H\|_{C^{\gamma - 1}_b} + |\delta (\mathbf X - \bar{\mathbf X})|_{p\text{-}\mathrm{var};[0,T]} + \|f - \bar f\|_{\infty} \\ \displaystyle
I^{\mathbf X}_t := \int_0^t H(Y_r) d\mathbf X_r,\ \mathcal I_t := \int_{0}^t f(r, Y_r,  Z_r) dr, \\ \displaystyle
\bar I^{\bar{\mathbf X}}_t := \int_0^t \bar H(\bar Y_r) d\bar{\mathbf X}_r,\  \bar{\mathcal I}_t := \int_{0}^t \bar f(r,\bar Y_r,\bar Z_r) dr. 
\end{cases}
\end{equation*}
By the \emph{a priori} estimate made by Theorem~\ref{thm:BMO bound}, we have for some constant $\hat K$ depending only on $k,\gamma$, and $\Theta'_1$, 
\begin{equation*}
\thicknm{(Y,Y')}_{\mathrm X;p,p;k,\infty;[0,T]} +  \thicknm{(\bar Y,\bar Y')}_{\bar{\mathrm X};p,p;k,\infty;[0,T]} \le \hat K.
\end{equation*}

First, we estimate $P^{Y} - P^{\bar Y}$. By Lemma~\ref{lem:Gamma and I}, we have (noting the dependence of $\hat K$ is absorbed by $(k,\gamma,\Theta'_1)$ and $\thicknm{Y ,Y'; \bar Y, \bar Y'}_{\mathrm X,\bar{\mathrm X};q,q';k,\infty;[s,t]} \le \thicknm{Y ,Y'; \bar Y, \bar Y'}_{\mathrm X,\bar{\mathrm X};p,p;k,\infty;[s,t]}$)
\begin{equation*}
\begin{aligned}
&\big\| \|R^{I^{\mathbf X}} - R^{\bar I^{\bar {\mathbf X}}}\|_{\frac{p}{2}\text{-}\mathrm{var}} \big\|_{\infty;k,\infty;[s,t]} + \big\| \|\delta(I^{\mathbf X} - \bar I^{\bar{\mathbf X}})\|_{p\text{-}\mathrm{var}} \big\|_{\infty;k,\infty;[s,t]} \lesssim_{k,\gamma,\Theta'_1} |\delta \mathbf X|_{p\text{-}\mathrm{var};[s,t]} \Big(\|Y_t - \bar Y_t\|_{\infty} 
\\ 
& \quad \quad + \thicknm{Y ,Y'; \bar Y, \bar Y'}_{\mathrm X,\bar{\mathrm X};p,p;k,\infty;[s,t]} + \|H - \bar H\|_{C^{\gamma - 1}_{b}}\Big) + |\delta( \mathbf X - \bar{\mathbf X} )|_{p\text{-}\mathrm{var};[s,t]}.
\end{aligned}
\end{equation*}
Moreover, by the same calculation leading to \eqref{e:Gamma-Gamma* i}, but with the difference of $(Y - \bar Y)_t$ and $f - \bar f$ added, we have 
\begin{equation*}	
\begin{aligned}
&\big\| \|\delta(\mathcal I - \bar{\mathcal I}) \|_{1\text{-}\mathrm{var}} \big\|_{\infty;k,\infty;[s,t]} \\ &\lesssim_{k,\gamma,\Theta'_1} (t - s) \|(Y - \bar Y)_t\|_{\infty} + (t - s)^{\frac{1}{2}} \, \thicknm{Y ,Y'; \bar Y, \bar Y'}_{\mathrm X,\bar{\mathrm X};p,p;k,\infty;[s,t]} + \|f - \bar f\|_{\infty}. 
\end{aligned}
\end{equation*}
Therefore, noting that $P^{Y} - P^{\bar Y} = - R^{I^{\mathbf X}} + R^{\bar I^{\bar{\mathbf X}}} - \delta (\mathcal I - \bar{\mathcal I})$, we have 
\begin{equation}\label{e:R  - bar R}
\begin{aligned}
&\big\| \|P^{Y} - P^{\bar Y}\|_{\frac{p}{2}\text{-}\mathrm{var}} \big\|_{\infty;k,\infty;[s,t]}\\ &\lesssim_{k,\gamma,\Theta'_1} \|Y_t - \bar Y_t\|_{\infty} +  \Big(|\delta \mathbf X|_{p\text{-}\mathrm{var};[s,t]} + (t - s)^{\frac{1}{2}}\Big) \thicknm{Y ,Y'; \bar Y, \bar Y'}_{\mathrm X,\bar{\mathrm X};p,p;k,\infty;[s,t]} + L.
\end{aligned}
\end{equation} 

Next, we estimate $M^Y - M^{\bar Y}$ and $Y - \bar Y$. Note that for $r \in [s,t]$,
\begin{equation*}
\delta (M^{Y} - M^{\bar Y})_{r,t} = - \E_{r}\left[Y_t - \bar Y_t + \delta (I^{\mathbf X} - \bar I^{\bar {\mathbf X}} + \mathcal I - \bar{\mathcal I})_{r,t}\right] + Y_t - \bar Y_t + \delta (I^{\mathbf X} - \bar I^{\bar {\mathbf X}} + \mathcal I - \bar{\mathcal I})_{r,t}. 
\end{equation*}
Then, since $I^{\mathbf X} - \bar I^{\bar {\mathbf X}}$ and $\mathcal I - \bar{\mathcal I}$ are already estimated as follows:
\begin{equation}\label{e:I + I - I - I}
\begin{aligned}
&\big\| \|\delta(I^{\mathbf X} - \bar I^{\bar{\mathbf X}})\|_{p\text{-}\mathrm{var}} \big\|_{\infty;k,\infty;[s,t]} + \big\| \|\delta(\mathcal I - \bar{\mathcal I}) \|_{1\text{-}\mathrm{var}} \big\|_{\infty;k,\infty;[s,t]}\\
& \lesssim_{k,\gamma,\Theta'_1} \|Y_t - \bar Y_t\|_{\infty} + \Big(|\delta \mathbf X|_{p\text{-}\mathrm{var};[s,t]} + (t - s)^{\frac{1}{2}}\Big)\thicknm{Y ,Y'; \bar Y, \bar Y'}_{\mathrm X,\bar{\mathrm X};p,p;k,\infty;[s,t]} + L, 
\end{aligned}
\end{equation}
we can follow the same calculation as \eqref{e:esti for M^n,m'}--\eqref{e:M i - M i-1} to obtain  
\begin{equation}\label{e:M - bar M}
\begin{aligned}
&\big\| \|\delta (M^Y - M^{\bar Y})\|_{p\text{-}\mathrm{var}} \big\|_{\infty;k,\infty;[s,t]} + \|\delta (M^Y - M^{\bar Y})\|_{\mathrm{BMO};[s,t]} \lesssim_{k,\gamma,\Theta'_1} \text{RHS of \eqref{e:I + I - I - I}.}
\end{aligned}
\end{equation}
Note $\delta (Y - \bar Y) = \delta (- I^{\mathbf X} + \bar I^{\bar{\mathbf X}} - \mathcal I + \bar{\mathcal I} + M^{Y} - M^{\bar Y})$. The above estimate together with \eqref{e:I + I - I - I} yields 
\begin{equation}\label{e:Y - bar Y}
\big\| \|\delta (Y - \bar Y)\|_{p\text{-}\mathrm{var}} \big\|_{\infty;k,\infty;[s,t]} \lesssim_{k,\gamma,\Theta'_1} \text{RHS of \eqref{e:I + I - I - I}.} 
\end{equation}

Now, for $Y' - \bar Y' = - H(Y) + H(\bar Y) - (H - \bar H)(\bar Y)$, by Lemma~\ref{lem:gamma - 2} (noting $H \in C^{\gamma}_b$ so that $H \in C^{2}_b$) and inequalities $|Y' - \bar Y'|\lesssim_{\Theta'_1} |Y - \bar Y| + \|H - \bar H\|_{\infty}$ and $\|\delta (H - \bar H)(\bar Y)\|_{p\text{-}\mathrm{var}} \le \|D(H - \bar H)\|_{\infty} \|\delta \bar Y\|_{p\text{-}\mathrm{var}}$, we have 
\begin{equation}\label{Y' - bar Y'}
\begin{aligned}
&\|Y' - \bar Y'\|_{\mathcal L^{\infty}([s,t] \times \Omega)} + \big\| \|\delta (Y' - \bar Y' )\|_{p\text{-}\mathrm{var}} \big\|_{\infty;k,\infty;[s,t]} \\ 
&\lesssim_{\Theta'_1} \Big(1 + \big\| \|\delta Y\|_{p\text{-}\mathrm{var}} \big\|_{\infty;k,\infty;[s,t]} + \|\delta \bar Y\|_{p\text{-}\mathrm{var}} \big\|_{\infty;k,\infty;[s,t]} \Big) \|Y - \bar Y\|_{\mathcal L^{\infty}([s,t] \times \Omega)}\\
&\quad + \big\| \|\delta (Y - \bar Y)\|_{p\text{-}\mathrm{var}} \big\|_{\infty;k,\infty;[s,t]} + \|H - \bar H\|_{C^1_b} \Big(1 + \big\| \|\delta \bar Y\|_{p\text{-}\mathrm{var}} \big\|_{\infty;k,\infty;[s,t]} \Big)\\
&\lesssim_{k,\gamma,\Theta'_1} \|Y_t - \bar Y_t\|_{\infty} + \big\| \|\delta (Y - \bar Y)\|_{p\text{-}\mathrm{var}} \big\|_{\infty;k,\infty;[s,t]} + \|H - \bar H\|_{C^{1}_{b}}. 
\end{aligned}
\end{equation}
By \eqref{e:Y - bar Y}, the RHS of the above estimate is, in turn, bounded by the RHS of \eqref{e:I + I - I - I}. 

Combining \eqref{e:R  - bar R}, \eqref{e:M - bar M}, and \eqref{Y' - bar Y'}, for some constant $C > 0$ independent of $t$ we have 
\begin{equation*}
\begin{aligned}
&\thicknm{Y ,Y'; \bar Y, \bar Y'}_{\mathrm X,\bar{\mathrm X};p,p;k,\infty;[s,t]} \\
&\le C \bigg( \|Y_t - \bar Y_t\|_{\infty} + \Big(|\delta \mathbf X|_{p\text{-}\mathrm{var};[s,t]} + (t - s)^{\frac{1}{2}}\Big) \thicknm{Y ,Y'; \bar Y, \bar Y'}_{\mathrm X,\bar{\mathrm X};p,p;k,\infty;[s,t]} + L \bigg). 
\end{aligned}
\end{equation*}
Hence, the following inequality holds whenever $(|\delta \mathbf X|_{p\text{-}\mathrm{var};[s,t]} + (t - s)^{1/2}) \le (2C)^{-1}$: 
\begin{equation*}
\thicknm{Y ,Y'; \bar Y, \bar Y'}_{\mathrm X,\bar{\mathrm X};p,p;k,\infty;[s,t]} \le 2C \|Y_t - \bar Y_t\|_{\infty} + 2 C L.  
\end{equation*}
In particular, noting $Y_T - \bar Y_T = \xi - \bar \xi$, it holds $\thicknm{Y ,Y'; \bar Y, \bar Y'}_{\mathrm X,\bar{\mathrm X};p,p;k,\infty;[t,T]} \le 2C \|\xi - \bar \xi\|_{\infty} + 2 C L$ whenever $(|\delta \mathbf X|_{p\text{-}\mathrm{var};[t,T]} + (T-t)^{1/2}) \le (2C)^{-1}$. 

Since for each $t \in [0,T]$, by \eqref{e:Y - Y q-var infty infty} in Lemma~\ref{lem:Y <= R^Y'} the difference $\|Y_t - \bar Y_t\|_{\infty}$ is dominated by $\|\xi - \bar \xi\|_{\infty} +\thicknm{Y ,Y'; \bar Y, \bar Y'}_{\mathrm X,\bar{\mathrm X};p,p;k,\infty;[t,T]}$. Therefore, the desired estimate can be verified via a backward telescoping argument, in the same way as in Step~4 of Proposition~\ref{prop:bounded by BMO}.  
\end{proof}

\begin{remark}\label{rem:on controlled fields}
In this paper, we assume that $H(\bm\cdot)$ is independent of time $t$. Under the assumption that $H$ is deterministic, indeed, one can also consider time-dependent $\mathrm X$-controlled vector fields of the type considered in, e.g., \cite{nx20,dause2026controlled}. 
We restrict ourselves to the time-independent case mainly to avoid lengthy but routine calculations.
However, the genuinely random vector fields case, such as $H(W_t,Y_t)$, is substantially more delicate, as explained in Remark~\ref{rem:deterministic}.
\end{remark} 

\begin{remark}\label{rem:under higher regularity}
Under higher regularity assumptions on $H(\bm\cdot)$, Eq.~\eqref{e:BSDE} can be solved through Doss--Sussmann transformation and It\^o--Wentzell formula. We sketch the argument of the well-posedness by this method under assumptions $H(\bm\cdot) \in C^{5}_b(\R,\R^d)$ and $\mathbf X \in \mathscr C^{0,p\text{-}\mathrm{var}}_{g}([0,T],\R^d)$. 

For $T' \in [0,T]$, let $\phi^{T'}(t,y)$ be the unique solution to the following RDE:
\begin{equation*}
\phi^{T'}(t,y) = y + \int_{t}^{T'} H(\phi^{T'}(r,y)) d\mathbf X_r, \quad (t,y) \in [0,T'] \times \R.
\end{equation*}
Then, by \cite[Proposition~11.11]{friz2010multidimensional}, $\phi^{T'}(t,y)$ forms a flow of $C^3$-diffeomorphisms. 

Thanks to the $C^3$-smoothness of $\phi^{T'}$, the approximation solution $(Y,Z)$ to Eq.~\eqref{e:BSDE} obtained by \cite[Theorem~3]{DF} has an explicit representation through classical quadratic BSDEs. Indeed, further assume that $\xi \in L^{\infty}$ and \ref{(f)} holds. Denote $\tilde{f}^{T'}$ as follows:
\begin{equation*}
\tilde{f}^{T'}(t,y,z) := \Big(f(t,\phi^{T'}(t,y),\partial_y \phi^{T'}(t,y) z) + \frac{1}{2} \partial^2_{yy} \phi^{T'}(t,y)|z|^2 \Big) / \partial_y \phi^{T'}(t,y),\quad (t,y,z) \in [0,T'] \times \R^{1+d}. 
\end{equation*}
Then, by the proof of \cite[Theorem~3]{DF}, there exists $\pi$, a partition on $[0,T]$, s.t. for each $[u,v] \in \pi$, 
\begin{equation*}
Y_t = \phi^{v}(t,\tilde Y^{v}_{t}),\quad Z_{t} = \partial_y\phi^{v}(t,\tilde Y^{v}_{t}) \tilde Z^{v}_t,\quad t \in [u,v].
\end{equation*}
Here, $(\tilde Y^v,\tilde Z^v)$ is the unique solution to the following quadratic BSDE in $\mathcal L^{\infty}([u,v]\times\Omega,\R) \times \mathcal L^{2}([u,v]\times\Omega,\R^d)$:
\begin{equation*}
\tilde Y^v_{t} = Y_v + \int_{t}^{v} \tilde f^{v}(r,\tilde Y^v_{r},\tilde Z^v_{r}) dr - \int_{t}^{v} \tilde Z^v_r dW_r,\quad t\in [u,v].
\end{equation*} 
Note that for $[T',T] \in \pi$, $Y_T = \xi$ is given so that $(\tilde Y^{T},\tilde Z^{T})$ can be solved on $[T',T]$. Consequently, the next terminal value $Y_{T'} = \phi^{T}(T',\tilde Y^{T}_{T'})$ can be obtained, and hence one can repeat this procedure on subintervals up to time $0$.  

To see that $(Y,Z)$ satisfies Definition~\ref{def:def of BSDEs}, it suffices to show that $(Y,-H(Y))$ is an \emph{RSM}, and that the integral equation Eq.~\eqref{e:BSDE} holds. Since $H(\bm\cdot) \in C^{5}_b$, by \cite[Example~3.16]{dause2026controlled}, for each $[u,v] \in \pi$, 
\begin{equation*}
(\phi^{v}, - H(\phi^v), \partial_y \phi^v, D H(\phi^v) H(\phi^v), \partial^2_{yy} \phi^{v} ,0) \in \mathcal D^{3/p}_{\mathbf X} \mathrm{Lip}^3_y (\R; \R),
\end{equation*}
where the controlled field space $\mathcal D^{3/p}_{\mathbf X} \mathrm{Lip}^3_y (\R; \R)$ is defined in \cite[Definition~3.6]{dause2026controlled}. Clearly $\tilde Y^v$ is a semimartingale and hence $(\tilde Y^v, 0)$ forms an \emph{RSM}. Therefore, by the rough stochastic It\^o--Wentzell formula \cite[Theorem~4.16]{dause2026controlled}, $(Y_t,-H(Y_t)) = (\phi^{v}(t,\tilde Y^{v}), -H(\phi^v(t,\tilde Y^{v}_t)) )$, $t\in[u,v]$, is an \emph{RSM}, and it holds 
\begin{equation*}
Y_t = Y_v + \int_{t}^{v} f(r, Y_r, Z_r) dr + \int_t^v H(Y_r) d\mathbf X_r - \int_t^v Z_r dW_r,\quad t \in [u,v]. 
\end{equation*}
This proves that $(Y,Z)$ is a solution to Eq.~\eqref{e:BSDE}. 

Finally, we show the uniqueness of the solution. Indeed, let $\{\mathbf X^n:=(\mathrm X^n , \mathbb X^n)\}_{n\ge 1} $ be a sequence of rough paths canonically lifted by smooth paths $\{\mathrm X^n\}_{n\ge 1}$, and assume $|\delta( \mathbf X^n - \mathbf X )|_{p\text{-}\mathrm{var}}\to 0$. By \cite[Theorem~3]{DF}, we see that 
\begin{equation*}
Y^n \to Y \ \text{ uniformly on }[0,T]\ \text{ a.s. and }\ Z^n \to Z \ \text{ in }\mathcal L^2([0,T]\times \Omega,\R^d). 
\end{equation*}
Here, $(Y^n,Z^n) \in \mathcal L^{\infty}([0,T]\times \Omega,\R) \times \mathcal L^{2}([0,T]\times \Omega,\R^d)$ is the unique solution to Eq.~\eqref{e:BSDE} with $\mathbf X$ replaced by $\mathbf X^n$. Furthermore, by Propositions~\ref{prop:classical BSDEs} and Theorem~\ref{thm:BMO bound}, we have 
\begin{equation*}
\sup_{n} \, \, \thicknm{(Y^n,-H(Y^n)) }_{\mathrm X^n;p,p;2,\infty} < \infty.
\end{equation*}
Then, due to the convergence of $(Y^n,Z^n)$, Fatou's lemma yields that $\thicknm{(Y,-H(Y)) }_{\mathrm X;p,p;2,\infty} < \infty$. Consequently, one can apply Corollary~\ref{cor:uniqueness} to see that $(Y,Z)$ is the unique solution to Eq.~\eqref{e:BSDE}. 

As discussed in \cite[Remark~3.17]{dause2026controlled}, due to the non-explosion condition obtained in \cite[Chapter~11]{friz2010multidimensional}, one may relax the $C^5_b$ assumption on $H(\bm\cdot)$ to $C^{\gamma + 2}_b$ with $\gamma > p$. However, it is still stronger than the condition $H(\bm\cdot) \in C^{\gamma}_b$ required in Theorem~\ref{thm:the existence}. 
\end{remark}

\subsection{The Multidimensional Case}\label{subsec:Multi dimensional}

Apart from the \emph{a priori} bound of Theorem~\ref{thm:BMO bound}, whose proof goes through the comparison principle and is therefore genuinely scalar, the whole argument above is insensitive to the dimension of the state space. We make this precise, since it isolates the single missing ingredient in the multidimensional theory.

Fix $m\ge 1$ and consider the $\R^m$-valued rough BSDE
\begin{equation}\label{e:BSDE multi}
Y_{t}
=
\xi
+
\int_{t}^{T} f(r,Y_{r},Z_{r})\,dr
+
\int_{t}^{T} H(Y_{r})\,d\mathbf X_r
-
\int_{t}^{T} Z_r\,dW_r,
\qquad t\in[0,T],
\end{equation}
where $\mathbf X \in \mathscr{C}^{p\text{-}\mathrm{var}}([0,T],\R^d)$, and
\begin{equation*}
\xi \in L^{\infty}(\mathcal F_T,\R^m), \quad
f: \Omega\times[0,T]\times \R^m \times \mathscr L(\R^d,\R^m) \to \R^m, \quad
H \in C^{\gamma}_b\big(\R^m, \mathscr L(\R^{d},\R^m)\big),
\end{equation*}
with $f$ satisfying \ref{(f)} (with $|\bm\cdot|$ the Euclidean norms of the
respective spaces). Definition~\ref{def:def of BSDEs}, the seminorm
\eqref{e:Psi}, and the solution map of Definition~\ref{def:solu map} are
unchanged, with $\R$ replaced by $\R^m$ and $\mathscr L(\R^d,\R^m)$ where
appropriate. We consider the following hypothesis, in which $k>1$ and $(q,q')$
are as in \ref{(gamma)}.

\begin{assumptionp}{($\mathbf{A}_{\mathbf{apr}}$)}
\noindent \\[0.5em]
\cond{$\mathrm{apr}$}{(apr)}:
There exists a constant $L>0$ such that, for every $s\in[0,T]$, every solution
$(\hat Y,\hat Z)$ of \eqref{e:BSDE multi} on $[s,T]$ with terminal value $\xi$
and with
$\thicknm{(\hat Y,-H(\hat Y))}_{\mathrm X;q,q';k,\infty;[s,T]} < \infty$
satisfies
\begin{equation*}
\|\hat Y\|_{\mathcal L^{\infty}([s,T]\times\Omega)} \le L.
\end{equation*}
\end{assumptionp}

\begin{theorem}\label{thm:multi-d conditional}
Assume \ref{(gamma)}, \ref{(f)}, and \ref{(apr)} for some $k>1$. Let $H\in C^{\gamma}_{b}(\R^m,\mathscr L(\R^d,\R^m))$, $\xi\in L^{\infty}$, and $\mathbf X\in \mathscr{C}^{p\text{-}\mathrm{var}}([0,T], \R^{d})$. 
Then, \eqref{e:BSDE multi} admits a solution $(Y,Z)$ with $(Y,-H(Y))\in \mathrm{RSM}^{(p,p)\text{-}\mathrm{var}}_{\mathrm X}$, and it is the unique solution satisfying $\thicknm{(Y,-H(Y))}_{\mathrm X;q,q';k,\infty} < \infty$. 
Moreover, for every $k'\ge 1$,
\begin{equation}\label{e:multi-d bound}
\thicknm{(Y,-H(Y))}_{\mathrm X;p,p;k',\infty} < \infty,
\end{equation}
and in particular
$\|Y\|_{\infty} + \|\delta (\int_{0}^{\bm\cdot} Z_r dW_r)\|_{\mathrm{BMO}} < \infty$.
\end{theorem}

\begin{proof}
All results of Section~\ref{sec:advanced analysis of rough path} are stated for processes with values in Euclidean spaces $V_1, V_2$, and their proofs use no scalar structure; the same is true of Lemma~\ref{lem:solu map}, Proposition~\ref{prop:invariant within ball}, and Proposition~\ref{prop:contraction}, the martingale representation theorem being applied componentwise. Propositions~\ref{prop:bounded by BMO} and \ref{prop:BMO and ess sup} likewise remain valid verbatim for Eq.~\eqref{e:BSDE multi}: the proof of the former only invokes Corollary~\ref{cor:int of the solution}, Lemma~\ref{lem:Y <= R^Y'}, Lemma~\ref{lem:Lepingle}, and Proposition~\ref{prop:Psi <= Psi + Psi}, while the proof of the latter applies the rough It\^o formula to $|Y_{\bm\cdot}|^2$ and uses the truncations $\tilde H(y) := \kappa(y)\, H(y)^{\top} y$ and $\tilde f(t,y,z) := \kappa(y)\, \langle y, f(t,\kappa(y)y,z)\rangle$, for which \eqref{e:for tilde f} continues to hold. Note that only \ref{(f)} is used, so $f$ may be random.

\emph{Existence.} 
Assumption~\ref{(apr)} provides precisely the bound \eqref{e:a pri esti} required in Step~1 of the proof of Theorem~\ref{thm:the existence}. Hence, for $k' \ge k$, Steps~1--3 of that proof apply with $k'$ in place of $k$, without any other change, and produce a solution on $[0,T]$ satisfying $\thicknm{(Y,-H(Y))}_{\mathrm X;q,q';k',\infty} < \infty$. Therefore, combining \ref{(apr)} with Propositions~\ref{prop:BMO and ess sup} and then \ref{prop:bounded by BMO} yields \eqref{e:multi-d bound}. 

\emph{Uniqueness.} 
Let $(Y,Z)$ and $(\bar Y,\bar Z)$ be two solutions with finite seminorm $\thicknm{\bm\cdot}_{\mathrm X;q,q';k,\infty}$. Let $\hat K>0$ be a constant such that 
\begin{equation*}
\thicknm{(Y,-H(Y))}_{\mathrm X;q,q';k,\infty} + \thicknm{(\bar Y,-H(\bar Y))}_{\mathrm X;q,q';k,\infty} \le \hat K.
\end{equation*}
For some fixed $\varepsilon \in (0,1)$, set $\hat \theta = \hat\theta(k,\gamma,\Theta_0,\hat K,\varepsilon)$ as in Proposition~\ref{prop:contraction}, and choose a partition $\{[t_i,t_{i+1}]\}_{i=0}^{N-1}$ of $[0,T]$ with $|\delta \mathbf X|_{p\text{-}\mathrm{var};[t_i,t_{i+1}]} + |t_{i+1}-t_i|^{1/2} \le \hat\theta$, which is possible by the superadditivity of the control function $\hat w$ used in Step~4 of Proposition~\ref{prop:bounded by BMO}. 
On $[t_{N-1},T]$ both triplets $(Y,-H(Y),Z)$ and $(\bar Y,-H(\bar Y),\bar Z)$ are fixed points of $\Phi^{\xi}_{[t_{N-1},T]}$, hence of $(\Phi^{\xi}_{[t_{N-1},T]})^{\circ 2}$, and they share the terminal value $\xi$; the contraction property therefore forces them to coincide on $[t_{N-1},T]$. Iterating backwards over the partition, with terminal value $Y_{t_{i+1}} = \bar Y_{t_{i+1}}$ at each stage, gives $(Y,Z) = (\bar Y, \bar Z)$ on $[0,T]$. 
\end{proof}

\begin{remark}\label{rem:once the a priori estimate is obtained}
In the scalar case, Theorem~\ref{thm:BMO bound} verifies \ref{(apr)} under the additional hypothesis \ref{(fdet)}, so that Theorem~\ref{thm:multi-d conditional} contains Theorem~\ref{thm:the existence}. The \emph{a priori} bound there is obtained by comparison with the deterministic RDEs of Theorem~\ref{thm:BMO bound}, and the comparison theorem is not available in the multidimensional setting; see Remark~\ref{rem:Why 1-d} for a discussion of why the classical RDE route to such a bound is also obstructed by the presence of the unknown $Z$. Verifying \ref{(apr)} for $m>1$ thus remains the only open point in the multidimensional theory, and we regard it as the natural next question. 

Assumption~\ref{(apr)} holds on a local interval $[t,T]$. Indeed, if $|\delta\mathbf X|_{p\text{-}\mathrm{var};[t,T]} + |T-t|^{1/2}$ is sufficiently small, then by a direct absorption argument from Proposition~\ref{prop:invariant within ball}, the Picard iteration sequence always stays in a ball with respect to the seminorm $\thicknm{\bm\cdot}_{\mathrm X;q,q';k,\infty;[t,T]}$. Moreover, the contraction property shown in Proposition~\ref{prop:contraction} implies that every solution $(\hat Y, \hat Z)$ coincides with the limit of the Picard iteration sequence; hence $\thicknm{(\hat Y,- H(\hat Y) )}_{\mathrm X;q,q';k,\infty;[t,T]}$ has an upper bound. 
We also refer to \cite{liang2025multidimensional} for further multidimensional settings, treated there by the Doss--Sussmann transformation. In their settings, Assumption~\ref{(apr)} is accessible on the whole interval $[0,T]$ provided either $\|\xi\|_{\infty}$ and $|\delta\mathbf X|_{p\text{-}\mathrm{var};[0,T]}$ are sufficiently small, or each component $H_i$ depends only on the corresponding component $y_i$.
\end{remark}

\section{Relation to BDSDEs}\label{sec:BDSDEs}

\subsection{Backward Causality of Rough BSDE Solutions}\label{subsec:causality}
We record a simple but important causality property of rough BSDEs.  Unlike
forward rough differential equations, where the solution up to time \(t\)
depends only on the past rough path \(\mathbf X|_{[0,t]}\), the solution of a rough BSDE on
\([0,T]\) depends, at time \(t\), only on the future increments of the rough
path on \([t,T]\).

For \(0\le t\le T\), denote by $\star$ the group product (also denoted by $\otimes$ in some references; see, e.g., \cite[Chapter~7.2.1]{friz2010multidimensional}) and write
\[
\theta_t\mathbf X
:=
\big(\mathbf X_u^{-1}\star \mathbf X_v\big)_{t\le u\le v\le T}
:=
\big(\delta \mathrm X_{u,v},\mathbb X_{u,v}\big)_{t\le u\le v\le T}
\]
for the translated restriction of \(\mathbf X\) to the time interval
\([t,T]\).  We shall say that a family of processes
\(\{(Y^{\mathbf X},Z^{\mathbf X})\}_{\mathbf X}\) is
\emph{\(\mathbf X\)-backward causal} if, for every \(t\in[0,T]\) and every two
rough paths \(\mathbf X,\widetilde{\mathbf X}\) satisfying
\[
\theta_t\mathbf X=\theta_t\widetilde{\mathbf X},
\]
one has
\[
Y_s^{\mathbf X}=Y_s^{\widetilde{\mathbf X}}
\quad\text{for all }s\in[t,T]\ \text{a.s.},
\qquad
Z_s^{\mathbf X}=Z_s^{\widetilde{\mathbf X}}
\quad\text{for Leb-a.e. }s\in[t,T]\ \text{a.s.}
\]
Equivalently, the restriction of the solution to \([t,T]\) is a function only
of the terminal value \(\xi\), the coefficients \(f,H\), the Brownian motion
\(W\), and the rough path increments \(\theta_t\mathbf X\).

The $\mathbf X$-backward causal property is immediate from the uniqueness. Indeed, let
\((Y^{\mathbf X},Z^{\mathbf X})\) be the solution of \eqref{e:BSDE} driven by
\(\mathbf X\).  For \(s\in[t,T]\), the same pair restricted to \([t,T]\)
satisfies
\[
Y_s^{\mathbf X}
=
\xi
+\int_s^T f(r,Y_r^{\mathbf X},Z_r^{\mathbf X})\,dr
+\int_s^T H(Y_r^{\mathbf X})\,d\mathbf X_r
-\int_s^T Z_r^{\mathbf X}\,dW_r .
\]
The rough integral over \([s,T]\subset[t,T]\) is defined entirely from the
increments \((\delta \mathrm X_{u,v},\mathbb X_{u,v})\), \(t\le u\le v\le T\).  Hence,
if \(\theta_t\mathbf X=\theta_t\widetilde{\mathbf X}\), then
\((Y^{\mathbf X},Z^{\mathbf X})|_{[t,T]}\) and
\((Y^{\widetilde{\mathbf X}},Z^{\widetilde{\mathbf X}})|_{[t,T]}\) solve the
same rough BSDE on \([t,T]\), with the same terminal condition \(\xi\), the
same coefficients \(f,H\), and the same driving restricted rough path.  By
Corollary~\ref{cor:uniqueness}, these two restrictions coincide.

Thus the solution can be represented in the following triangular form: 
\[
(Y_s^{\mathbf X},Z_s^{\mathbf X})_{s\in[t,T]}
=
\Phi_{[t,T]}\big(\xi,f,H,W,\theta_t\mathbf X\big),
\]
where \(\Phi_{[t,T]}\) denotes the solution map of the rough BSDE on the interval \([t,T]\), mapping the terminal condition, coefficients, Brownian motion, and the shifted rough path to the corresponding solution pair. In particular,
the solution is independent of the values of \(\mathbf X\) before time \(t\).  This is
the precise sense in which rough BSDEs are backward causal with respect to their rough
drivers.

\subsection{Randomization and Relation to BDSDEs}\label{subsec:randomization of BDSDEs}
In this section, we show that the solution of rough BSDE~\eqref{e:BSDE} can be randomized and identified with a solution of the following backward doubly stochastic differential equation (BDSDE):
\begin{equation}\label{e:BDSDE}
Y_{t} = \xi + \int_{t}^{T} f(r,Y_r,Z_r) dr + \int_{t}^{T} H(Y_r) \circ \overset{\leftarrow}{d \mathrm B}_r - \int_{t}^{T} Z_r dW_r,\quad t \in [0,T], 
\end{equation}
where $\mathrm B$ is another Brownian motion that is independent of $W$, and 
$$
\int_{t}^{T} H(Y_r) \circ \overset{\leftarrow}{d \mathrm B}_r := \int_{t}^{T} H(Y_r)  \overset{\leftarrow}{d \mathrm B}_r + \frac{1}{2}\int_{t}^{T} DH(Y_r)H(Y_r) dr
$$
is a (backward) Stratonovich integral. Here, $\int H(Y_r)  \overset{\leftarrow}{d \mathrm B_r}$ is a backward It\^o integral, as defined in \cite[Chapter~3.4]{Kunita90}, which is an It\^o integral with respect to the backward filtration $\mathcal F^{\mathrm B}_{t,T} := \sigma\{\delta \mathrm B_{s,T},\ s\in[t,T]\}$. 

Since we are working with random rough paths lifted from Brownian motions, we assume our rough path $\mathbf{X}$ (consider its restriction on $[s,t]$) to take values in the space of geometric rough paths $\mathscr{C}^{0,1/p}_g([s,t],\R^d)$, which we equip with the corresponding Borel $\sigma$-algebra $\mathfrak{C}^{1/p}_{[s,t]}$ for any $0\le s<t \le T$. We write $\mathfrak{C}^{1/p}$ when $s=0$ and $t=T$. Furthermore,
we denote by $\mathfrak P$ the predictable $\sigma$-algebra on $[0,T] \times \Omega$ generated by $W$, and thus a process is predictably measurable if and only if it is measurable with respect to $\mathfrak P$. In the following, $\mathcal B(\mathbf M)$ denotes the Borel $\sigma$-algebra on a metric space $\mathbf M$.

\begin{proposition}\label{prop:measurability of Y^X}
Assume \ref{(gamma)}, $\xi \in L^{\infty}$, and \ref{(fdet)} hold. Let $H \in C^{\gamma}_b (\R,\R^d)$. For any $\mathbf X = (\mathrm X,\mathbb X) \in \mathscr C^{0,1/p}_{g}$, denote by $(Y^{\mathbf X},Z^{\mathbf X})$ the unique solution to Eq.~\eqref{e:BSDE} on $[0,T]$ obtained by Theorem~\ref{thm:the existence}. Then, there exists a version $(\tilde Y,\tilde Z )(t,\omega;\mathbf X)$ of $(Y^{\mathbf X},Z^{\mathbf X})(t,\omega)$ in the sense that
\begin{itemize}
\item[(1)] the map $(t,\omega,\mathbf X) \mapsto (\tilde Y,\tilde Z )(t,\omega; \mathbf X)$ is $\mathcal B([0,T]) \otimes \MF_T \otimes \mathfrak C^{1/p}$-measurable, and $(\MF_t  \otimes \mathfrak C^{1/p} )_t$-adapted; 
\vspace{1mm}
\item[(2)] for any $\mathbf X \in \mathscr C^{0,1/p}_{g}$, it holds that $\tilde Z(t,\omega;\mathbf X) = Z^{\mathbf X}(t,\omega)$ a.e.-$dt \times d\mathbb P$, and that $\tilde Y(t,\omega;\mathbf X) = Y^{\mathbf X}(t,\omega)$ for all $t\in[0,T]$, a.s.;
\vspace{1mm}
\item[(3)] $(\tilde Y ,\tilde Z )(\cdot;{ \mathbf X})$ is $\BX$-{\it backward causal} in the sense that for any $t\in [0,T)$, and $(\tilde Y,\tilde Z)$ restricted on $[t,T]$, we have
\begin{equation*}
\begin{split}
&\tilde Z ( \cdot \ ;\mathbf X) =  \tilde Z ( \cdot \ ;\theta_t\mathbf X ) \  \text{ a.e.-Leb}|_{[t,T]} \times d\mathbb P;\\
&\tilde Y(s,\cdot; \mathbf X) = \tilde Y (s,\cdot; \theta_t\mathbf X )\  \text{ for all }\  s\in[t,T], \ a.s.   
\end{split}
\end{equation*}
\end{itemize}
\end{proposition}

\begin{proof}
The joint measurability can be established by adapting the argument of \cite{FLZ26,BS25}, which proceeds in two steps: first, one proves the existence of a jointly measurable version when \(\BX\) is smooth; second, one obtains a jointly measurable version of the original solution by an approximation argument. However, here we take a different approach motivated by \cite[Theorem~IV.30]{Meyer1978}. We equip the space $\mathscr{C}^{0,1/p}_g$ with the filtration $\mathbb C:=(\mathcal C_t)_t$ and $\sigma$-algebra $\mathcal C_t:=  \sigma(\{\BX_{s \vee . }^{-1}\star \BX_T | \ \BX \in \mathscr{C}^{0,1/p}_g, T-t \le s \le T  \})$.

Consider the mapping $$(t,\BX) \mapsto (Y^{\BX}, Z^{\BX} )((T-t)\vee\cdot, \cdot ) \in \mathcal L^0([0,T] \times \Omega, \FP, \text{Leb} \times \PP ; \R^{1+d}),$$ with $\mathcal L^0([0,T] \times \Omega, \FP, \text{Leb} \times \PP; \R^{1+d})$ equipped with the weak topology and the trivial filtration $( \mathcal P_t)_t:=(\FP)_t$. 

By Theorem \ref{thm:stability}, we see that the above mapping is continuous, and thus $\big(\mathcal B([0,T]) \otimes \FC^{1/p}\big)/ \mathcal B (\mathcal L^0) $-measurable. Furthermore, it is easy to check that the mapping is $\mathbb C$-progressively measurable. Moreover,   noting that $\FP$ is generated by the Brownian motion $W$, and thus it is separable. Then according to \cite[Proposition~3.9]{DKP24} with  $(\Omega, \mathcal F, \mathbb F) $ and  $(\Omega',  \mathcal G, \mathbb G, \PP)$ therein taken as $(\mathscr{C}^{0,1/p}_g,  \FC^{1/p}, \mathbb C)$ and $([0,T] \times \Omega, \FP, (\FP)_t, \text{Leb}\times\PP)$ respectively, there exists a $\big(\mathcal B([0,T]) \otimes \FC^{1/p} \otimes  \FP \big) /\mathcal B( \R^{1+d} )$-measurable, $\mathbb C \otimes (\FP)_t$-adapted process $$(\bar Y, \bar Z):[0,T] \times \mathscr{C}^{0,1/p}_g \times [0,T] \times \Omega \rightarrow   \R^{1+d },$$ such that for any $(t,\BX),$ 
$$
(\bar{Y}, \bar{Z})(t,\BX,\cdot)= (Y^{\BX}, Z^{\BX})((T-t)\vee \cdot, \cdot) \text{ in } \mathcal L^0([0,T]\times \Omega, \FP,  \text{Leb} \times \PP).
$$
Now let $\{\varepsilon_n\}_{n\ge 1}$ be a sequence of strictly positive numbers that converges to $0$. We take 
\begin{equation}
(\tilde Y, \tilde Z)(t,\omega, \BX):= \limsup_{n \to \infty} \frac{1}{\varepsilon_n} \int_{t}^{t+\varepsilon_n} (\bar{Y}, \bar{Z})(T-t, \BX, s, \omega)ds, 
\end{equation}
and thus $(\tilde Y, \tilde Z)$ is $(\mathcal C_{T-t} \otimes \MF^W_{t+})_t$-adapted. Moreover, for any fixed $\BX$, and $s\ge t,$ 
$$
(\bar{Y}, \bar{Z})(T-t, \BX, s, \omega)= (Y,Z)^{\BX}(t\vee s, \omega)= (Y^{\BX},Z^{\BX}) (s,\omega), \ \ ds\otimes d\PP\text{-a.s.}
$$
which implies, by Lebesgue's differentiation theorem, $(\tilde Y, \tilde Z)(t,\omega, \BX)= (Y,Z)^{\BX}(t,\omega), \ dt\otimes d\PP$-a.s. Thus it is easy to check all the three conditions from the proposition hold.
\end{proof}

In the following, we replace the rough path $\BX$ by the randomized rough path $\BB$. For simplicity of notation, we work with the product space. 
Consider 
\begin{equation}  \label{sec51productbasis}
( {\Omega}, \MG, \PP)  = (\Omega', \MG',\PP') \otimes (\Omega'', \MG'', \PP''),
\end{equation} 
where $(\Omega', \MG',\PP')$ supports a $d$-dimensional Brownian motion $W$, %$(\FP')_t$ is the predictable $\sigma$-algebra, 
and $(\Omega'', \MG'',\PP'')$ is a probability space with a $d$-dimensional Brownian motion $\mathrm B$. Moreover, we denote by $(\MF')_t$ the augmented filtration of $\MF^W$, $\MF''_{t,T}:= \sigma(\mathrm B_s- \mathrm B_t| s\in [t,T])$, $\MF''_t:=\MF''_{0,t}$, and $\MF_t:= \MF'_t \vee \MF''_{t,T}$ the standard non-monotone family of $\sigma$-algebras for BDSDEs. Denote by $\E'$ and $\E''$ the expectation of $(\Omega',\mathcal G',\mathbb P')$ and $(\Omega'',\mathcal G'',\mathbb P'')$, respectively. In the following, we write $\omega = (\omega',\omega'') \in \Omega$.

In view of \cite[Chapter~3]{friz2020course}, for each $(s,t)\in \Delta_T$, define $\mathbb{B}^{\text{It\^o}}_{s,t}=\int_s^t \delta \mathrm B_{s,r}d\mathrm B_r$ as an It\^o integral and $\mathbb{B}_{s,t}^{\mathrm{Strato}}:=\int_s^t \delta \mathrm B_{s,r}\circ d\mathrm B_r=\mathbb{B}^{\text{It\^o}}_{s,t}+\frac12 (t-s)I_d$ as a Stratonovich integral.
It is well-understood that $\BB^{\mathrm{Strato}}(\omega'')=(\mathrm B(\omega''),\mathbb{B}^{\mathrm{Strato}}(\omega'')) \in \cap_{\alpha\in(1/3,1/2)} \mathscr{C}^{0,\alpha}_g $ for all $\omega''\in N_1^c$, where $N_1$ is a $\PP''$-null set of $\Omega''$. Moreover, for any $t\in[0,T]$, the {\it lifting mapping} 
\begin{equation}\label{ito-lift}
\begin{array}{llll}
\BB := (\mathrm B, \mathbb B) := \BB^{\mathrm{Strato} }: & [0,T]\times \Omega'' & \rightarrow & (\mathscr{C}^{0,\alpha}_g, \FC^{1/p})	\\[2mm]
 & (t,\omega'') &\mapsto & \BB^{\mathrm{Strato} } (\omega'')_{.\wedge t}
 \end{array}   
\end{equation} 
is progressively measurable w.r.t. $(\MF''_t )_t$. For the backward Stratonovich integral and rough integral, we have

\begin{proposition}\label{prop:backward-equiv} 
Suppose that $(Z,Z')$ defined on $ [0,T] \times \Omega' \times \mathscr{C}^{0,1/p}_g  $ is $\mathcal B([0,T]) \otimes \MF'_T \otimes \mathfrak C^{1/p}$-measurable and $(\MF'_t \otimes \mathfrak C^{1/p})_t$-adapted satisfying the following conditions:
\begin{itemize}
\item[(1)] for any $\BX  \in \mathscr{C}^{0,1/p}_g,$ $(Z,Z')^{\BX}(\cdot):= (Z,Z')(\ \cdot \ ; \BX)   \in \mathrm{RSM}^{(p,p')\text{-}\mathrm{var}}_{\mathrm X}(\Omega')$, where $p' \ge p \ge 2$ such that \( 2/p+ 1/{p'}>1\);

\

\item[(2)] $(Z, Z', M^Z)^{\BX}$ is $\BX$-{\it backward causal}.
\end{itemize}
Then for any $t\in[0,T],$ 
\begin{equation}\label{eq:rough=ito} 
\begin{split}
\Big(  \int_t^T (Z, Z')^{\mathbf{X}} (\omega') d
\mathbf{X} \Big)   \bigg|_{\mathbf{X} = \mathbf{B} (\omega'')} 
&= \int_t^T Z^{\mathbf{B} (\omega'')}_r (\omega')   \overset{\leftarrow}{d \mathrm B}_r (\omega'') - \frac{1}{2} \int_t^T (Z')^{\BB(\omega'')}_r(\omega') dr\\
&=: \int_t^T Z^{\mathbf{B} (\omega'')}_r (\omega') \circ \overset{\leftarrow}{d \mathrm B}_r (\omega''), \ \ \ \PP\text{-a.s.}
\end{split}
\end{equation}

\end{proposition}

\begin{proof}
Let $(\bar Z, \bar Z')(\omega', \omega''):=(Z,Z')^{\BB(\omega'')}(\omega')$.
Since $(Z,Z')^{\BX}$ is $(\MF'_t \otimes \FC^{1/p})_t$-adapted, continuous, and $\BX$-backward causal, $\bar Z $ is $(\MF_t)_t$-adapted, and thus the backward Stratonovich integral on the left-hand side of \eqref{eq:rough=ito}, is well-defined in the sense that 
$$
\int_t^T \bar Z_s  \circ \overset{\leftarrow}{d \mathrm B}_s  = \lim_{|\MP| \rightarrow 0} \sum_{[u,v] \in \MP} \left(\bar Z_v \delta \mathrm B_{u,v}- \frac12 \bar Z'_u(v-u) \right), \ \ \ \text{in $\PP$-probability},
$$
where $\MP$ is any interval partition of $[t,T].$ 

On the other hand, by a modification of \cite[Theorem~3.4]{FLZ26-2}, there exists an $\MF'_T \otimes \FC^{1/p}$-measurable, $\BX$-backward causal version of $\int_t^T (Z,Z')^{\BX}\,d\BX$ such that, for every $\BX$, it coincides with $\int_t^T (Z,Z')^{\BX}\,d\BX$ $\PP'$-almost surely. We still denote this version  by $\int_t^T(Z,Z')^{\BX}d\BX$. Then by \cite[Corollary~8.7]{FLZ26}, we have 
$$
\int_t^T (Z,Z')^{\mathbf{X}  }d\BX \Big|_{\BX=\BB} = \lim_{|\MP| \rightarrow 0} \sum_{[u,v] \in \MP} (\bar Z_u \delta \mathrm B_{u,v}+ \bar Z'_u \B_{u,v} ), \ \ \ \text{in $\PP$-probability.}
$$

To complete the proof, we only need to show the difference of right-hand sides of the above two identities 
\begin{align}\nonumber
&  \sum_{[u,v] \in \MP } (\delta \bar Z_{u,v} \delta \mathrm B_{u,v} - \bar Z'_u \B_{u,v} -\frac12 \bar Z'_u(v-u))\\ \nonumber
= & \sum_{[u,v] \in \MP }  \left[ \bar Z'_u (\delta \mathrm B_{u,v} \delta \mathrm B_{u,v} - (v-u)I_d) \right]  -  \sum_{[u,v] \in \MP } \bar Z'_u \B^{\text{It\^o}}_{u,v} + \sum_{[u,v] \in \MP } R^{\bar Z}_{u,v} \delta \mathrm B_{u,v}\\ \label{eq:decomp-roughsemi-inte}
=: &\sum_{[u,v] \in \MP }  \left[ \bar Z'_v (\delta \mathrm B_{u,v} \delta \mathrm B_{u,v} - (v-u) I_d) \right]  -  \sum_{[u,v] \in \MP } \bar Z'_v \B^{\text{It\^o}}_{u,v} + \sum_{[u,v] \in \MP } R^{\bar Z}_{u,v} \delta \mathrm B_{u,v} + \sum_{[u,v] \in \MP } \Delta_{u,v}
\end{align}
converges to zero in $\PP$-probability as $|\MP|\rightarrow 0$. It is standard to show $\sum_{[u,v]\in \MP} \Delta_{u,v}\to 0$  in $\PP$-probability. Moreover, since $\bar Z'$ is $(\MF_t)_t$-adapted, we have that the first summation of \eqref{eq:decomp-roughsemi-inte} converges to zero. Similarly, the second summation also converges to zero. It remains to show the third term converges to zero. Indeed, noting that $ R^{\bar Z}_{u,v} = P^{\bar Z}_{u,v} + \delta M^{\bar Z}_{u,v}$, the third summation equals
\begin{equation}\label{eq:fur-decom}
\sum_{[u,v]\in \MP} P^{\bar Z}_{u,v} \delta \mathrm B_{u,v} + \sum_{[u,v]\in \MP} \delta M^{\bar Z}_{u,v} \delta \mathrm B_{u,v}.   
\end{equation}
For the first summation of \eqref{eq:fur-decom}, since  $P^{Z}(\omega';\BX) \in C^{\frac{pp'}{p+p'}\text{-var}},$ $\PP'$-a.s., we have 
$$
\sum_{[u,v]\in \MP}   P^{Z}_{u,v}(\omega';\BX) \delta \mathrm X_{u,v} \rightarrow 0, \ \ \  \PP'\text{-}a.s.
$$
and thus the term involving $P^{\bar Z}(\omega)= P^Z(\omega';\BX)|_{\BX=\BB(\omega'')}$ times $\delta \mathrm B$ converges to zero $\PP$-a.s. For the second summation of \eqref{eq:fur-decom}, without loss of generality, assume that $\delta M^Z (\omega';\BX)$ is a square-integrable martingale on $(\Omega', \MF', \PP')$; otherwise, one may apply a localization argument first. By the Burkholder-Davis-Gundy inequality, we have 
$$
\E'\Big[ \Big|\sum_{[u,v]\in \MP} \delta  {M}^Z_{u,v} \delta \mathrm X_{u,v} \Big|^2 \Big] \lesssim \sum_{[u,v] \in \MP} |\delta \mathrm X_{u,v}|^2 \E' \left[|\delta \langle M^Z \rangle_{u,v}|\right] \lesssim \sup_{[u,v] \in \MP} |\delta \mathrm X_{u,v}|^2 \E'\left[ | \delta \langle M^Z \rangle_{0,T} |\right] \rightarrow 0. 
$$
It follows that the second summation of \eqref{eq:fur-decom} converges to zero in $\PP$-probability as well.    
\end{proof}

Now we are ready to establish the equivalence between rough BSDEs and BDSDEs. In view of \cite{PardouxPeng1994BDSDE}, we call a pair $(Y,Z) $ a square-integrable solution to \eqref{e:BDSDE} if $(Y,Z)$ satisfies the equation and $(Y,Z)\in \mathcal S^2([0,T]\times \Omega) \times \mathcal L^2([0,T]\times \Omega)$, where $\mathcal S^2$ is the set of continuous adapted processes $y$ satisfying $\|y\|_{\mathcal S^2}^2:= \E[\sup_{t\in [0,T]} |y_t|^2]<\infty. $  

\begin{theorem}\label{thm:BDSDE}
Under the same assumption as in Proposition~\ref{prop:measurability of Y^X} but with the probability space $(\Omega, \MF, \PP)$ there replaced by $(\Omega', \MF', \PP')$. Let $(  Y^{\mathbf X}, Z^{\mathbf X})$ be the jointly measurable, $\BX$-backward causal version obtained in Proposition~\ref{prop:measurability of Y^X}. Denote by $(\bar Y, \bar Z) (\omega):=  (Y^{ \mathbf B(\omega'')}(\omega'), Z^{\BB(\omega'')}(\omega') )$.
Then $(\bar Y, \bar Z)$ is the unique $L^2$-integrable solution to BDSDE~\eqref{e:BDSDE}. 
\end{theorem}

\begin{proof}
We complete the proof in two steps. First, we show that $(\bar Y,\bar Z)$ satisfies \eqref{e:BDSDE} $\PP$-almost surely. We then prove that $(\bar Y,\bar Z)\in\mathcal S^2\times\mathcal L^2$. It therefore follows from \cite[Theorem~1.1]{PardouxPeng1994BDSDE} that $(\bar Y,\bar Z)$ is the unique square-integrable solution.

{\bf Step 1.} Since $(\bar Y,\bar Z)$ is $(\MF_t)_t$-adapted, we may argue exactly as in the proof of \cite[Proposition~8.2]{FLZ26} to obtain
\begin{align*}
&\int_{t}^{T} Z^{\mathbf X}_r  dW_r \Big|_{\BX=\BB(\omega'')} = \int_{t}^{T} \bar Z_r dW_r, \quad \int_{t}^{T} f(r,Y^{\mathbf X}_r ,Z^{\mathbf X}_r ) dr\Big|_{\BX=\BB(\omega'')} = \int_{t}^{T} f(r, \bar Y_r, \bar Z _r) dr . 
\end{align*}
To show that $(\bar Y, \bar Z)$ satisfies \eqref{e:BDSDE}, it suffices to verify the following  identity
$$
\int_{t}^{T}H(Y^{\mathbf X}_r) d\mathbf X_{r} \Big|_{\BX=\BB(\omega'')} = \int_{t}^{T}H( \bar Y_r) \circ \overset{\leftarrow}{d \mathrm B}_r,
$$ 
which indeed follows from Proposition \ref{prop:backward-equiv} by properly choosing $(p,p')$.

{\bf Step 2.} Now we prove the square-integrability of $(Y^{\mathbf B},Z^{\mathbf B})$. By the proof of Theorem~\ref{thm:BMO bound}, we have that $Y^{\mathbf X}$ can be dominated by the solutions of RDEs. By Proposition~\ref{prop:RDE with drift}, the uniform norm of these solutions of RDEs is bounded by a constant multiple of $1 + |\delta \mathbf X|^p_{1/p}$; hence the same bound applies to $Y^{\mathbf X}$. Consequently, we have the following estimate for the randomization $Y^{\mathbf B}$: 
\begin{equation*}
\esssup_{\omega' \in \Omega'} \Big\{ \sup_{t \in [0,T]}|Y^{\mathbf B(\omega'')}_t(\omega')|\Big\} \lesssim 1 + |\delta \mathbf B|^p_{1/p}(\omega''),\ \text{ for all }\ \omega'' \in \Omega''. 
\end{equation*}
Then, the exponential integrability of $|\delta \mathbf B|_{1/p}$ (see, e.g., \cite[Corollary~13.14]{friz2010multidimensional}) implies that 
\begin{equation*}
\E \Big[\sup_{t \in [0,T]} |Y^{\mathbf B}_t|^2 \Big] < \infty.
\end{equation*}

For the integrability of $\int_{0}^{T}|Z^{\mathbf B}_r|^2 dr$, by It\^o's formula (see, e.g., \cite[Lemma~2.1]{ssz19}), it holds that 
\begin{equation*}
\begin{aligned}
&\E'\Big[ |Y^{\mathbf B}_{0}|^2 +  \int_{0}^{T} |Z^{\mathbf B}_r|^2 dr \Big]\\
&= \E'\Big[|\xi|^2 +  \int_{0}^{T} \left(  Y^{\mathbf B}_r \big( 2 f(r,Y^{\mathbf B}_r,Z^{\mathbf B}_r) +  DH(Y^{\mathbf B}_r) H(Y^{\mathbf B}_r) \big) + |H(Y^{\mathbf B}_r)|^2 \right) dr + 2\int_{0}^{T} Y^{\mathbf B}_r H(Y^{\mathbf B}_r) d\overset{\leftarrow}{\mathrm B}_r \Big].
\end{aligned}
\end{equation*}
By \ref{(f)} and Young's inequality, we have that $2 |yf(r,y,z)| \le 2 C_{f} (|y| + |y|^2 + |yz|) \le \tilde C (1 + |y|^2) + \frac{1}{2} |z|^2$ for some constant $\tilde C > 0$. Furthermore, we also have $|\xi|^2 \le \sup_{s\in[0,T]} |Y_s|^2$, $|y DH(y) H(y)| \le \hat C + |y|^2$, and $|H(y)|^2 \le \hat C $ for some constant $\hat C > 0$.  Hence, for some constant $C_0 > 0$
\begin{equation}\label{e:1/2 Z^2}
\frac{1}{2}\E'\Big[\int_{0}^{T} |Z^{\mathbf B}_r|^2 dr \Big] \le C_0 + C_0 \E'\Big[\sup_{s\in[0,T]}|Y^{\mathbf B}_{s}|^2\Big] + 2\E'\Big[\int_{0}^{T}Y^{\mathbf B}_r H(Y^{\mathbf B}_r) d\overset{\leftarrow}{\mathrm B}_r\Big].  
\end{equation}
Since $H$ is bounded, we have that $\sup_{s\in[0,T]}|Y^{\mathbf B}_s H(Y^{\mathbf B}_s)|^2$ is integrable. Then, by It\^o's isometry, $\int_{t}^{T} Y^{\mathbf B}_r H(Y^{\mathbf B}_r) d \overset{\leftarrow}{\mathrm B}_{r}$ is a backward square-integrable martingale with respect to $(\mathcal F''_{t,T})_t$, and hence 
\begin{equation*}
\E\Big[\int_{0}^{T}Y^{\mathbf B}_r H(Y^{\mathbf B}_r) d\overset{\leftarrow}{\mathrm B}_r \Big] = 0. 
\end{equation*}
Consequently, taking $\E''[\bm\cdot]$ in \eqref{e:1/2 Z^2} and applying Fubini's theorem, we obtain
\begin{equation*}
\E\Big[ \int_{0}^{T} |Z^{\mathbf B}_r|^2 dr \Big] \lesssim 1 + \E\Big[ \sup_{s\in[0,T]} |Y^{\mathbf B}_s|^2 \Big] < \infty,
\end{equation*}
where the last inequality follows from the square-integrability of $\sup_{s\in[0,T]}|Y^{\mathbf B}_s|$. 
\end{proof}

\appendix

\section{Auxiliary Results on Rough and Rough Stochastic Analysis}\label{sec:Appendix A}

In this section, we collect some results in rough and rough stochastic analysis.

The following L\'epingle's inequality for $p$-variation characterizes the $k$-moments of $p$-variation of martingales by its quadratic process.

\begin{lemma}\label{lem:Lepingle}
Assume $p>2$, $k\ge 1$, and $t\in[0,T]$. Let $M$ be a continuous martingale such that $\E[\delta \langle M\rangle_{t,T}^{k/2}]<\infty$.  Then, there exists a positive number $C$ depending only on $p$ and $k$ such that
\begin{equation}\label{e:BDG}
C^{-1}\E_{t}[\delta \langle M\rangle_{t,T}^{\frac k2}]\le \E_{t}[\|\delta M\|^{k}_{p\text{-}\mathrm{var};[t,T]}] \le C\E_{t}[ \delta \langle M\rangle_{t,T}^{\frac k2}].
\end{equation}
Consequently, there exists a positive constant $\hat C$ depending only on $p$ and $k$ such that 
\begin{equation}\label{e:BMO vs p-var}
\hat C^{-1}\|\delta M\|_{\mathrm{BMO}} \le \left\|\|\delta M\|_{p\text{-}\mathrm{var}}\right\|_{\infty;k,\infty} \le \hat C \|\delta M\|_{\mathrm{BMO}}.
\end{equation}
\end{lemma}

\begin{proof}
The proof of \eqref{e:BDG} can be found in \cite[Theorem~14.12]{friz2010multidimensional}. For \eqref{e:BMO vs p-var}, by \cite[Corollary~2.1]{Kazamaki1994} it follows that there exists a constant $\tilde{C}$ depending only on $k$ such that
\begin{equation}\label{e:k-BMO vs BMo}
\tilde{C}^{-1}\|\delta M\|^k_{\mathrm{BMO}} \le \esssup_{(t,\omega)\in[0,T]\times\Omega} \E_{t}[\delta \langle M\rangle_{t,T}^{\frac k2}] \le \tilde{C}\|\delta M\|^k_{\mathrm{BMO}}.
\end{equation}
Hence, by taking essential supremum to \eqref{e:BDG} and applying \eqref{e:k-BMO vs BMo}, we have \eqref{e:BMO vs p-var}.
\end{proof}

The following Lemma extends  \cite[Theorem~1.1]{friz2023rough} from the usual $L^k$-norm to the conditional $L^{k}$-norm, which is essential for estimating the remainder of the composition of the nonlinear function $H(\bm\cdot)$ and $(Y,Y') \in \mathrm{RSM}^{(p,p')\text{-}\mathrm{var}}_{\mathrm X}$. 

\begin{lemma}\label{lem:p-var martingale estim}
Let $k\ge 1,\ 1<k_1\le \infty,\ 1\le k_{0}<\infty$ and $0<r,r_1 < \infty$. Suppose $1/r < 1/r_1 + 1/2$, and $1/k = 1/k_{1} + 1/k_{0}$. Let $(F_{t})_{0\le t\le T}$ be a continuous adapted process and $M_{t}$ be a continuous martingale. Denote by
\begin{equation*}
A_{s,t} := \int_{s}^{t}\delta F_{s,r} dM_{r}.
\end{equation*}
Then, there is a deterministic constant $C$ depending only on $r,r_1,k,k_1,k_0$ such that 
\begin{equation*}
\big\|\|A_{\bm\cdot,\bm\cdot}\|_{r\text{-}\mathrm{var};[s,t]} | \mathcal F_{s} \big\|_{k} \le C \big\|\|\delta F\|_{r_1\text{-}\mathrm{var};[s,t]} | \mathcal F_{s} \big\|_{k_1} \big\|\|\delta M_{s,\bm\cdot}\|_{\infty;[s,t]} | \mathcal F_{s} \big\|_{k_0}.
\end{equation*}
In addition, for $r_1 = \infty$ and $1/r < 1/2$, we have
\begin{equation*}
\big\|\|A_{\bm\cdot,\bm\cdot}\|_{r\text{-}\mathrm{var};[s,t]} | \mathcal F_{s} \big\|_{k} \le C \big\|\|F\|_{\infty;[s,t]} | \mathcal F_{s}  \big\|_{k_1} \big\|\|\delta M_{s,\bm\cdot}\|_{\infty;[s,t]} | \mathcal F_{s}  \big\|_{k_0}.
\end{equation*}
\end{lemma}

\begin{proof}
Assume at first that the underlying probability space is Borel which
guarantees the existence of a regular conditional probability $\mathbb P_s
(\omega) = \mathbb P (\bm\cdot | \mathcal F_s )$. Obviously $M$
restricted to $[s, t]$ is a  ($\mathbb P_s , (\mathcal F_r : r \geqslant s$))
-martingale, almost surely, so that the conditional estimate follows
immediately from the unconditional one. One can also avoid the use of regular conditional probabilities. Indeed, most arguments in the proof of \cite[Theorem~1.1]{friz2023rough} can be adapted to versions involving conditional expectations. The key point is that the vector-valued BDG inequality shown in \cite[Lemma~2.6]{friz2023rough} remains valid in the conditional setting. The extension to conditional estimates relies solely on standard scalar probabilistic arguments (specifically, localization via $\mathcal F_s$-measurable indicator functions, which amounts to replacing each martingale $(h^k_n: n \ge 0)$, in notation of \cite[Lemma~2.6]{friz2023rough}, by $h^k_n 1_O$, time-$0$ measurable set $O$, which yields immediately the required conditional form of \cite[Lemma~2.6]{friz2023rough}.) 
\end{proof}

\begin{lemma}[RDEs with BV drifts]
\label{lem:BV perturbed RDE}
Let $p\in(2,3)$ and 
$\mathbf X\in\mathscr C^{p\text{-}\mathrm{var}}([0,T],\R^d)$.
Let $A\in C^{1\text{-}\mathrm{var}}([0,T],\R^e)$, with $A_0 =0$, $V\in C_b^{2}(\R^e;\mathscr L(\R^d,\R^e))$, and assume $(Y, V(Y)) \in \mathscr D^{(p,p)\text{-}\mathrm{var}}_{\mathrm X}([0,T],\R^e)$ satisfies 
\[
Y_t = Y_0 + A_t +\int_0^t V(Y_r)\,d\mathbf X_r.
\]
Then, with $\rho(s,t) := \triplenorm{\delta \mathbf X}^p_{p\text{-}\mathrm{var};[s,t]} + \|\delta A\|_{1\text{-}\mathrm{var};[s,t]}$,
\[
\|\delta Y\|_{p\text{-}\mathrm{var};[s,t]}
\le C\bigl(\rho(s,t)^{1/p}+\rho(s,t)\bigr),
\]
where $C$ depends only on $p,d,e$, and $\|V\|_{C_b^{2}}$.
\end{lemma}

\begin{proof}
Equipping $(A,\mathrm X)$ with its canonical joint rough-path lift, the mixed iterated integrals being defined by Riemann--Stieltjes integration. By basic Riemann--Stieltjes estimates, whenever $\rho(s,t)\le1$, this lift is controlled by $C_p \, \rho$ in $p$-variation, with $C_p > 0$ depending only on $p$. Hence the standard local RDE estimate (see, e.g., \cite[Lemma~10.7]{friz2010multidimensional}), applied to the vector fields $(I_e,V(\bm\cdot))$, yields constants $\eta\in(0,1]$ and $C_0$, depending only on $p,d,e$ and $\|V\|_{C_b^{2}}$, such that
\[
\|\delta Y\|_{p\text{-}\mathrm{var};[u,v]}
\le C_0 \rho(u,v)^{1/p}
\quad\text{whenever }\rho(u,v) \le \eta.
\]

Since $\rho$ is a control function, a greedy partition of $[s,t]$ at level $\eta$ has at most $1+\eta^{-1}\rho(s,t)$ intervals. Also, for $(u,v) \in \Delta_{[s,t]}$ satisfying $\rho(u,v) \le \eta$, it holds that 
\begin{equation*}
\rho(u,v)^{1/p} \bigl( 1 + \eta^{-1}\rho(s,t) \bigr) \le \rho(u,v)^{1/p} + \eta^{(1-p)/p} \rho(s,t) \le (1 + \eta^{(1-p)/p}) \bigl( \rho(s,t)^{1/p} + \rho(s,t) \bigr). 
\end{equation*}
Therefore, summing the local estimates proves
\[
\|\delta Y\|_{p\text{-}\mathrm{var};[s,t]}
\le C_0 (1 + \eta^{(1-p)/p}) \bigl( \rho(s,t)^{1/p}+\rho(s,t) \bigr).
\qedhere
\]
\end{proof}

The following proposition offers an estimate for RDEs with unbounded drift, seemingly unavailable in the existing literature, including \cite[Chapter~12]{friz2010multidimensional}, \cite[Theorem~3.4]{teichmann2011another}, and \cite{RiedelScheutzow2017}.

\begin{proposition}\label{prop:RDE with drift}
Assume $p\in(2,3)$, $\gamma>p$, and let
$\mathbf X\in\mathscr C^{p\text{-}\mathrm{var}}([0,T],\R^d)$. Let $F(t,\cdot)$ be a uniformly Lipschitz function with $\sup_{t \in [0,T]}|F(t,0)| < \infty$ and $H \in C_b^{\gamma}(\R^e;\mathscr L(\R^d,\R^e))$.
Then, the backward RDE
\begin{equation}\label{e:backward RDE}
\psi_t
=
y+\int_t^T F(r,\psi_r)\,dr
+
\int_t^T H(\psi_r)\,d\mathbf X_r
\end{equation}
admits a unique solution with
$(\psi,-H(\psi))\in\mathscr D_{\mathrm X}^{(p,p)\text{-}\mathrm{var}}([0,T],\R^e)$.
Moreover, if
\[
|y|
+
\sup_{t\in[0,T]}|F(t,0)|
+
\sup_{t\in[0,T]}
\|\delta F(t,\bm\cdot)\|_{1\text{-}\mathrm{H\ddot{o}l}}
+
\|H\|_{C_b^{2}} + T
\le M,
\]
then for some constant $C_{p,d,e,M}$ depending only on $p,d,e$ and $M$, it holds that 
\begin{equation}\label{e:estimate for RDEs}
\|\delta \psi\|_{p\text{-}\mathrm{var};[0,T]} 
\le
C_{p,d,e,M} 
\Big(
1+
\triplenorm{\delta \mathbf X}_{p\text{-}\mathrm{var};[0,T]}^{\,p}
\Big).
\end{equation}
\end{proposition}

\begin{proof}
We first prove the estimate \eqref{e:estimate for RDEs}. Assume $(\psi,-H(\psi)) \in \mathscr D^{(p,p)\text{-}\mathrm{var}}_{\mathrm X}([0,T],\R^e)$ is a solution to Eq.~\eqref{e:backward RDE}. 
Set
$A_t:=-\int_0^t F(r,\psi_r)\,dr$. Then, by applying 
Lemma~\ref{lem:BV perturbed RDE} backwardly on $[t,T]$ with vector field $V := - H$, it yields 
\[
\|\delta \psi\|_{p\text{-}\mathrm{var};[t,T]}
\le
\bar C_{p,d,e,M}
\bigl(\rho(t,T)^{1/p} + \rho(t,T)\bigr),
\quad \text{where }  
\rho(t,T)
:=
\triplenorm{\delta \mathbf X}_{p\text{-}\mathrm{var};[t,T]}^{\,p}
 +
 \int_t^T|F(r,\psi_r)|\,dr,
\]
and $\bar C_{p,d,e,M} > 0$ is a constant depending only on $p,d,e$, and $M$. 
Using $u^{1/p}\le1+u$, 
$|F(r,\psi_r)|\le M(1+|\psi_r|)$, and $|\psi_r| \le \|\delta \psi\|_{p\text{-}\mathrm{var};[r,T]} + |y|$, we obtain that for some constant $\tilde C_{p,d,e,M} > 0$, 
\[
\|\delta \psi\|_{p\text{-}\mathrm{var};[t,T]}
\le
\tilde C_{p,d,e,M} \Big(1 + \triplenorm{\delta \mathbf X}_{p\text{-}\mathrm{var};[t,T]}^{\,p}\Big) 
+
\tilde C_{p,d,e,M} \int_t^T \|\delta \psi\|_{p\text{-}\mathrm{var};[r,T]} \,dr.
\]
The desired estimate follows from Gronwall's inequality (in backward form).

After obtaining the \emph{a priori} estimate for the solution, one can prove the well-posedness of the RDE via a truncation argument. Indeed, when $F$ is bounded, the condition $H \in C^{\gamma}_b$ is sufficient to guarantee well-posedness for rough SDEs (see \cite{fhl21}), and hence for RDEs. Define the truncated function $F^{n}(t,y) := F(t,\pi_n (y))$ with 
\begin{equation*}
\pi_n(y):=
\begin{cases}
y, & |y|\le n,\\
n y/{|y|}, & |y|>n.
\end{cases}
\end{equation*}
Then, the \emph{a priori} estimate ensures that the corresponding truncated solution is, in fact, a solution for sufficiently large $n$; the well-posedness then follows. 
\end{proof}

\begin{lemma}\label{lem:gamma - 2}
Assume $\gamma \in (2,3]$ and $G:\R\to \R$ satisfies $DG \in C^{\gamma - 2}_b(\R,\R)$. Then for any functions $y,\bar y:[0,T] \to \R$, we have for all $(s,t) \in \Delta$, 
\begin{equation}\label{e:G y - G bar y}
|\delta (G(y) - G(\bar y))_{s,t}| \le \|DG\|_{C^{\gamma - 2}_b} \Big( (|\delta y_{s,t}|^{\gamma - 2} + |\delta \bar y_{s,t}|^{\gamma - 2}) |(y - \bar y)_{t}| + |\delta (y - \bar y)_{s,t}| \Big).
\end{equation}
Moreover, for $1 \le q \le q'$, assume $\gamma - 1 \ge (q+q')/q'$. Then, it holds 
\begin{equation}\label{e:G y - G bar y'}
\begin{aligned}
&\|\delta (G(y) - G(\bar y))\|_{q'\text{-}\mathrm{var};[s,t]} \\
&\le \|DG\|_{C^{\gamma - 2}_b} \Big( \big( \|\delta y\|_{q\text{-}\mathrm{var};[s,t]}^{\gamma - 2} + \|\delta \bar y\|_{q\text{-}\mathrm{var};[s,t]}^{\gamma - 2}\big) \sup_{r\in[s,t]}|(y - \bar y)_r| + \|\delta (y - \bar y)\|_{q'\text{-}\mathrm{var};[s,t]} \Big)
\end{aligned}
\end{equation}
\end{lemma}

\begin{proof}
By the triangular inequality and Newton-Leibniz's formula, we have 
\begin{equation*}
\begin{aligned}
|\delta (G(y) - G(\bar y))_{s,t}| &\le \int_{0}^{1} |DG (\bar y_{t} + \lambda (y_t - \bar y_{t}) ) - DG (\bar y_{s} + \lambda (y_s - \bar y_{s}) )| d \lambda \cdot | (y - \bar y)_{t} | \\
& \quad + \int_{0}^{1} |DG(\bar y_s + \lambda(y_{s} - \bar y_{s}))| d \lambda \cdot |\delta(y - \bar y)_{s,t}|\\
&\le \|DG\|_{C^{\gamma - 2}_b} (|\delta y_{s,t}| + |\delta \bar y_{s,t}|)^{\gamma - 2} |(y - \bar y)_{t}| + \|DG\|_{\infty} |\delta (y - \bar y)_{s,t}|,
\end{aligned}
\end{equation*}
which yields \eqref{e:G y - G bar y}. Moreover, taking $q'$-variation to both sides of \eqref{e:G y - G bar y}, and noting $1/q' \le (\gamma - 2)/q$, it then follows \eqref{e:G y - G bar y'}.  
\end{proof}

\section{Symbolic Index}\label{app:symbolic_index}

\renewcommand{\arraystretch}{1.2}
\setlength{\tabcolsep}{6pt}
\begin{longtable}{@{}p{0.24\linewidth} p{0.56\linewidth} p{0.15\linewidth}@{}}
\textbf{Symbol} & \textbf{Meaning} & \textbf{Reference} \\
\hline
\endfirsthead
\textbf{Symbol} & \textbf{Meaning} & \textbf{Reference} \\
\hline
\endhead

$(\Omega,\mathcal F,\mathbb P)$ & underlying probability space & Sec.~\ref{sec:notations} \\
$(\mathcal F_t)_{t\in[0,T]}$ & augmented filtration generated by Brownian motion $W$ & Sec.~\ref{sec:notations} \\
$\|\xi\|_k$ & $L^k$-norm of a random variable $\xi$ & Sec.~\ref{sec:notations} \\
$\|S\|_{\mathcal L^{k}([s,t] \times \Omega)}$ & $L^k$-norm of a progressively measurable process $S$ & Sec.~\ref{sec:notations} \\
$\E_t[\bm \cdot]$ & conditional expectation $\E[\bm \cdot | \mathcal F_t]$ & Sec.~\ref{sec:notations} \\
$\|\xi|\mathcal G\|_m$, $\|\xi|\mathcal G\|_{m,n}$ & conditional $L^m$-norm and mixed $(m,n)$-norm & Eq.~\ref{e:mixed moment} \\
$\|F\|_{\infty;m,n}$ & supremum in time of the conditional mixed $(m,n)$-norm of a two parameter process $F$& Sec.~\ref{sec:notations} \\
$\Delta_{[s,t]}$, $\Delta$ & $\{(u,v):s\leq u\leq v\leq t\}$ and $\Delta_{[0,T]}$, respectively & Sec.~\ref{sec:notations} \\
$\mathcal P_{[s,t]}$, $|\pi|$ & partitions of $[s,t]$ and the mesh size of a partition $\pi$ & Sec.~\ref{sec:notations} \\
$w$ & control function on $\Delta$: $w(s,u)+w(u,t)\leq w(s,t)$ & Sec.~\ref{sec:notations} \\
$\mathscr L(V,K)$, $V\otimes K$ & space of linear maps from $V$ to $K$, and the algebraic tensor product & Sec.~\ref{sec:notations} \\
$\delta g_{u,v}$ & increment $g_v-g_u$ of a one-parameter path $g$ & Sec.~\ref{sec:notations} \\
$\|\bm \cdot\|_{\lambda\text{-}\mathrm{var};[s,t]}$,\ $(\|\bm\cdot\|_{\lambda\text{-}\mathrm{var}})_{s,t}$ & $\lambda$-variation norm and induced two-parameter function & Sec.~\ref{sec:notations}\\
$\|\bm\cdot\|_{\alpha\text{-}\mathrm{H\ddot{o}l};[s,t]}$ & H\"older seminorm & Sec.~\ref{sec:notations} \\
$\|\delta M\|_{\mathrm{BMO};[s,t]}$ & BMO norm of a continuous local martingale $M$ on $[s,t]$ & Sec.~\ref{sec:notations} \\
$\Theta$ & parameter set $(T,d,q,q',p)$ & Eq.~\ref{e:Theta} \\
$\lesssim$, $\lesssim_K$ & inequality up to a constant depending on $\Theta$, or on $(\Theta,K)$ & Sec.~\ref{sec:notations} \\
$C_b^\gamma(V,K)$ & functions with bounded derivatives up to order $\lfloor\gamma\rfloor$ and H\"older-continuous highest derivative & Sec.~\ref{sec:notations} \\
$\mathscr C^{p\text{-}\mathrm{var}}$ $(\mathscr C^{0,p\text{-}\mathrm{var}}_g)$ & space of (geometric) $p$-rough paths & Def.~\ref{def:rough path} \\
$\mathbf X=(\mathrm X,\mathbb X)$ & rough path and its first- and second-level components & Def.~\ref{def:rough path} \\
$|\delta \mathbf X|_{\lambda\text{-}\mathrm{var}}$ & $\lambda$-variation norm of rough path & Sec.~\ref{sec:notations}\\
$\triplenorm{\delta \mathbf X}_{\lambda\text{-}\mathrm{var}}$ & homogeneous $\lambda$-variation norm of rough path & Sec.~\ref{sec:notations} \\
$\mathscr C^{1/p}$ $(\mathscr C^{0,1/p}_g)$ & space of (geometric) $1/p$-H\"older rough paths & Def.~\ref{def:Holde rough path} \\
$\mathrm{RSM}_{\mathrm X}^{(p,p')\text{-}\mathrm{var}}$ & space of rough semimartingales controlled by $\mathrm X$ & Def.~\ref{def:rough semimartingale} \\
$P^Y$, $R^Y$, $M^Y$ &  controlled pathwise remainder part, controlled remainder part, and martingale part of $Y$ & Def.~\ref{def:rough semimartingale} \\
$\thicknm{(Y,Y')}_{\mathrm X;q,q';k,\infty}$ &  seminorm of rough semimartingales  & Eq.~\ref{e:Psi} 
\end{longtable}

\bibliographystyle{plain}

\bibliography{Reference-roughBSDE}

\end{document}